\documentclass[11pt]{article}
\usepackage[margin=1.3in]{geometry}
\usepackage{amssymb,amsfonts,amsmath,amsthm,amscd,dsfont,mathrsfs,bbold,pifont}
\usepackage{blkarray}
\usepackage{graphicx,float,psfrag,epsfig,color}
\usepackage{comment}
\usepackage{subcaption}

\usepackage{algorithm}
\usepackage[noend]{algpseudocode}

\makeatletter
\def\BState{\State\hskip-\ALG@thistlm}
\makeatother

	\usepackage{microtype}
	\usepackage[pdftex,pagebackref=true,colorlinks]{hyperref}
	\hypersetup{linkcolor=[rgb]{.7,0,0}}
	\hypersetup{citecolor=[rgb]{0,.7,0}}
	\hypersetup{urlcolor=[rgb]{.7,0,.7}}

\newcommand\independent{\protect\mathpalette{\protect\independent}{\perp}} 
\def\independent#1#2{\mathrel{\rlap{$#1#2$}\mkern2mu{#1#2}}}

\newcommand{\bin}{\mathrm{Bin}}

\newcommand{\F}{\mathbb{F}}

\newcommand{\mZ}{\mathbb{Z}}

\newcommand{\pp}{\mathbb{P}}
\newcommand{\E}{\mathbb{E}}
\DeclareMathOperator{\Var}{Var}
\newcommand{\e}{\varepsilon}

\DeclareMathOperator{\diag}{diag}

\newcommand{\1}{\mathbb{1}}
\newcommand{\snr}{\mathrm{SNR}}
\newcommand{\sbm}{\mathrm{SBM}}

\newtheorem{theorem}{Theorem}
\newtheorem{lemma}{Lemma}
\newtheorem{corollary}{Corollary}
\newtheorem{definition}{Definition}
\newtheorem{remark}{Remark}

\newtheorem{conjecture}{Conjecture}

\begin{document}

\title{Detection in the stochastic block model with multiple clusters: proof of the achievability conjectures, acyclic BP, and the information-computation gap\footnote{Information-theoretic part in the Proc.\ of ISIT 2016 \cite{colin3isit}, algorithmic part to appear in the Proc.\ of NIPS 2016 \cite{colin3nips}.}
}

\author{Emmanuel Abbe\thanks{Program in Applied and Computational Mathematics, and EE Department, Princeton University, USA, \texttt{eabbe@princeton.edu}. This research was partly supported by the NSF CAREER Award CCF-1552131, the ARO grant W911NF-16-1-0051, and the Google Faculty Research Award.} \and 
Colin Sandon\thanks{Department of Mathematics, Princeton University, USA,
\texttt{sandon@princeton.edu}.
}}

\date{}
\maketitle

\begin{abstract}
In a paper that initiated the modern study of the stochastic block model, Decelle et al., backed by Mossel et al., made the following conjecture: Denote by $k$ the number of balanced communities, $a/n$ the probability of connecting inside communities and $b/n$ across, and set $\mathrm{SNR}=(a-b)^2/(k(a+(k-1)b)$; for any $k \geq 2$, it is possible to detect communities efficiently whenever $\mathrm{SNR}>1$ (the KS threshold), whereas for $k\geq 4$, it is possible to detect communities information-theoretically for some $\mathrm{SNR}<1$. Massouli\'e, Mossel et al.\ and Bordenave et al.\ succeeded in proving that the KS threshold is efficiently achievable for $k=2$, while Mossel et al.\ proved that it cannot be crossed information-theoretically for $k=2$. The above conjecture remained open for $k \geq 3$. 

This paper proves this conjecture, further extending the efficient detection to non-symmetrical SBMs with a generalized notion of detection and KS threshold. For the efficient part, a linearized acyclic belief propagation (ABP) algorithm is developed and proved to detect communities for any $k$ down to the KS threshold in time $O(n \log n)$. Achieving this requires showing optimality of ABP in the presence of cycles, a challenge for message passing algorithms. The paper further connects ABP to a power iteration method with a nonbacktracking operator of generalized order, formalizing the interplay between message passing and spectral methods.  For the information-theoretic (IT) part, a non-efficient algorithm sampling a typical clustering is shown to break down the KS threshold at $k=4$. The emerging gap is shown to be large in some cases; if $a=0$, the KS threshold reads $b \gtrsim k^2$ whereas the IT bound reads $b \gtrsim k \ln(k)$, making the SBM a good study-case for information-computation gaps.
\end{abstract}

\thispagestyle{empty}
\newpage

\tableofcontents

\thispagestyle{empty}
\newpage

\pagenumbering{arabic}


\section{Introduction}
The stochastic block model (SBM) is a canonical model of networks with communities, and a natural model to study various central questions in machine learning, algorithms and statistics. The model serves in particular as a test bed for clustering and community detection algorithms, commonly used in social networks \cite{social1}, protein-to-protein interactions networks \cite{ppi2}, gene expressions \cite{gene-expr2}, recommendation systems \cite{amazon}, medical prognosis \cite{tumor}, DNA folding \cite{irineo}, image segmentation \cite{image1}, natural language processing \cite{ball} and more.

The SBM emerged independently in multiple scientific communities. The block model terminology, which seems to have dominated in the recent years, comes from the machine learning and statistics literature \cite{holland,sbm1,sbm3,sbm4,bickel,newman2,snij,rohe,choi}, while the model is typically called the planted partition model in theoretical computer science \cite{bui,dyer,boppana,jerrum,condon,carson,mcsherry}, and the inhomogeneous random graphs model in the mathematical literature \cite{bollo_inhomo}. 
Although the model was defined as far back as the 80s, it resurged in recent years due in part to the following fascinating conjecture established first in \cite{decelle}, and backed in \cite{Mossel_SBM1}, from deep but non-rigorous statistical physics arguments:

\begin{conjecture}\label{c1}
Let $(X,G)$ be drawn from SBM$(n,k,a,b)$, i.e., $X$ is uniformly drawn among partitions of $[n]$ into $k$ balanced clusters, and $G$ is a random graph on the vertex set $[n]$ where edges are placed independently with probability $a/n$ inside the clusters and $b/n$ across. Define $\snr=\frac{(a-b)^2}{k(a+(k-1)b)}$ and say that an algorithm detects communities if it takes as an input the graph $G$ and outputs a clustering $\hat{X}$ that is positively correlated with $X$ with high probability. Then, 
\begin{enumerate}
\item[(i)] Irrespective of $k$, if $\snr>1$, it is possible to detect communities in polynomial time,  i.e., the Kesten-Stigum (KS) threshold can be achieved efficiently;
\item[(ii)] If\footnote{The conjecture requires $k\ge 5$ when imposing the constraint that $a>b$, and $k\ge 4$ is enough in general.} $k \geq 4$, it is possible to detect communities information-theoretically for some $\snr$ strictly below 1.
\end{enumerate}
\end{conjecture}

We prove this conjecture in this paper. The problem was settled already for the case of $k=2$: It was proved in \cite{massoulie-STOC,Mossel_SBM2} that the KS threshold can be achieved efficiently for $k=2$, with an alternative proof later given in \cite{bordenave}, and \cite{Mossel_SBM1} shows that no information-computation gap takes places for $k=2$ with a tight converse. It was also shown in \cite{bordenave} that for SBMs with multiple communities satisfying a certain asymmetry condition (i.e., the requirement that $\mu_k$ is a simple eigenvalue in Theorem 5 of \cite{bordenave}), the KS threshold can be achieved. Yet, \cite{bordenave} does not resolve Conjecture \ref{c1} for $k \geq 3$.

An interesting challenge raised by part (i) of the conjecture is that standard clustering methods, commonly used in applications, fail to achieve the KS threshold. This includes spectral methods based on the adjacency matrix or standard Laplacians, as well as SDPs. For standard spectral methods, a first issue is that the fluctuations in the node degrees produce high-degree nodes that disrupt the eigenvectors from concentrating on the clusters.\footnote{This issue is further enhanced on real networks where degree variations are large.}  A classical trick is to trim such high-degree nodes \cite{coja-sbm,Vu-arxiv,sbm-groth,new-vu}, throwing away some information, but this does not suffice to achieve the KS threshold. SDPs are a natural alternative, but they also stumble\footnote{The recent results of \cite{ankur_SBM} on robustness to monotone adversaries suggest that SDPs can in fact not achieve the KS threshold.} before the KS threshold \cite{sbm-groth,montanari_sen}, focusing on the most likely rather than typical clusterings. As we shall show in this paper, and as already investigated in \cite{redemption,bordenave} for two communities, a linearized BP algorithm, or equivalently a spectral algorithm on a nonbacktracking operator, provides instead a solution to the conjecture. 

The nonbacktracking matrix $B$ of a graph was introduced by Hashimoto \cite{hashimoto} to study the Ihara zeta function, with the identity $\det (I - z B)=\frac{1}{\zeta(z)}$, where $\zeta$ is the Ihara zeta function of the graph. In particular, the poles of the Ihara zeta function are the reciprocal of the eigenvalues of $B$. Studying the spectrum of a graph thus implies properties on the location of the Ihara zeta function. The matrix is further used to define the graph Riemann hypothesis \cite{terras}, and studying its spectrum for random graphs such as the block model allows for generalizations of notions of Ramanujan graphs and Friedman's Theorem to non-regular cases, as discussed in \cite{bordenave}. The operator that we study is a natural extension of the classical nonbacktracking operator of Hashimoto, where we prohibit not only standard backtracks but also finite cycles. 

In their original paper \cite{decelle}, Decelle et al.\ conjecture that belief propagation (BP) achieves the KS threshold, and in fact, gives the the optimal accuracy in the reconstruction of the communities. However, the main issue when applying BP to the SBM is the classical one: the presence of cycles in the graph makes the behavior of the algorithm much more difficult to understand, and BP is susceptible to settling down in the wrong fixed points.\footnote{Empirical studies of BP on loopy graph show that convergence still takes place in some cases \cite{loopy}.} This is a long standing challenge in the realm of message passing algorithms for graphical models. Moreover, achieving the KS threshold requires precisely running BP to an extent where the graph is not even tree-like, thus precluding simple tricks. Numerical simulations suggest that starting with a purely random initialization, i.e., letting each vertex in the graph guess its community membership at random, and running BP works. However, no method is currently known to control random initialization, as discussed in \cite{Mossel_SBM2}. We develop here an alternate approach, using a linearized version of belief propagation. 

Further, the paper proves part (ii) of the conjecture, crossing the KS threshold at $k=4$ using a non-efficient algorithm that samples a typical clustering (i.e., a clustering having the right proportions of edges inside and across clusters). Note that the information-computation gap concerns the gap between the KS and information-theoretic thresholds, which is the gap between the computational and information-theoretic thresholds only under non-formal evidences \cite{decelle}. 
However, the IT bound that results from our analysis gives a gap to the KS threshold which is large in some cases (quasi-linear in $k$), making the SBM a good study-case for such gap phenomena. 


\subsection{Our results}\label{results}
The following are obtained:
\begin{enumerate}
\item A linearized acyclic belief propagation (ABP) algorithm is developed and shown to detect communities down to the KS threshold with complexity $O(n \log n)$, proving part (i) of Conjecture 1. A more general result applying to arbitrary (possibly asymmetrical) SBMs with a generalized notion of detection and KS threshold is also developed. The complexity of ABP is either comparable or improved compared to prior algorithms for $k=2$ \cite{massoulie-STOC,Mossel_SBM2,bordenave}, while ABP achieves universally the KS threshold (see Theorem \ref{main1});
\item An algorithm that samples a clustering with typical volumes and cuts is shown to detect communities below the KS threshold at $k=4$, proving part (ii) of Conjecture 1;
\item A connection between ABP and a power iteration method on a non-backtracking operator is developed, extending the operator of \cite{hashimoto} to higher order non backtracks, and formalizing the interplay described in \cite{redemption} between linearized BP and nonbacktracking operators;
\item An information-theoretic (IT) bound is derived for the symmetric SBM. For $a=0$, it is shown that detection is information-theoretically solvable if $b >c k \ln k + o_k(1)$, $c \in [1,2]$. Thus the information-computation gap --- defined as the gap between the KS threshold and the IT bound --- can be large since the KS threshold reads $b > k(k-1)$. 
Our bound interpolates the optimal threshold at $a=0$, and is conjectured to be tight in the scaling of $b$ for small $b$ and any $k$, and in the scaling of $k$ for large $k$.
\item An efficient algorithm is shown to learn the parameters $a,b,k$ in the symmetric SBM down to the KS threshold. 
\end{enumerate}

To achieve the KS threshold, we rely on a linearized version of BP that can handle cycles. The simplest linearized\footnote{Different forms of approximate message passing algorithms have been studied, such as in \cite{amp} for compressed sensing.} version of BP is to simply repeatedly update beliefs about a vertex's community based on its neighbor's suspected communities while ignoring the part of that belief that results from the beliefs about that vertex's community to prevent a feedback loop. However, this only works ideally if the graph is a tree. The correct response to a cycle would be to discount information reaching the vertex along either branch of the cycle to compensate for the redundancy of the two branches. However, due to computational issues we simply prevent information from cycling around small cycles in order to limit feedback. We also add steps where a multiple of the beliefs in the previous step are subtracted from the beliefs in the current step to prevent the beliefs from settling into an equilibrium where vertices' communities are sytematically misrepresented in ways that add credibility to each other. We refer to Section \ref{ABP} for a complete description of the algorithm and Section \ref{proof_tech1} for further intuition on how it performs. 

The fact that ABP is equivalent to a power iteration method on a non-backtracking operator results from its linearized form, as pointed out first  in \cite{redemption}.  This provides an intriguing synergy between message passing and spectral algorithms. 
It further allows us to interpret the obstructions of spectral methods through the lens of BP. The risk of obtaining eigenvectors that concentrate on singular structures (e.g., high degree nodes for Laplacian's), is related to the risk that BP settles down in wrong fixed points (e.g., due to cycling around high-degree nodes). Rather than removing such obstructions, ABP mitigates the feedback coming from the loops, giving rise to a non-backtracking operator, which extends the operator of \cite{hashimoto} by considering higher order nonbacktracks. In addition to simplifying the proofs, considering higher order nonbacktracking operators can help mitigating short loops in more general models, which is of independent interest. Further details are provided in Section \ref{proof_tech1}.

Note that solving the detection problem has direct implications on obtaining algorithms having optimal agreement (i.e., least fraction of mislabelled vertices) for the SBM. Once one has reasonable initial guesses of which communities the vertices are in, one can simply use full belief propagation to improve this to a better agreement. In order to do that from that output of ABP, one needs to first convert the division of vertices into two sets that are correlated with their communities to an assignment of each vertex to a nontrivial probability distribution for how likely it is to be in each community. Then, for each adjacent $v$ and $v'$, one can determine the probability distribution of what community $v$ is in based of the signs of $y'_{(v'',v)}$ for all $v''\ne v'$. Finally, use these as the starting probabilities for a belief propagation algorithm of depth $\ln(n)/3\ln(\lambda_1)$. See Section \ref{open} for further details on how this can be done. The main algorithmic challenge for obtaining optimal agreement in the SBM seems then captured by the detection problem.

To cross the KS threshold information theoretically, we rely on a non-efficient algorithm that samples a typical clustering. Upon observing a graph drawn from the SBM, the algorithm builds the set of all partitions of the $n$ nodes that have a typical fraction of edges inside and across clusters, and then samples a partition uniformly at random from that set. The analysis of the algorithm reveals three different regimes, that reflect three layers of refinement in the bounds on the typical set's size. In a first regime, no bad clustering (i.e., partition of the nodes that classifies close to $1/k$ of the vertices correctly) is typical with high probability based on a union-bound, and the algorithm samples only good clusterings with high probability. This allows us to cross the KS threshold for $k=5$ when $a=0$ but does not give the right bound at $b=0$. In a second regime, the large number of tree-like components in the graph is exploited, finding some bad clusterings to be typical but unlikely to be sampled. This gives a regime where the algorithm succeeds with the right bound at $b=0$, but not the right approximation at small $b$. To address the latter, a finer estimate on the typical set's size is obtained by also exploiting parts of the giant that are tree-like. Finally, we tighten our estimates on the typical set's size by taking into account vertices that are not saturated, i.e., whose neighbors do not cover all communities. The final bound crosses the KS threshold at $k=4$, interpolates the optimal threshold at $a=0$, and is conjectured to be tight in the scaling of $b$ for small $b$ and in the scaling of $k$ for large $k$ and small $a$. Further details are in Section \ref{proof_tech2}.

The learning of the parameters $a,b,k$ is done similarly as for the case $k=2$ \cite{Mossel_SBM1}. Note that learning the parameters when $k$ is unknown was previously settled only for diverging degrees \cite{colin2nips}, with related results in \cite{borgs_nips}.

\subsection{Related literature}

Several methods were proved to succeed down to the KS threshold for two communities. The first is based\footnote{Related ideas relying on shortest paths were also considered in \cite{bhatt-bickel}.} on a spectral method from the matrix of self-avoiding walks (entry $(i,j)$ counts the number of self-avoiding walks of moderate size between vertices $i$ and $j$) \cite{massoulie-STOC}, the second on counting weighted non-backtracking walks between vertices \cite{Mossel_SBM2}, and the third on a spectral method with the matrix of non-backtracking walks between directed edges (each edge is replaced with two directed edges and entry $(e,f)$ is one if and only if edge $e$ follows edge $f$) \cite{bordenave}. The first method has a complexity of $O(n^{1+\e})$, $\e>0$, while the second method affords a lesser complexity of $O(n \log^2 n)$ but with a large constant (see discussion in \cite{Mossel_SBM2}). These two methods were the first to achieve the KS threshold for two communities. The third method is based on a thorough analysis of the spectrum of the non-backtracking operator and allows going beyond the SBM with 2 communities, requiring however a certain asymmetry in the SBM parameters to obtain a result for detection (the precise condition is the requirement on $\mu_k$ being a simple eigenvalue of $M$ in Theorem 5 of \cite{bordenave}), thus falling short of proving Conjecture 1.(i) for $k \geq 3$ (since the second eigenvalue in this case has multiplicity at least 2). Note that a certain amount of symmetry is needed to make the detection problem interesting. For example, if the communities have different average degrees, detection becomes trivial. Thus the symmetric model SBM$(n,k,a,b)$ is in a sense the most challenging model for detection.

The non-backtracking operator was proposed first for the SBM in \cite{redemption}, also described as a linearization of BP. A precise spectral analysis of this operator is developed in \cite{bordenave} and applied to the SBM. This approach gives a fascinating approach to community detection, giving the first rigorous understanding on why nonbacktracking operators achieve the KS threshold. Besides the previously mentioned shortcomings for Conjecture 1 in the symmetric case, the operator suffers from an increase in dimension, as the derived matrix scales with the number of edges rather than vertices (specifically $2|E| \times 2|E|$, where $|E|$ is the number of edges).\footnote{The non-backtracking matrix is also not normal and has thus a complex spectrum; an interesting heuristic based on the Bethe Hessian operator was proposed in \cite{florent_bethe} to address the dimensionality and normality issues.} Nonbactracking spectral methods were also developed recently for the problem of detecting a single planted community \cite{}.

Our results are closest to \cite{Mossel_SBM2,bordenave}, while diverging in several key parts. 
A few technical expansions in the paper are similar to those carried in \cite{Mossel_SBM2}, such as the weighted sums over nonbacktracking walks and the SAW decomposition from \cite{Mossel_SBM2}, which are similar to our compensated nonbacktracking walk counts and standard decomposition. Our modifications are however developed to cope with general SBMs rather than the 2-symmetric case, in particular to compensate for the dominant eigenvalues in the latter setting, which is delicate due to the numerous potentially close eigenvalues. Our algorithm complexity is also slightly reduced by a logarithmic factor. 

Our algorithm is also closely related to \cite{bordenave}, which focuses on extracting the eigenvectors of the standard nonbacktracking operator. However, our proof technique is different than the one in \cite{bordenave}, so that we can cope with the setting of Conjecture 1. Also, we do not proceed with the eigenvectors extraction, but implement the algorithm in a belief propagation fashion. This avoids building the nonbacktracking matrix whose dimension grows with the number of edges. Note that from a spectral point of view, the power iteration method that we use is not relying on a traditional deflation method that subtracts the dominant eigenvector. Such an approach is likely to work in the symmetric SBM, but in the general SBM, we rely on a different approach that subtracts large eigenvalues times the identity matrix. Another difference from \cite{bordenave} is that we rely on nonbacktracking operators of higher orders $r$. While $r=2$ is arguably the simplest implementation and may suffice for the sole purpose of achieving the KS threshold, a larger $r$ may be beneficial in practice. For example, an adversary may add triangles for which ABP with $r=2$ would fail while larger $r$ would succeed. Finally, the approach of ABP can be extended beyond the linearized setting to improve the algorithm's accuracy.

For the information-theoretic part, a few papers have studied information-theoretic bounds and information-computation tradeoffs for SBMs with a growing number of communities \cite{chen-xu}, two unbalanced communities \cite{neeman-k}, and a single community \cite{am_1comm}. No results seemed known for the symmetric SBM and Conjecture 1(b). Shortly after this paper posting, \cite{banks} obtained bounds on the information theoretic threshold in an independent effort using moment methods. The bound in \cite{banks} crosses at $k=5$ rather than $k=4$ and does not interpolate to the giant component bound for $b=0$.

\subsection{Related models}
Exact recovery is a stronger recovery requirement than detection, which has long been studied for the SBM \cite{bui,dyer,boppana,snij,jerrum,condon,carson,mcsherry,bickel,rohe,choi,sbm-algos,Vu-arxiv,chen-xu,levina,abbs}, and more recently in the lens of sharp thresholds \cite{abh,mossel-consist,prout,new-xu,afonso_single,harrison,jog,prout2}. The notion of exact recovery requires a reconstruction of the complete communities with high probability. It was proved in \cite{abh,mossel-consist} that exact recovery has a sharp threshold for 
SBM$(n,2,a \log(n),b\log(n))$ at $|\sqrt{a}-\sqrt{b}|=1$, which can be achieved efficiently. As opposed to detection which can exploit variations in degrees, exact recovery becomes harder when considering general SBMs, where communities have different relative sizes and different connectivity parameters. In \cite{colin1}, it was proved that for the general SBM with linear size communities, exact recovery has a sharp threshold at the CH-divergence, and the threshold is proved to be efficiently achievable (without knowing the parameters in \cite{colin2}). This further improves on the result of \cite{Vu-arxiv} that apply to the logarithmic degree regime in full generality. Thus, for exact recovery with linear size communities, there is no information-computation gap. 
When considering sub-linear communities and coarser regime of the parameters, \cite{chen-xu} gives evidences that exact recovery can again have information-computation gaps. We also conjecture that similar phenomenon can take place in the setting of \cite{colin1} for exact recovery when $k$ is larger than $\log(n)$.

Finally, many variants of the SBM can be studied, such as the labelled block model \cite{airoldi,label_marc,jiaming}, the censored block model \cite{abbetoc,abbs,Chen_Goldsmith_ISIT2014,abbs-isit,rough,new-vu,florent_CBM}, the degree-corrected block model \cite{newman2}, overlapping block models \cite{fortunato} and more. While most of the fundamental challenges seem to be captured by the SBM already, these represent important extensions for applications.

\section{Results}




\begin{definition}
For positive integers $k,n$, a probability distribution $p\in (0,1)^k$, and a $k\times k$ symmetric matrix $Q$ with nonnegative entries, we define $\sbm(n,p,Q/n)$ as the probability distribution over ordered pairs $(\sigma,G)$ of an assignment of vertices to one of $k$ communities and an $n$-vertex graph generated by the following procedure. First, each vertex $v \in V(G)$ is independently assigned a community $\sigma_v$ under the probability distribution $p$. Then, for every $v\ne v'$, an edge is drawn in $G$ between $v$ and $v'$ with probability $Q_{\sigma_v,\sigma_{v'}}/n$, independently of other edges. We define $\Omega_i=\{v:\sigma_v=i\}$.
\end{definition}
We sometimes say that $G$ is drawn under $\sbm(n,p,Q/n)$ without specifying $\sigma$. The SBM is called symmetric if $p$ is uniform and if $Q$ takes the same value on the diagonal and the same value outside the diagonal. 
\begin{definition}
$(\sigma,G)$ is drawn under $\mathrm{SBM}(n,k,a,b)$, if $p_i=1/k$, $Q_{i,i}=a$ and $Q_{i,j}=b$ for every $i,j \in [k]$, $i \neq j$.
\end{definition}
Our goal is to find an algorithm that can distinguish between vertices from one community and vertices from another community in a non trivial way, as defined below. 

\begin{definition}
Let $A$ be an algorithm that takes a graph as input and outputs a partition of its vertices into two sets. $A$ solves detection\footnote{Detection is also called weak recovery.} (or detects communities) in graphs drawn from $\sbm(n,p,Q/n)$ if there exists $\epsilon>0$ such that the following holds. When $(\sigma,G)$ is drawn from $\sbm(n,p,Q/n)$ and $A$ divides its vertices into $S$ and $S^c$, with probability $1-o(1)$, there exist $i,j \in [k]$ such that $|\Omega_i\cap S|/|\Omega_i|-|\Omega_j\cap S|/|\Omega_j|>\epsilon$. Detection is solvable efficiently if the algorithm runs in polynomial time in $n$, and information-theoretically if no such complexity bound is obtained. 
\end{definition}

In other words, an algorithm solves detection if it divides the graph's vertices into two sets such that vertices from different communities have different probabilities of being assigned to one of the sets. An alternate definition by Decelle et al.\ \cite{decelle} says that an algorithm succeeds at detection if it divides the vertices into $k$ sets and there exists $\epsilon>0$ such that with high probability there exists an identification of the sets with the communities such that the algorithm classifies at least $\max p_i+\epsilon$ of the vertices correctly:

\begin{definition}
An algorithm $\hat{\sigma}:2^{[n] \choose 2} \to [k]^n$ solves max-detection in $\mathrm{SBM}(n,p,Q)$ if for some $\e>0$,  
\begin{align}
\pp\{A(\sigma,\hat{\sigma})\geq \max_{i \in [k]} p_i +\e\} = 1-o(1),
\end{align} 
where $(\sigma,G)\sim \mathrm{SBM}(n,p,Q)$ and for $x,y \in [k]^n$, $A(x,y) = \max_{\pi \in S_k} \frac{1}{n} \sum_{i=1}^n \1(x_i=\pi(y_i))$ denotes the agreement\footnote{Permutations $\pi$ need to be considered since only the partition needs to be detected and not the actual labels.} between $x$ and $y$. 
\end{definition}

In the $k$ community symmetric case, these definitions (detection and max-detection) are equivalent, but under asymmetry, this may 
not hold. 
Consider a two community asymmetric case where $p=(.2,.8)$. An algorithm that could find a set containing $2/3$ of the vertices from the large community and $1/3$ of the vertices from the small community would satisfy our definition; however, it would not satisfy Decelle's definition because each vertex is more likely to be in the large community than the small one no matter what the algorithm outputs. In general, our definition is satisfied by any algorithm that produces nontrivial amounts of evidence on what communities the vertices are in, while Decelle's definition requires the algorithm to sometimes produce enough evidence to overcome the prior probability. This may not always be possible (take for example the extreme case of an SBM with two community where each vertex is in community 1 with probability 0.99 and each pair of vertices in community 1 have an edge between them with probability $2/n$, while vertices in community 2 never have edges). If all communities have the same size then this distinction is meaningless, and we have:

\begin{lemma}
Let $k>0$, $Q$ be a $k\times k$ symmetric matrix with nonnegative entries, $p$ be the uniform distribution over $k$ sets, and $A$ be an algorithm that solves detection in graphs drawn from $\sbm(n,p,Q/n)$. Then $A$ also solves detection according to Decelle's criterion, provided that we consider it as returning $k-2$ empty sets in addition to its actual output.
\end{lemma}

\begin{proof}
Let $(\sigma,G)$ be drawn from $\sbm(n,p,Q/n)$ and $A(G)$ return $S$ and $S'$. There exists $\epsilon>0$ such that with high probability there exist $i$ and $j$ such that $|\Omega_i\cap S|/|\Omega_i|-|\Omega_j\cap S|/|\Omega_j|>\epsilon$. So, if we map $S$ to community $i$ and $S'$ to community $j$, the algorithm classifies at least 
\[|\Omega_i\cap S|/n+|\Omega_j\cap S'|/n=|\Omega_j|/n+|\Omega_i\cap S|/n-|\Omega_j\cap S|/n\ge 1/k+\epsilon/k-o(1)\]
 of the vertices correctly with high probability.
\end{proof}

\subsection{Achieving the KS threshold efficiently}
We present first a result that applies to the general SBM. We next specify the result for symmetric SBMs, and provide the ABP algorithm in the next section. Given parameters $p$ and $Q$ for the SBM, let $P$ be the diagonal matrix such that $P_{i,i}=p_i$ for each $i \in [k]$. Also, let $\lambda_1,...,\lambda_h$ be the distinct eigenvalues of $PQ$ in order of nonincreasing magnitude. Our results are in terms of the following notion of SNR:
\begin{definition}
The signal to noise ratio of $\sbm(n,p,Q/n)$ is defined by $$\snr=\lambda_2^2/\lambda_1.$$ 
\end{definition}
This paper shows that efficient detection is possible if $\snr >1$. In the $k$ community symmetric case $\sbm(n,k,a,b)$ where vertices are connected with probability $a/n$ inside communities and $b/n$ across, we have $\snr=(\frac{a-b}{k})^2/(\frac{a+(k-1)b}{k})=(a-b)^2/(k(a+(k-1)b))$, which is the quantity in Conjecture 1.

\begin{theorem}\label{main1}
Let $p\in (0,1)^k$ with $\sum p=1$, $Q$ be a symmetric matrix with nonnegative entries, $P$ be the diagonal matrix such that $P_{i,i}=p_i$, and $\lambda_1,...,\lambda_{h}$ be the distinct eigenvalues of $PQ$ in order of nonincreasing magnitude. If $\lambda_2^2>\lambda_1$ then there exist constants $r,c$ and $m=\Theta(\log(n))$ such that the acyclic belief propagation algorithm with these parameters solves detection in $SBM(n,p,Q/n)$. The algorithm can be run in $O(n\log n)$ time.
\end{theorem}
The proof is in Section \ref{proofs}. 

\begin{corollary}\label{main1b}
ABP solves detection in $\sbm(n,k,a,b)$ if
\begin{align}
\frac{(a-b)^2}{k(a+(k-1)b)} > 1
\end{align} 
and can be run in $O(n \log n)$ time.
\end{corollary}

\begin{remark}
Our definition of detection also extends to deciding whether or not an algorithm distinguishes between two specific communities, in the sense that there exists $\epsilon>0$ such that the fraction of the vertices from one of these communities assigned to $S$ differs from the fraction of the vertices from the other community assigned to $S$ by at least $\epsilon$ with high probability. The right version of our ABP algorithm can then distinguish between communities $i$ and $j$ if there exists an eigenvector $w$ of $PQ$ with an eigenvalue of magnitude greater than $\sqrt{\lambda_1}$ such that $w_i\ne w_j$. 
\end{remark}

\subsubsection{Acyclic Belief Propagation (ABP) Algorithm}\label{ABP}
We present here two versions of our main algorithm. A simplified version ABP$^*$ that applies to the symmetric SBM and that can easily be implemented, and the general version ABP that is used to prove Theorem \ref{main1}. The general version has additional steps that are used to prove the theorem, but these can be removed for practical applications. The intuitions behind the algorithms are discussed in Section \ref{simplified}, and in Section \ref{spectral}, we show that how the algorithms can be viewed as applying a power iteration method on a nonbacktracking operator $W^{(r)}$ of generalized order (where $r$ denotes the order of the nonbacktracks). The algorithms below have a message passing implementation and correspond to linearized version of belief propagation in which we originally guess what communities each vertex is likely to be in and then determine what communties each vertex's neighbors provide evidence for it to be in. Then we use that information to update our beliefs about what evidence each vertex provides about its neighbors' communities and repeat while mitigating cycles. \\

{\em $\mathrm{ABP}^*(G,m,r)$:}
\begin{enumerate}

\item For each adjacent $v$ and $v'$ in $G$, randomly draw $y^{(1)}_{v,v'}$ from a Gaussian distribution with mean $0$ and variance $1$. Also, consider $y^{(t)}_{v,v'}$ as having a value of $0$ whenever $t<1$.


\item 
For each $1<t\le m$, set 
\begin{align}
x^{(t-1)}_{v,v'}=y^{(t-1)}_{v,v'}-\frac{1}{2|E(G)|}\sum_{(v'',v''')\in E(G)} y^{(t-1)}_{v'',v'''} \label{bias}
\end{align}
for all adjacent $v$ and $v'$. For each adjacent $v,v'$ in $G$ that are not part of a cycle of length $r$ or less, set
\[y^{(t)}_{v,v'}=\sum_{v'':(v',v'')\in E(G),v''\ne v} z^{(t-1)}_{v',v''},\]
and for the other adjacent $v,v'$ in $G$, let the other vertex in the cycle that is adjacent to $v$ be $v'''$, the length of the cycle be $r'$, and set \[y_{v,v'}^{(t)}=\sum_{v'':(v',v'')\in E(G),v''\ne v} z_{v',v''}^{(t-1)}-\sum_{v'':(v,v'')\in E(G),v''\ne v',v''\ne v'''} z_{v,v''}^{(t-r')}\]
unless $t=r'$, in which case, set $y_{v,v'}^{(t)}=\sum_{v'':(v',v'')\in E(G),v''\ne v} z_{v',v''}^{(t-1)}-z^{(1)}_{v''',v}$.

\item Set $y'_v=\sum_{v':(v',v)\in E(G)} y^{(m)}_{v,v'}$ for every $v\in G$.
 Return $(\{v:y'_v>0\},\{v: y'_v\le 0\})$.
\end{enumerate}

\noindent
{\bf Remarks:}\\
\noindent
(1) In the $r=2$ case, one does not need to find cycles and one can exit step 2.\ after the second line. As mentioned above, we rely on a less compact version of the algorithm to prove the theorem, but expect that the above also succeeds at detection as long as $m>2\ln(n)/\ln(\snr)$. \\

\noindent
(2) What the algorithm does if $(v,v')$ is in multiple cycles of length $r$ or less is unspecified above, as there is no such edge with probability $1-o(1)$ in the sparse SBM. This can be modified for more general settings. The simplest such modification would be to apply this adjustment independently for each such cycle, and thus to set \[y_{v,v'}^{(t)}=\sum_{v'':(v',v'')\in E(G),v''\ne v} z_{v',v''}^{(t-1)}-\sum_{r'=1}^r\sum_{v''':(v,v''')\in E(G)} C^{(r')}_{v''',v,v'} \sum_{v'':(v,v'')\in E(G),v''\ne v',v''\ne v'''} z_{v,v''}^{(t-r')},\] where $C^{(r')}_{v''',v,v'}$ denotes the number of length $r'$ cycles that contain $v''',v,v'$ as consecutive vertices, substituting $z_{v''',v}^{(1)}$ for $\sum_{v'':(v,v'')\in E(G),v''\ne v',v''\ne v'''} z_{v,v''}^{(t-r')}$ when $r'=t$.
 This will not quite count r-nonbacktracking walks, but we believe that it would give a good enough approximation.\\

\noindent
(3) The purpose of setting $z^{(t-1)}_{v,v'}=y^{(t-1)}_{v,v'}-\frac{1}{2|E(G)|}\sum_{(v'',v''')\in E(G)} y^{(t-1)}_{v'',v'''}$ is to ensure that the average value of the $y^{(t)}$ is approximately $0$, and thus that the eventual division of the vertices into two sets is roughly even. There is an alternate way of doing this in which we simply let $z^{(t-1)}_{v,v'}=y^{(t-1)}_{v,v'}$ and then compensate for any bias of $y^{(t)}$ towards positive or negative values at the end. More specifically, we define $Y$ to be the $n\times m$ matrix such that for all $t$ and $v$, $Y_{v,t}=\sum_{v':(v',v)\in E(G)} y^{(t)}_{v,v'}$, and $M$ to be the $m\times m$ matrix such that $M_{i,i}=1$ and $M_{i,i+1}=-\lambda_1$ for all $i$, and all other entries of $M$ are equal to $0$. Then we set  $y'=YM^{m'}e_m$, where $e_m\in\mathbb{R}^m$ denotes the unit vector with 1 in the $m$-th entry, and $m'$ is a suitable integer.\\

\noindent
The full version of the algorithm applying to the general SBM is as follows.\\ 

\noindent
$ABP(G,m,r,c, (\lambda_1,...,\lambda_h)):$
\begin{enumerate}
\item Initialize:
\begin{enumerate}
\item Set $s=2$ unless $h>2$ and $|\lambda_2|=|\lambda_3|$, in which case set $s=3$.

\item Set $\gamma=(1-\lambda_1/\lambda^2_2)/2$.

\item Set $l=\max((s-1)/\ln((1-\gamma)|\lambda_s|)+s-1,2(2r+1)(s-1))$.

\item Assign each edge of $G$ independently with probability $\gamma$ to a set $\Gamma$. Then, remove these edges from $G$.


\item For every vertex $v\in G$, randomly draw $x_v$ from a Gaussian distribution with mean $0$ and variance $1$.

\item For each adjacent $v$ and $v'$, set $y_{v,v'}^{(1)}=x_{v'}$, and $y_{v,v'}^{(t)}=0$ for all $t<1$.

\end{enumerate}

\item Propagate:

\begin{enumerate}

\item For each $1\le t\le m$, and each adjacent $(v,v')\in E(G)$, set 
\[y_{v,v'}^{(t)}=\sum_{v'':(v',v'')\in E(G),v''\ne v} y_{v',v''}^{(t-1)}\] unless $(v,v')$ is part of a cycle of length $r$ or less. If it is, then let the other vertex in the cycle that is adjacent to $v$ be $v'''$, and the length of the cycle be $r'$\footnote{What the algorithm does if $(v,v')$ is in multiple cycles of length $r$ or less is unspecified as there is no such edge with probability $1-o(1)$ in the SBM. One can adapt this for more general models.}. Set \[y_{v,v'}^{(t)}=\sum_{v'':(v',v'')\in E(G),v''\ne v} y_{v',v''}^{(t-1)}-\sum_{v'':(v,v'')\in E(G),v''\ne v',v''\ne v'''} y_{v,v''}^{(t-r')}\]
unless $t=r'$. In that case, set 
\[y_{v,v'}^{(t)}=\sum_{v'':(v',v'')\in E(G),v''\ne v} y_{v',v''}^{(t-1)}-x_v.\]

\item Set $Y$ to be the $n\times m$ matrix such that for all $t$ and $v$, \[Y_{v,t}=\sum_{v':(v',v)\in E(G)} y^{(t)}_{v,v'}\]

\item For each $s'<s$, set $M_{s'}$ to be the $m\times m$ matrix such that $M_{i,i}=1$ and $M_{i,i+1}=-(1-\gamma)\lambda_{s'}$ for all $i$, and all other entries of $M_{s'}$ are equal to $0$. Also, let $e_m\in\mathbb{R}^m$ be the vector with an $m$-th entry of $1$ and all other entries equal to $0$. Set $$y^{(m)}=Y\left(\prod_{s'<s} M_{s'}^{\lceil \frac{m-r-(2r+1)s'}{l}\rceil}\right)e_m.$$

\item For each $v$, set 
\[y'_v=\sum_{v':(v,v')\in \Gamma} y_{v'}^{(m)}\]
and set $y''_v$ to the sum of $y'_{v'}$ over all $v'$ that have shortest paths to $v$ of length $\lfloor \sqrt{\log\log n}\rfloor$.

\end{enumerate}

\item Assign:

\begin{enumerate}

\item Set $c'=c\cdot \sqrt{\sum_{v\in G} (y''_v)^2/n}$. Create sets of vertices $S_1$ and $S_2$ as follows. For each vertex $v$, if $y''_v<-c'$, assign $v$ to $S_1$. If $y''_v>c'$, then assign $v$ to $S_2$. Otherwise, assign $v$ to $S_2$ with probability $1/2+y''_v/2c'$ and $S_1$ otherwise.

\item Return $(S_1,S_2)$.

\end{enumerate}
\end{enumerate}

\subsection{Crossing the KS threshold information-theoretically}
The following gives a region of the parameters in the symmetric SBM where detection can be solved information-theoretically. 
\begin{theorem}\label{main2}
Let $d:=\frac{a+(k-1)b}{k}$, assume $d>1$, and let $\tau=\tau_d$ be the unique solution in $(0,1)$ of $\tau e^{-\tau}=de^{-d}$, i.e., $\tau = \sum_{j=1}^{+ \infty} \frac{j^{j-1}}{j!} (de^{-d})^j$. The Typicality Sampling Algorithm detects\footnote{Setting $\delta>0$ small enough gives the existence of $\e>0$ for detection.} communities in $\sbm(n,k,a,b)$ if
\begin{align}
&  \frac{a \ln a + (k-1) b \ln b}{k}  - \frac{a+(k-1)b}{k} \ln \frac{a+(k-1)b}{k} \\
&\quad  >  \min \left(\frac{1-\tau}{1-\tau k/(a+(k-1)b)} 2 \ln (k), 2\ln (k)- 2\ln(2)e^{-a/k}(1-(1-e^{-b/k})^{k-1}) \right). \label{bound1}
\end{align}
\end{theorem}
\noindent
This bound strictly improves on the KS threshold for $k\ge 4$: 
\begin{corollary}
Conjecture 1 part (ii) holds. 
\end{corollary}


\begin{remark}
Note that \eqref{bound1} simplifies to
\begin{align}
&\frac{1}{2 \ln k} \left( \frac{a \ln a + (k-1) b \ln b}{k}  - d\ln d \right) > \frac{1-\tau}{1-\tau /d} =:f(\tau,d),
\end{align}
and since $f(\tau,d)<1$ when $d>1$ (which is needed for the presence of the giant), detection is already solvable in $\sbm(n,k,a,b)$ if 
\begin{align}
\frac{1}{2 \ln k} \left( \frac{a \ln a + (k-1) b \ln b}{k}  - d \ln d \right) >1.
\end{align}
As we shall see in Lemma \ref{bad_bound}, the above corresponds to the regime where there is no bad clustering that is typical with high probability. However, the above bound is not tight in the extreme regime of $b=0$, since it reads $a>2k$ as opposed to $a>k$, and only crosses the KS threshold at $k=5$. 
\end{remark}

Defining $a_k(b)$ as the unique solution of 
\[\frac{1}{2 \ln k} \left( \frac{a \ln a + (k-1) b \ln b}{k}  - d\ln d \right) = \min\left(f(\tau,d),1-\frac{e^{-a/k}(1-(1-e^{-b/k})^{k-1})\ln(2)}{\ln(k)}\right)\]
and simplifying the bound in Theorem \ref{main2} gives the following. 
\begin{corollary}\label{corol_ext}
Detection is solvable  
\begin{align}
&\text{in}\quad \sbm(n,k,0,b) \quad \text{if} \quad b >   \frac{2k \ln k}{(k-1) \ln \frac{k}{k-1}} f(\tau,b(k-1)/k), \label{a0} \\ 
&\text{in}\quad \sbm(n,k,a,b) \quad \text{if} \quad a > a_k(b), \quad \text{where } a_k(0)=k. \label{b0}
\end{align}
\end{corollary}
\begin{remark}
Note that \eqref{b0} approaches the optimal bound given by the presence of the giant at $b=0$, and we further conjecture that $a_k(b)$ gives the correct first order approximation of the information-theoretic bound for small $b$. 
\end{remark}
\begin{remark}
Note that the $k \ln k$ scaling in \eqref{a0} improves significantly on the KS threshold given by $b>k (k-1)$ at $a=0$. Relatedly, note that the $k$-colorability threshold for Erd\H{o}s-R\'enyi graphs grows as $2k\ln k$ \cite{naor_col}. This coule be used to obtain an information-theoretic bound, but the constant would be looser than the one obtained here. 
\end{remark}
\begin{remark}
We also believe that the above gives the correct scaling in $k$ for $a=0$, i.e., that for $b< (1-\e)k \ln (k) + o_k(1)$, $\e>0$, detection is information-theoretically impossible. To see this, consider $v\in G$, $b=(1-\epsilon)k \ln(k)$, and assume that we know the communities of all vertices more than $r=\ln(\ln(n))$ edges away from $v$. For each vertex $r$ edges away from $v$, there will be approximately $k^{\epsilon}$ communities that it has no neighbors in. Then vertices $r-1$ edges away from $v$ have approximately $k^{\epsilon}\ln(k)$ neighbors that are potentially in each community, with approximately $\ln(k)$ fewer neighbors suspected of being in its community than in the average other community.  At that point, the noise has mostly drowned out the signal and our confidence that we know anything about the vertices' communities continues to degrade with each successive step towards $v$.  
\end{remark}
In \cite{banks}, a different approach than the one described in previous remark is developed based on a contiguity argument and estimates from \cite{naor_col}, and formally proves that the scaling in $k$ is in fact tight. 

\subsubsection{Typicality Sampling Algorithm}
Define 
\begin{align}
\mathrm{Bal}(n,k,\e)&=\{x \in [k]^n:  \forall i \in [k], |\{ u \in [n]: x_u = i\}|/n \in [ 1/k -\e, 1/k + \e]  \}
\end{align}
and $\mathrm{Bal}(n,k)=\mathrm{Bal}(n,k,\log n/\sqrt{n})$, the set of vectors in $[k]^n$ with an asymptotically uniform fraction of components in each community.

Given an $n$-vertex graph $G$ and $\delta>0$, the algorithm draws $\hat{\sigma}_{\mathrm{typ}}(G)$ uniformly at random in  
\begin{align*}
T_\delta(G)=&\{   x \in \mathrm{Bal}(n,k,\delta):\\
&  \sum_{i=1}^k |\{  G_{u,v} : (u,v) \in {[n] \choose 2} \text{ s.t. } x_u=i, x_v=i \}|  \geq \frac{an}{2k} (1-\delta),\\
& \sum_{i,j \in [k], i<j} |\{  G_{u,v} : (u,v) \in {[n] \choose 2} \text{ s.t. } x_u=i, x_v=j \}|  \leq \frac{bn(k-1)}{2k} (1+\delta) \},
\end{align*}
where the above assumes that $a>b$; flip the above two inequalities in the case $a<b$.

\subsection{Learning the model}
To learn the parameters, we count cycles of slowly growing length as already done in \cite{Mossel_SBM1} for $k=2$, using non-backtracking walks to approximate the count. 

\begin{lemma}\label{learn}
If $\snr >1$, there exists a consistent and efficient estimator for the parameters $a,b,k$ in $\sbm(n,k,a,b)$.
\end{lemma}





\section{Achieving the KS threshold: proof technique}\label{proof_tech1}
Recall the parameters:  
$k$ and $n$ are positive integers, $p\in (0,1)^k$ with $\sum p_i=1$, and $Q$ is a $k\times k$ symmetric matrix with nonnegative entries. Then $SBM(n,p,Q/n)$ generates $n$-vertex graphs by the following procedure. First, each vertex $v$ is randomly and independently assigned a community $\sigma_v$ such that the probability that $\sigma_v=i$ is $p_i$ for each $i$. Then, each pair of vertices $v$ and $v'$ have an edge put between them with probability $Q_{\sigma_v,\sigma_{v'}}/n$.
Also, let $\Omega_1,...,\Omega_k$  be the communities and $P$ be the $k\times k$ diagonal matrix such that $P_{i,i}=p_i$ for each $i$. Now, consider the $2$-community symmetric stochastic block model. In this case, $k=2$, $p=[1/2,1/2]$, $Q_{i,j}$ is $a$ if $i=j$ and $b$ otherwise for some $a,b$. Now, let $\lambda_1=\frac{a+b}{2}$ be the average degree of a vertex in a graph drawn from this model, and $\lambda_2=\frac{a-b}{2}$ be the other eigenvalue of $PQ$. Throught this section we say $f$ is approximately $g$ or $f\approx g$ when $|f-g|=o(|f+g|)$ with probability $1-o(1)$.

Our goal is to determine which of $v$'s vertices are in each community with an accuracy that is nontrivially better than that attained by random guessing. Obviously, the symmetry between communities ensures that we can never tell whether a given vertex is in community $1$ or community $2$, so the best we can hope for is to divide the vertices into two sets such that there is a nontrivial difference between the fraction of vertices from community $1$ that are assigned to the first set and the fraction of vertices from community $2$ that are assigned to the first set.
\subsection{Amplifying luck}

Let $x$ and $n-x$ be the numbers of vertices in community $1$ and community $2$ respectively. If the vertices are assigned sets at random, then the expected numbers of vertices from each community in the first set are $\frac{x}{2}$ and $\frac{n-x}{2}$. However, by the central limit theorem, the probability distribution of the actual number of vertices from a given community in the first set is approximately a gaussian distribution with the mean stated previously and a variance of $\frac{x}{4}$ or $\frac{n-x}{4}$ as appropriate. That means that the probability distribution of the difference between the fraction of the vertices from community $1$ assigned to the first set and the fraction of the vertices from the community $2$ assigned to the first set is also approximately a bell curve. This one has a mean of $\frac{1}{2}-\frac{1}{2}=0$ and a variance of
\[\frac{x}{4}/x^2+\frac{n-x}{4}/(n-x)^2=\frac{1}{4x}+\frac{1}{4(n-x)}\approx \frac{1}{2n}+\frac{1}{2n}=\frac{1}{n}\]
That means that it has a standard deviation of $\approx 1/\sqrt{n}$, and the difference between the fraction of the vertices from community $1$  assigned to the first set and the fraction of the vertices from community $2$ assigned to the first set will typically have a magnitude on the order of $1/\sqrt{n}$. Label the sets $S_1$ and $S_2$ such that the fraction of the vertices from $\Omega_1$ that were assigned to $S_1$ is at least as large as the fraction of the vertices from $\Omega_2$ that were assigned to $S_1$. For the rest of this section, we will consider the difference between these fractions to be fixed.

Consider determining the community of $v$ using belief propagation, assuming some preliminary guesses about the vertices $t$ edges away from it, and assuming that the subgraph of $G$ induced by the vertices within $t$ edges of $v$ is a tree. For any vertex $v'$ such that $d(v,v')<t$, let $C_{v'}$ be the set of the children of $v'$. If we believe based on either our prior knowledge or propagation of beliefs up to these vertices that $v''$ is in community $1$ with probability $\frac{1}{2}+\frac{1}{2}\epsilon_{v''}$ for each $v''\in C_{v'}$, then the algorithm will conclude that $v'$ is in community $1$ with a probability of 
\[\frac{\prod_{v''\in C_{v'}} (\frac{a+b}{2}+\frac{a-b}{2}\epsilon_{v''})}{\prod_{v''\in C_{v'}} (\frac{a+b}{2}+\frac{a-b}{2}\epsilon_{v''})+\prod_{v''\in C_{v'}} (\frac{a+b}{2}-\frac{a-b}{2}\epsilon_{v''})}.\]
If all of the $\epsilon_{v''}$ are close to $0$, then this is approximately equal to
\[\frac{1+\sum_{v''\in C_{v'}}\frac{a-b}{a+b}\epsilon_{v''}}{2+\sum_{v''\in C_{v'}} \frac{a-b}{a+b}\epsilon_{v''}+\sum_{v''\in C_{v'}} -\frac{a-b}{a+b}\epsilon_{v''}}=\frac{1}{2}+\frac{a-b}{a+b}\sum_{v''\in C_{v'}}\frac{1}{2}\epsilon_{v''}.\]

That means that the belief propagation algorithm will ultimately conclude that $v$ is in community $1$ with a probability of approximately $\frac{1}{2}+\frac{1}{2}(\frac{a-b}{a+b})^t\sum_{v'':d(v,v'')=t} \epsilon_{v''}$. If there exists $\epsilon$ such that $E_{v''\in \Omega_1}[\epsilon_{v''}]=\epsilon$ and $E_{v''\in \Omega_2}[\epsilon_{v''}]=-\epsilon$ (recall that $\Omega_i=\{v:\sigma_v=i\}$), then on average we would expect to assign a probability of approximately $\frac{1}{2}+\frac{1}{2}\left(\frac{(a-b)^2}{2(a+b)}\right)^t\epsilon$ to $v$ being in its actual community, which is enhanced as $t$ increases when $\snr>1$.

Equivalently, given a vertex $v$ and a small $t$, the expected number of vertices that are $t$ edges away from $v$ is approximately $(\frac{a+b}{2})^t$, and the expected number of these vertices in the same community as $v$ is approximately $(\frac{a-b}{2})^t$ greater than the expected number of these vertices in the other community. So, if we had some way to independently determine which community a vertex is in with an accuracy of $\frac{1}{2}+\epsilon$ for small $\epsilon$, we could guess that each vertex is in the community that we think that the majority of the vertices $t$ steps away from it are in to determine its community with an accuracy of roughly $\frac{1}{2}+\left(\frac{(a-b)^2}{2(a+b)}\right)^{t/2}\epsilon$.

To obtain initial estimates, we simply guess the vertices' communities at random as described in the previous section, with the expectation that the fractions of the vertices from the two communities assigned to a community will differ by $\theta(1/\sqrt{n})$ by the Central Limit Theorem. 
Once we have even such weak information on which vertex is in which community, we can try to improve our classification of a given vertex by factoring in our knowledge of what communities the nearby vertices are in. For a vertex $v$ and integer $t$, let $N_t(v)$ be the number of vertices $t$ edges away from $v$, $\Delta_t(v)$ be the difference between the number of vertices $t$ edges away from $v$ that are in community $1$ and the number of vertices $t$ edges away from $v$ that are in community $2$, and $\widetilde{\Delta}_t(v)$ be the difference between the number of vertices $t$ edges away from $v$ that are in $S_1$ and the number of vertices $t$ edges away from $v$ that are in $S_2$. For small $t$,

\[E[N_t(v)]\approx \left(\frac{a+b}{2}\right)^t\]
and 
\[E[\Delta_t(v)]\approx \left(\frac{a-b}{2}\right)^t\cdot (-1)^{\sigma_v}\]

For any fixed values of $N_t(v)$ and $\Delta_t(v)$, the probability distribution of $\widetilde{\Delta}_t(v)$ is essentially a Gaussian distribution with a mean of $\Theta(\Delta_t(v)/\sqrt{n})$ and a variance of $\approx N_t(v)$ because it is the sum of $N_t(v)$ nearly independent variables that are approximately equally likely to be $1$ or $-1$. So, $\widetilde{\Delta}_t(v)$ is positive with a probability of $\frac{1}{2}+\Theta(\Delta_t(v)/\sqrt{|N_t(v)|n})$. In other words, if $v$ is in community $1$ then $\widetilde{\Delta}_t(v)$ is positive with a probability of 
\[\frac{1}{2}-\Theta\left(\left( \frac{a-b}{2}\right)^t\cdot\left(\frac{a+b}{2}\right)^{-t/2}/\sqrt{n}\right)\] 
and if $v$ is in community $2$ then $\widetilde{\Delta}_t(v)$ is positive with a probability of 
\[\frac{1}{2}+\Theta\left(\left( \frac{a-b}{2}\right)^t\cdot\left(\frac{a+b}{2}\right)^{-t/2}/\sqrt{n}\right)\] 
If $(a-b)^2\le 2(a+b)$, then this is not improving the accuracy of the classification, so this technique is useless. On the other hand, if $(a-b)^2>2(a+b)$, the classification becomes more accurate as $t$ increases. However, this formula says that to classify vertices with an accuracy of $1/2+\Omega(1)$, we would need to have $t$ such that
\[\left( \frac{a-b}{2}\right)^{2t}=\Omega\left( \left(\frac{a+b}{2}\right)^{t}n\right)\]
However, unless\footnote{If $a=0$ and $k=2$ or $b=0$, then classifying vertices based on the sign of $\widetilde{\Delta}_t(v)$ for suitable $t$ is likely to work, but this is pointlessly complicated because every component of the graph consists of all vertices of one community or vertices of alternating communities.} $a$ or $b$ is $0$, that would imply that
\[\left( \frac{a+b}{2}\right)^{2t}=\omega\left(\left( \frac{a-b}{2}\right)^{2t}\right)=\omega\left( \left(\frac{a+b}{2}\right)^{t}n\right)\]
which means that $(\frac{a+b}{2})^{t}=\omega(n)$. It is obviously impossible for $N_t(v)$ to be greater than $n$, so this $t$ is too large for the approximation to hold. The problem is, the approximation assumes that each vertex at a distance of $t-1$ from $v$ has one edge leading back towards $v$, and that the rest of its edges lead towards new vertices. Once a significant fraction of the vertices are less than $t$ edges away from $v$, a significant fraction of the edges incident to vertices $t-1$ edges away from $v$ are part of loops and thus do not lead to new vertices.

\begin{figure}[H]
\centering
\begin{subfigure}{.5\textwidth}
  \centering
  \includegraphics[width=0.9\linewidth]{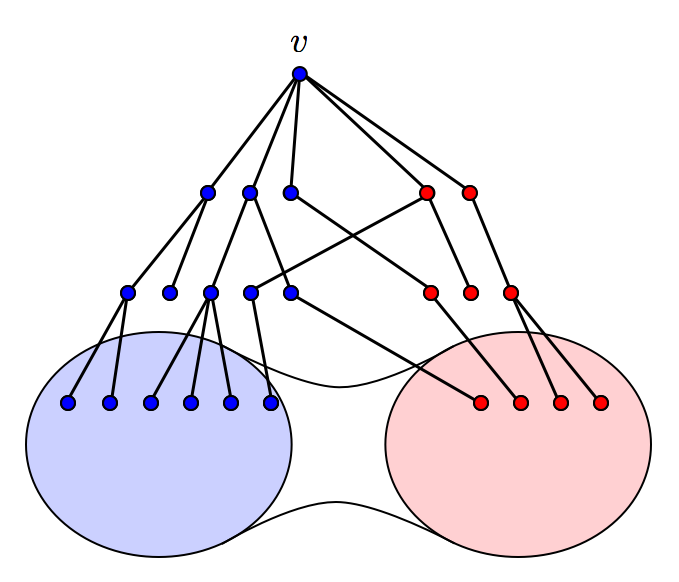}
  \label{depth2}
\end{subfigure}%
\begin{subfigure}{.5\textwidth}
  \centering
  \includegraphics[width=0.9\linewidth]{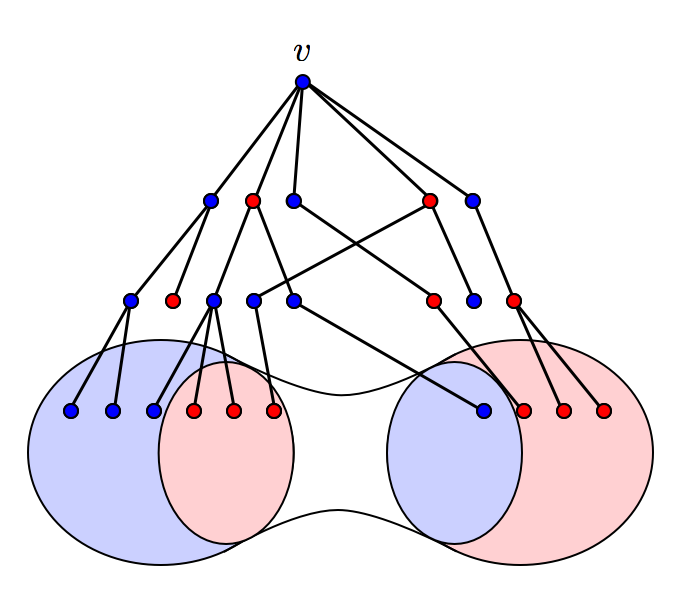}
  \label{depth2r}
\end{subfigure}
\caption{The left figure shows the neighborhood of vertex $v$ pulled from the SBM graph at depth $c\log_{\lambda_1} n$, $c<1/2$, which is a tree with high probability. If one had an educated guess about each vertex's label, of good enough accuracy, then it would be possible to amplify that guess by considering only such small neighborhoods (deciding with the majority at the leaves). However, we do not have such an educated guess. We thus initialize our labels purely at random, obtaining a small advantage of roughly $\sqrt{n}$ vertices by luck (i.e., the central limit theorem), in either an agreement or disagreement form. This is illustrated in agreement form in the right figure. We next attempt to amplify that lucky guess by exploiting the information of the SBM graph. Unfortunately, the graph is too sparse to let us amplify that guess by considering tree like or even loopy neighborhoods; the vertices would have to be exhausted. This takes us to considering walks.}       
\label{tree}
\end{figure}

\subsection{Nonbacktracking walks}

An obvious way to solve the problem caused by running out of vertices would be to simply count the walks of length $t$ from $v$ to vertices in $S_1$ or $S_2$. Recall that a {\em walk} is a series of vertices such that each vertex in the walk is adjacent to the next, and a {\em path} is a walk with no repeated vertices. The last vertex of such a walk will be adjacent to an average of approximately $a/2$ vertices in its community outside the walk and $b/2$ vertices in the other community outside the walk. However, it will also be adjacent to the second to last vertex of the walk, and maybe some of the other vertices in the walk as well. As a result, the number of walks of length $t$ from $v$ to vertices in $S_1$ or $S_2$ cannot be easily predicted in terms of $v$'s community. So, the numbers of such walks are not useful for classifying vertices.

We could deal with this issue by counting paths\footnote{This type of approach is considered in \cite{bhatt-bickel}.} of length $t$ from $v$ to vertices in $S_1$ and $S_2$. Given a path of length $t-1$, the expected number of vertices outside the path in the same community as its last vertex that are adjacent to it is approximately $a/2$ and the expected number of vertices outside the path in the opposite community as its last vertex that are adjacent to it is approximately $b/2$. So, the expected number of paths of length $t$ from $v$ is approximately $(\frac{a+b}{2})^t$ and the expected difference between the number that end in vertices in the same community as $v$ and the number that end in the other community is approximately $(\frac{a-b}{2})^t$. The problem with this is that counting all of these paths is inefficient.

The compromise we use is to count nonbacktracking walks ending at $v$, i.e. walks that never repeat the same edge twice in a row. We can efficiently determine how many nonbacktracking walks of length $t$ there are from vertices in $S_i$ to $v$ by using the fact that the number of nonbacktracking walks of length $t$ starting at a vertex in $S_i$ and having $v'$ and $v$ as their last two vertices is equal to the sum over all $v''\ne v$ such that $v''$ is adjacent to $v'$ of the number of nonbacktracking walks of length $t-1$ starting at a vertex in $S_i$ and having $v''$ and $v'$ as their last two vertices. Furthermore, most nonbacktracking walks of a given length that is logarithmic in $n$ are paths, so it seems reasonable to expect that counting nonbacktracking walks instead of paths in our algorithm will have a negligible effect on the accuracy.

\begin{figure}[H]
\centering
  \includegraphics[width=0.5\linewidth]{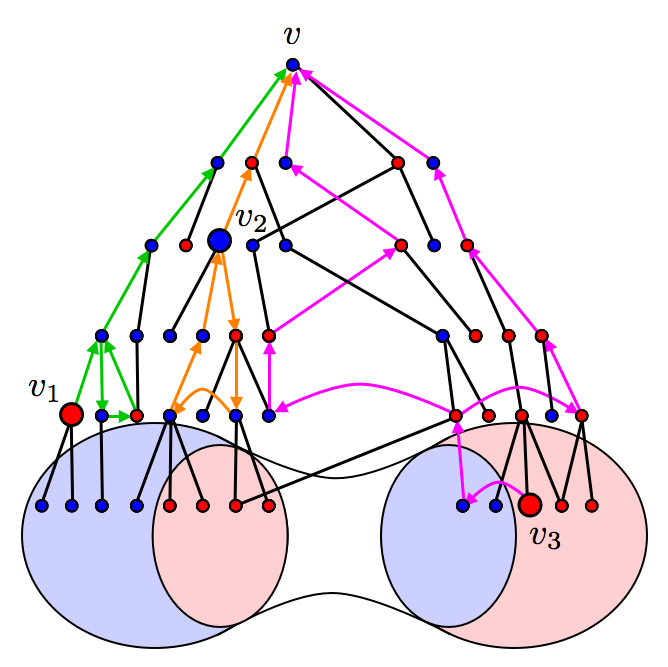}
\caption{This figure extends Figure \ref{tree} to a larger neighborhood. The ABP algorithm amplifies the belief of vertex $v$ by considering all the walks of a given length that end at it. To avoid being disrupted by backtracking or cycling the beliefs on short loops, the algorithm considers only walks that do not repeat the same vertex within $r$ steps, i.e., $r$-nonbacktracking walks. For example, when $r=3$ and when the walks have length $7$, the green walk starting at vertex $v_1$ is discarded, whereas the orange walk starting at the vertex $v_2$ is counted. Note also that the same vertex can lead to multiple walks, as illustrated with the two magenta walks from $v_3$. Since there are approximately equally many such walks between any two vertices, if the majority of the vertices were initially classified as blue, this is likely to classify all of the vertices as blue. We hence need a compensation step to prevent the classification from becoming biased towards one community.}
\label{loops}
\end{figure}

More precisely, that suggests the following approach. Define $y_{v,v'}^{(t)}$ to be the number of nonbacktracking walks of length $t$ that start at vertices in $S_2$ and end in the directed edge $(v',v)$ minus the number of nonbacktracking walks of length $t$ that start at vertices in $S_1$ and end in $(v',v)$. Also, define $y_{v}^{(t)}$ to be the overall difference between the number of nonbacktracking walks of length $t$ from vertices in $S_2$ to $v$ and the number of nonbacktracking walks of length $t$ from vertices in $S_1$ to $v$. Their values can be efficiently computed by means of the following procedure:
\begin{enumerate}
\item For every $(v,v')\in E(G):$

\hspace{1 cm} If $v'\in S_2$, set $y_{v,v'}^{(1)}=1$

\hspace{1 cm} Otherwise, set $y_{v,v'}^{(1)}=-1$

\item For every $1<t\le m$ and $(v,v')\in E(G):$

\hspace{1 cm} Set $y_{v,v'}^{(t)}=\sum_{v'': (v'',v')\in E(G),v''\ne v} y_{v',v''}^{(t-1)}$

\item For every $(v,v')\in E(G)$ and $v\in G$

\hspace{1 cm} Set $y_v^{(m)}=\sum_{v':(v,v')\in E(G)} y_{v,v'}^{(m)}$
\end{enumerate}
One way of viewing this algorithm is that $y_{v,v'}^{(t)}$ represents our current belief about what community $v'$ is in, disregarding any information derived from the fact that it is next to $v$. We start with fairly unconfident beliefs about the vertices' communities, and then derive more and more confident beliefs about the vertices' communities by taking our beliefs about their neighbors' communities into account.

\subsection{Compensation for the average value}

$y_{v,v'}^{(1)}$ has an average value of $\Theta(1/\sqrt{n})$ for $v'$ in community $2$ and $-\Theta(1/\sqrt{n})$ for $v'$ in community $1$. Also, $y_{v,v'}^{(1)}$ has a variance of order $1$. For a random $(v,v')\in E(G)$, $v'$ will have an average of approximately $a/2$ neighbors other than $v$ in its community and $b/2$ neighbors other than $v$ in the other community. So, by induction on $t$, we would expect that $y_{v,v'}^{(t)}$ would have an average value of $\Theta((\frac{a-b}{2})^t/\sqrt{n})$ for v' in community $2$ and $-\Theta((\frac{a-b}{2})^t/\sqrt{n})$ for v' in community $1$. Since $y_{v',v''}^{(t-1)}$ should be approximately independent for different $v''$ adjacent to $v'$, we would also expect that $y_{v,v'}^{(t)}$ would have an empirical variance of approximately $(\frac{a+b}{2})^t$, and thus a standard deviation of approximately $\sqrt{(\frac{a+b}{2})^t}$. So, for $t$ such that $(\frac{a-b}{2})^t/\sqrt{n}>\sqrt{(\frac{a+b}{2})^t}$, we would expect that we could determine the community of $v'$ from $y_{v,v'}^{(t)}$ with accuracy $1/2+\Omega(1)$.

The problem with this reasoning is that the average value over all $(v,v')\in E(G)$ of $y_{v,v'}^{(1)}$ will not be exactly $0$. It will also tend to have an absolute value on the order of $1/\sqrt{n}$. That means that the average value over all $(v,v')\in E(G)$ of $y_{v,v'}^{(t)}$ will have an absolute value of $\Theta((\frac{a+b}{2})^t/\sqrt{n})$. If we hold the average value of $y_{v,v'}^{(t-1)}$ fixed then that means that $E[y_{v,v'}^{(t)}|N_1(v')]$ will have an empirical variance of $\Theta((\frac{a+b}{2})^{2t}/n)$, and thus that $y_{v,v'}^{(t)}$ will also have an empirical variance of at least $\Theta((\frac{a+b}{2})^{2t}/n)$. This implies that the standard deviation of $y_{v,v'}^{(t)}$ will always be much greater than the difference between the average value of $y_{v,v'}^{(t)}$ for $v'$ in community $1$ and  the average value of $y_{v,v'}^{(t)}$ for $v'$ in community $2$, which would render attempts to classify $v'$ based on $y_{v,v'}^{(t)}$ ineffective.


\begin{remark}
The simple way to fix this would be to add a step where we subtract the average value of $y^{(t)}$ from every element of $y^{(t)}$ so its sum is $0$ for every $t$ like we did in the version of $ABP$ in Section \ref{ABP}. However, this does not extend easily to the general Stochastic Block Model, and we want a solution that does.
\end{remark}

In order to prevent this, we need to stop the average value of $y_{v,v'}^{(t)}$ from getting too large. It will tend to multiply by roughly $\frac{a+b}{2}$ each time $t$ increases by $1$, so the average value of $y_{v,v'}^{(t)}-\frac{a+b}{2}y_{v,v'}^{(t-1)}$ will probably be much smaller than the average value of $y_{v,v'}^{(t)}$. So, if we pick some $1<i\le m$ and redefine $y_{v,v'}^{(i)}$ so that 
\begin{align}
y_{v,v'}^{(i)}=-\frac{a+b}{2}y_{v,v'}^{(i-1)}+\sum_{v'': (v'',v')\in E(G),v''\ne v} y_{v',v''}^{(i-1)}\label{cancel1}
\end{align}
for all $(v,v')\in E(G)$, the average value of $y_{v,v'}^{(i)}$ will be much smaller than it would have been. Since the difference between the average values of $y_{v,v'}^{(t)}$ over $v'$ in different communities grows as $\Theta((\frac{a-b}{2})^t/\sqrt{n})$, this redefinition will merely change the difference between the average values over $v'$ in different communities of $y_{v,v'}^{(i)}$ from $\Theta((\frac{a-b}{2})^i/\sqrt{n})$ to 
\[\Theta\left(\left(\frac{a-b}{2}\right)^{i}/\sqrt{n}-\frac{a+b}{2}\cdot\left(\frac{a-b}{2}\right)^{i-1}/\sqrt{n}\right)=-\Theta\left(b\left(\frac{a-b}{2}\right)^{i-1}/\sqrt{n}\right).\]
However, the average value of $y_{v,v'}^{(i)}$ will still be nonzero, and if we continue to set $y_{v,v'}^{(t)}=\sum_{v'': (v'',v')\in E(G),v''\ne v} y_{v',v''}^{(t-1)}$ for all $t> i$, the average value of $y_{v,v'}^{(t)}$ would resume increasing in magnitude faster than the average difference between $y_{v,v'}^{(t)}$ for different communities of $v'$. This creates the risk that it would still eventually get too large. So, in order to actually fix the problem, it may be necessary to repeat the step where its magnitude is reduced. More precisely, we may have to choose several indices $t_0,t_1,...t_{m'}$ and redefine $y_{v,v'}^{(t_i)}$ for each $i$ so that   
\begin{align}
y_{v,v'}^{(t_i)}=-\frac{a+b}{2}y_{v,v'}^{(t_i-1)}+\sum_{v'': (v'',v')\in E(G),v''\ne v} y_{v',v''}^{(t_i-1)}\label{cancel2}
\end{align}
for every $(v,v')\in E(G)$.

Once we have made these modifications, it will be the case that for sufficiently large $m$, the average value for $v'$ in community $1$ of $y_{v,v'}^{(m)}$ will differ from the average value for $v'$ in community $2$ of $y_{v,v'}^{(m)}$ by a constant multiple of the standard deviation of $y_{v,v'}^{(m)}$. Then, we define $y_v^{(m)}=\sum_{v':(v,v')\in E(G)} y_{v,v'}^{(m)}$ in order to simplify it to a function of one vertex which is still correlated with the vertex's community. Unfortunately, this does not guarantee that simply dividing the vertices into those with a positive value of $y_v^{(m)}$ and those with a negative value of $y_v^{(m)}$ will give a useful partition. It could be the case that the fraction of $v$ for which $y_v^{(m)}$ is positive is the same for both communities, but $y_v^{(m)}$ is typically more strongly positive or less strongly negative for $v$ in one community than $v$ in the other. So, we randomly assign each vertex to a set with a probability that scales linearly with $y_v^{(m)}$ in order to ensure that having a higher average value of $y_v^{(m)}$ actually leads to having greater representation in one of the proposed communities.\\

\subsection{Vanilla ABP}
We present first a simplified version of our algorithm. Our proof relies on a modified version described below, but this version captures the essence of our algorithm while avoiding technicalities required for the proof.


{\em $\mathrm{ABP}^*(G,m,m',r,\lambda_1)$:}
\begin{enumerate}

\item For each adjacent $v$ and $v'$ in $G$, randomly draw $y^{(1)}_{v,v'}$ from a Gaussian distribution with mean $0$ and variance $1$. Also, consider $y^{(t)}_{v,v'}$ as having a value of $0$ whenever $t<1$.


\item For each $1<t\le m$, and each adjacent $v$ and $v'$ in $G$, set
\[y^{(t)}_{v,v'}=\sum_{v'':(v',v'')\in E(G),v''\ne v} y^{(t-1)}_{v',v''}\]
 unless $(v,v')$ is part of a cycle of length $r$ or less.\footnote{See Remark (2) in Section \ref{ABP} for how to modify the algorithm in the case of multiple cycles.} If it is, then let the other vertex in the cycle that is adjacent to $v$ be $v'''$, and the length of the cycle be $r'$, and set \[y_{v,v'}^{(t)}=\sum_{v'':(v',v'')\in E(G),v''\ne v} y_{v',v''}^{(t-1)}-\sum_{v'':(v,v'')\in E(G),v''\ne v',v''\ne v'''} y_{v,v''}^{(t-r')}\]
unless $t=r'$, in which case, set 
\[y_{v,v'}^{(t)}=\sum_{v'':(v',v'')\in E(G),v''\ne v} y_{v',v''}^{(t-1)}-y^{(1)}_{v''',v}.\]

\item Set $Y$ to be the $n\times m$ matrix such that for all $t$ and $v$, \[Y_{v,t}=\sum_{v':(v',v)\in E(G)} y^{(t)}_{v,v'}\]

\item Set $M$ to be the $m\times m$ matrix such that $M_{i,i}=1$ and $M_{i,i+1}=-\lambda_1$ for all $i$, and all other entries of $M$ are $0$. Also, let $e_m\in\mathbb{R}^m$ be the vector with an $m$-th entry of $1$ and all other entries equal to $0$. Set $$y'=YM^{m'}e_m.$$

\item Return $(\{v:y'_v>0\},\{v: y'_v\le 0\})$.
\end{enumerate}

\begin{remark}In the $r=2$ case, one does not need to find cycles and one can exit step 2.\ after the second line. In general when running this algorithm, one should use variables that are accurate to within a factor of less than $(\lambda_2/\lambda_1)^m$. As mentioned above, we rely on a less compact version of the algorithm to prove the theorem, but expect that the above also succeeds at detection as long as $m>2\ln(n)/\ln(\snr)+\omega(m')$ and $m'>m\ln(\lambda^2_1/\lambda^2_2)/(\ln(n)-\omega(1))$. 
\end{remark}

\begin{remark}
This version of ABP is mostly the same as the one in Section \ref{ABP}. However, the previous version compensates for biases in $y^{(1)}$ by subtracting the average value of $y^{(t)}$ from each entry in $y^{(t)}$ for every $t$ in order to prevent it from ever becoming significantly biased. This version compensates for biases by a variant of the method explained in the previous subsection instead. More specifically, it uses the method explained in the previous subsection except that it calculates all of the $y^{(t)}$ without using any form of compensation and then adjusts the results in order to essentially apply the compensation retroactively.
\end{remark}

\noindent
{\bf Implementation details.} 
Note that $ABP^*$ has several differences from the algorithm outlined in previous sections. First of all, we initialize the $y_{v,v'}^{(0)}$ using a random value that is drawn from a normal distribution because a probability distribution that is a multidimensional normal distribution is easier to analyse than a probability distribution that is evenly distributed over the vertices of an $n$-dimensional hypercube. Secondly, we require that our walks never repeat the same vertex within $r$ steps for some $r$, rather than merely requiring that they not backtrack. This allows us to use the expected number of walks between vertices in our analysis without worrying about the tiny probability that there is a dense tangle in the graph with a huge number of nonbacktracking walks between its vertices. Making this modification to the algorithm requires adding  an extra part to the recursion step where walks that just repeated a vertex are cancelled out, specifically the second half of step $3$ of the algorithm above. The resulting algorithm is called the {\em acyclic belief propagation} algorithm because it counts walks that do not contain any small cycles. Thirdly, we move all of the recursion steps that compensate for the average value to the end of the algorithm. This is possible because the operation that takes $y^{(t-1)}$ as input and outputs a list that has a value of $$-\frac{a+b}{2}y_{v,v'}^{(t-1)}+\sum_{v'': (v'',v')\in E(G),v''\ne v} y_{v',v''}^{(t-1)}$$ for each $(v,v')\in E$ commutes with the one that simply outputs a list that has a value of $\sum_{v'': (v'',v')\in E(G),v''\ne v} y_{v',v''}^{(t-1)}$ for each $(v,v')\in E$. The algorithm generates $y^{(m)}$ by applying these two operations in some sequence to $y^{(1)}$, so we can calculate it by applying the later operation $m-m'$ times and then applying the former operation $m'$ times. Actually, the algorithm takes this one step farther by applying the later operation $m$ times and then calculating how the result would have changed if it had applied the former the appropriate number of times, but it still has the same result. 

The full Acyclic Belief Propagation algorithm also has a few differences from $ABP^*$ that make it easier to prove that it works. In particular, we randomly select a small fraction of the graph's edges at the beginning of the algorithm. Then, we require one specific step of each nonbacktracking walk to use one of the selected edges, and all of their other steps to use edges that have not been selected. Since the selected edges are nearly independent of the rest of the graph, this allows us to more easily prove that the values of $y_{v',v''}^{(t-1)}$ for $v'$ adjacent to $v$ will not become dependent in a way that disrupts the algorithm.


\subsection{ABP for the symmetric SBM}\label{simplified}

\begin{theorem}
Let $a$ and $b$ be positive real numbers such that $(a-b)^2>2(a+b)$, and $S$ be the $2$-community symmetric stochastic block model with these parameters. There exist constants $\epsilon, l, r,c>0$ and $m=\Theta(\log(n))$ such that when the basic $2$-community symmetric acyclic belief propagation algorithm is run on these parameters and a random $G\in S$, the expected difference between the fraction of vertices from community $1$ that are in $S_1$ and the fraction of vertices from community $2$ that are in $S_1$ is at least $\epsilon$.
\end{theorem}

The basic $2$-community symmetric acyclic belief propagation algorithm is as follows.

$2CS-ABP^\star(G,m,r,l,c,\lambda_1):$
\begin{enumerate}

\item Find all cycles of length $r$ or less in $G$. 

\item For every vertex $v\in G$, randomly assign $x_v$ according to a Normal distribution with mean $0$ and variance $1$. 
For each adjacent $v$ and $v'$, let $y_{v,v'}^{(1)}=x_{v'}$, and $y_{v,v'}^{(t)}=0$ for all $t\le 0$.

\item For each $1\le t\le m$, and each adjacent $(v,v')\in E(G)$, set 
\[y_{v,v'}^{(t)}=\sum_{v'':(v',v'')\in E(G),v''\ne v} y_{v',v''}^{(t-1)}\] unless $v,v'$ is part of a cycle of length $r$ or less. If it is, then let the other vertex in the cycle that is adjacent to $v$ be $v'''$, and the length of the cycle be $r'$. Set \[y_{v,v'}^{(t)}=\sum_{v'':(v',v'')\in E(G),v''\ne v} y_{v',v''}^{(t-1)}-\sum_{v'':(v,v'')\in E(G),v''\ne v',v''\ne v'''} y_{v,v''}^{(t-r')}\]
unless $t=r'$. In that case, set 
\[y_{v,v'}^{(t)}=\sum_{v'':(v',v'')\in E(G),v''\ne v} y_{v',v''}^{(t-1)}-x_v.\]

\item Set $Y$ to be the $n\times m$ matrix such that for all $t$ and $v$, \[Y_{v,t}=\sum_{v':(v',v)\in E(G)} y^{(t)}_{v,v'}\]

\item Set $M$ to be the $m\times m$ matrix such that $M_{i,i}=1$ and $M_{i,i+1}=-\lambda_1$ for all $i$, and all other entries of $M$ are equal to $0$. Also, let $e_m\in\mathbb{R}^m$ be the vector with an $m$-th entry of $1$ and all other entries equal to $0$. Set $$y^{(m)}=Y M^{\lceil \frac{m-3r-1}{l}\rceil}e_m.$$

\item Set $c'=c\cdot \sqrt{\sum_{v\in G} (y^{(m)}_v)^2/n}$. Create sets of vertices $S_1$ and $S_2$ as follows. For each vertex $v$, if $y^{(m)}_v<-c'$, assign $v$ to $S_1$. If $y^{(m)}_v>c'$, then assign $v$ to $S_2$. Otherwise, assign $v$ to $S_2$ with probability $1/2+y^{(m)}_v/2c'$ and $S_1$ otherwise. Return $(S_1,S_2)$.
\end{enumerate}

We believe that the difference between the fraction of vertices from community $1$ that this algorithm puts in $S_1$ and the fraction of vertices from community $2$ that this algorithm puts in $S_1$ will be at least $\epsilon$ with probability $1-o(1)$. However, in order to prove that we can detect communities reliably we use the following slightly modified form of the algorithm.

\begin{theorem}
Let $a$ and $b$ be positive real numbers such that $(a-b)^2>2(a+b)$, and $S$ be the $2$-community symmetric stochastic block model with these parameters. There exist constants $\epsilon, l, r,c,\gamma>0$ and $m=\Theta(\log(n))$ such that when the $2$-community symmetric acyclic belief propagation algorithm is run on these parameters and a random $G\in S$, the difference between the fraction of vertices from community $1$ that are in $S_1$ and the fraction of vertices from community $2$ that are in $S_1$ is at least $\epsilon$ with probability $1-o(1)$.
\end{theorem}

The $2$-community symmetric acyclic belief propagation algorithm is as follows.

$2CS-ABP(G,m,r,l,c,\lambda_1,\gamma):$
\begin{enumerate}

\item Randomly and independently select each edge in $G$ with probability $\gamma$. Put all of the selected edges in a set $\Gamma$, and remove them from the graph.


\item For every vertex $v\in G$, randomly assign $x_v$ according to a Normal distribution with mean $0$ and variance $1$.
 For each adjacent $v$ and $v'$, let $y_{v,v'}^{(1)}=x_{v'}$, and $y_{v,v'}^{(t)}=0$ for all $t\le 0$.

\item For each $1\le t\le m$, and each adjacent $(v,v')\in E(G)$, set 
\[y_{v,v'}^{(t)}=\sum_{v'':(v',v'')\in E(G),v''\ne v} y_{v',v''}^{(t-1)}\] unless $v,v'$ is part of a cycle of length $r$ or less. If it is, then let the other vertex in the cycle that is adjacent to $v$ be $v'''$, and the length of the cycle be $r'$. Set \[y_{v,v'}^{(t)}=\sum_{v'':(v',v'')\in E(G),v''\ne v} y_{v',v''}^{(t-1)}-\sum_{v'':(v,v'')\in E(G),v''\ne v',v''\ne v'''} y_{v,v''}^{(t-r')}\]
unless $t=r'$. In that case, set 
\[y_{v,v'}^{(r')}=\sum_{v'':(v',v'')\in E(G),v''\ne v} y_{v',v''}^{(r'-1)}-x_v.\]

\item Set $Y$ to be the $n\times m$ matrix such that for all $t$ and $v$, \[Y_{v,t}=\sum_{v':(v',v)\in E(G)} y^{(t)}_{v,v'}\]

\item Set $M$ to be the $m\times m$ matrix such that $M_{i,i}=1$ and $M_{i,i+1}=-(1-\gamma)\lambda_1$ for all $i$, and all other entries of $M$ are equal to $0$. Also, let $e_m\in\mathbb{R}^m$ be the vector with an $m$-th entry of $1$ and all other entries equal to $0$. Set $$y^{(m)}=Y M^{\lceil \frac{m-3r-1}{l}\rceil}e_m.$$

\item For each $v\in G$, set \[y'_v=\sum_{v':(v,v')\in\Gamma} y^{(m)}_{v'}\]

\item For each $v\in G$, set $y''_v$ to be the sum of $y'_{v'}$ for all vertices $v'$ that are exactly $\lfloor \sqrt{\ln\ln n}\rfloor$ edges away from $v$.

\item Set $c'=c\cdot \sqrt{\sum_{v\in G} (y''_v)^2/n}$. Create sets of vertices $S_1$ and $S_2$ as follows. For each vertex $v$, if $y''_v<-c'$, assign $v$ to $S_1$. If $y''_v>c'$, then assign $v$ to $S_2$. Otherwise, assign $v$ to $S_2$ with probability $1/2+y''_v/2c'$ and $S_1$ otherwise. Return $(S_1,S_2)$.
\end{enumerate}

\begin{remark}
Note that the last theorem is simply Theorem \ref{main1} in the $2$ community symmetric case with some parts of the initialization step replaced by an assumption that parameters are chosen correctly. Also, the one before it follows from a slightly simplified version of the same proof.
\end{remark}

\begin{remark}
For the $k$-community symmetric stochastic block model, the above is largely unchanged. The key differences are that $\lambda_1=\frac{a+(k-1)b}{k}$, $\lambda_2=\frac{a-b}{k}$, the requirement for the algorithm to work is that $(a-b)^2>k(a+(k-1)b)$, and if the requirements are met the $k$CS-ABP$^*$ algorithm distinguishes between every pair of communities with expected accuracy at least $\epsilon$.
\end{remark}

\begin{remark}
In this algorithm, the initialization step and computing $y^{(t)}$ for a given $t$ both run in $O(n)$ time. $y^{(m)}$ can be computed in $O(n\log n)$ time if it is done the efficient way, $y'$ takes $O(n)$ time to compute, and computing $y''$ takes $O(n\log n)$ time. Determining $S_1$ and $S_2$ from $y''$ can also be done in $O(n)$ time, so this whole algorithm can be run in $O(n\log n)$ time.
\end{remark}

\subsection{Spectral view of ABP}\label{spectral}
An alternative perspective on this algorithm is the following. Assume for the moment that there are exactly $n/2$ vertices in each community, and let $M$ be the expected adjacency matrix, the matrix such that $M_{v,v'}$ is $a/n$ if $v$ and $v'$ are in the same community and $b/n$ if they are not. This matrix has an eigenvector whose entries are all $1$ with eigenvalue $\frac{a+b}{2}$, an eigenvector whose entries are $\pm 1$ with their signs determined by the relevant vertices' communities that has eigenvalue $\frac{a-b}{2}$, and all of its other eigenvalues are $0$. 

Now, let $M'$ be the graph's actual adjacency matrix. The result above suggests that the second eigenvector of $M'$ may have entries that are correlated with the vertices' communities. The problem with this reasoning is that while $M'$ has an expected value of $M$, $(M')^2$ has an expected value of roughly $M^2+\frac{a+b}{2}I$ because for every $i,j$, $M'_{i,j}=1\iff M'_{j,i}=1$, with the result that $E[\sum_{j} M'_{i,j}\cdot M'_{j,i}]$ is very different from $\sum_{j} E[M'_{i,j}]\cdot E[M'_{j,i}]$. In other words, the square of the adjacency matrix counts walks of length $2$ and has an expected value that is significantly different from the square of the expected adjacency matrix due to backtracking.

In order to avoid this issue, we define the graph's nonbacktracking walk matrix $W$ as a matrix over the vector space with an orthonormal basis consisting of a vector for each directed edge in the graph. $W_{(v_1,v_2),(v'_1,v'_2)}$ is defined to be $1$ if $v'_2=v_1$ and $v_2\ne v'_1$ and $0$ otherwise. In other words, it has a $1$ for every case where one directed edge leads to another that is not the same edge in the other direction.

Now, let $w\in R^{2|E(G)|}$ be the vector whose entries are all $1$, and $w'\in R^{2|E(G)|}$ be the vector such that $w'_{(v_0,v_1)}$ is $1$ if $v_0$ is in community $1$ and $-1$ if $v_0$ is in community $2$. As mentioned before, for a small $t$ and a random $(v,v')\in E(G)$, there will be an average of approximately $(\frac{a+b}{2})^t$ directed edges $t$ edges in front of $(v,v')$, and approximately $(\frac{a-b}{2})^t$ more of these edges will have ending vertices in the same community as $v'$ than in the other community on average. So, $w\cdot W^t w\approx 2|E(G)|(\frac{a+b}{2})^t$ and $w'\cdot W^tw'\approx 2|E(G)|(\frac{a-b}{2})^t$. That strongly suggests that $W$ has eigenvectors that are correlated with $w$ and $w'$ that have eigenvalues of approximately $\frac{a+b}{2}$ and $\frac{a-b}{2}$ respectively. It also seems plausible that $W$'s other eigenvalues have relatively small magnitudes.

If this is true, then one can gain information on which vertices of $G$ are in each community from the second eigenvector of $W$. One could simply calculate the second eigenvector of $W$ directly. However, it is significantly faster to pick a random vector $w''$ and then compute $W^m w''$ for some suitable $m$. The resulting vector will be approximately a linear combination of $W$'s main eigenvectors. Unfortunately, it will be much closer to being a multiple of its first eigenvector than its second. If we multiply $(W-\frac{a+b}{2}I)$ by the resulting vector, such as in \eqref{cancel1}, the component of the vector that is proportional to the first eigenvector will be mostly cancelled out. However, since its eigenvalue is not exactly $\frac{a+b}{2}$, it will not be cancelled out completely, and might still be too large. Luckily, if we instead multiply $(W-\frac{a+b}{2}I)^{m'}$ by the resulting vector for suitable $m'$, such as in \eqref{cancel2}, the component of the vector that is proportional to the first eigenvalue will be essentially cancelled out, leaving a vector that is approximately a multiple of $W$'s second eigenvector, and thus correlated with $G$'s communities. We believe that this would succeed in detecting communities in the SBM, but in order to make it easier to prove that our algorithm works we actually use $r$-nonbacktracking walks. This corresponds to using the graph's $r$-nonbacktracking walk matrix, which is defined as follows.
\begin{definition}
For any $r$, the graph's $r$-nonbacktracking walk matrix, $W^{(r)}$, is a matrix over the vector space with an orthonormal basis consisting of a vector for each directed path of length $r-1$ on the graph. $W^{(r)}_{(v_1,v_2,...,v_r),(v'_1,v'_2,...,v'_r)}$ is $1$ if $v'_{i+1}=v_i$ for each $1\le i<r$ and $v'_1\ne v_r$, otherwise it is $0$. In other words, $W^{(r)}$ maps a path of length $r-1$ to the sum of all paths resulting from adding another element to the end of the path and deleting its first element.
\end{definition}

\begin{figure}[H]
\centering
  \includegraphics[width=0.4\linewidth]{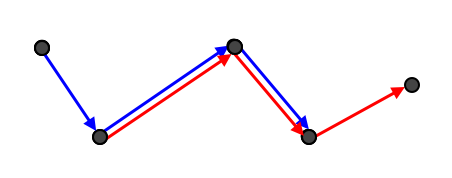}
\caption{These two paths (blue and red) lead to an entry of 1 in the $W^{(4)}$ matrix.}
\label{loops}
\end{figure}

Performing these calculations is essentially what the acyclic belief propagation algorithm does. From this perspective, the algorithm roughly translates to:
\begin{enumerate}
\item Choose $y^{(1)}$ randomly such that each element is independently drawn from a Normal distribution.

\item For each $1< t\le m$, let $y^{(t)}=W^{(r)}y^{(t-1)}$.

\item Change $y^{(m)}$ to $(W^{(r)}-\frac{a+b}{2}I)^{m'}y^{(m-m')}$, where $m'=\lceil\frac{m-3r-1}{l}\rceil$

\item Set\footnote{Proceed as in step 6 of 2CS-ABP* for the proof.} $y'_v=\sum_{v':(v',v)\in E(G)} y^{(m)}_{v,v'}$ for every $v\in G$.
 Return $(\{v:y'_v>0\},\{v: y'_v\le 0\})$.

%
%
\end{enumerate}
Even though $r=2$ might suffice to achieve the KS threshold in the SBM, the use of larger $r$ might help for other graph models, e.g., having more short cycles. 

\subsection{ABP for the general SBM}
Now, consider a graph $G$ drawn from $SBM(n,p,Q/n)$ with arbitrary $p$ and $Q$. Also, let $\lambda_1,\lambda_2,...\lambda_h$ be the distinct eigenvalues of $PQ$ in order of nonincreasing magnitude. If the parameters are such that vertices from different communities have different expected degrees, then one can detect communities by simply dividing its vertices into those with above-average degrees and those with below-average degrees. So, assume that the expected degree of a vertex is independent of its community. Detecting communities in the general case runs into some obstacles that do not apply in the $2$-community symmetric case. First of all, it is much less clear that assigning vertices to sets randomly is a useful start. Also, even if we did have reasonable preliminary guesses of which community each vertex was in, it is not obvious how to determine a vertex's community based on the alleged communities of the vertices a fixed distance from it.

For the moment, assume that for each vertex $v$, we have a vector $x_v$ such that we believe $v$ is in community $i$ with probability $p_i+x_v\cdot e_i$ for each $i$, where all elements of $x_v$ are small. Furthermore, assume that $x_v$ is generated independently of $v$'s neighbors. The correct belief about the probability that $v$ is in each community once its neighbors are taken into account is
\[p+x_v+\frac{1}{\lambda_1}\sum_{v':(v,v')\in E[G]} PQ x_{v'}\]
up to nonlinear terms in the $x$'s. So, given $m$ small enough that the set of vertices within $m$ edges of $v$ is a tree, the correct belief about what community $v$ is in once all of the vertices within $m$ edges of $v$ are taken into account is
\[p+\sum_{0\le m'\le m}\lambda_1^{-m'}\sum_{v':d(v,v')=m'} (PQ)^{m'} x_{v'}\]
up to nonlinear terms in the $x$'s. So, the logical belief about the probability that $v$ is in each community based only on the preliminary guesses concerning the vertices $m$ edges away from $v$ is
\[p+\lambda_1^{-m}\sum_{v':d(v,v')=m} (PQ)^m x_{v'}.\]
 Conveniently, this expression is linear, so if $w$ is an eigenvector of $PQ$ with eigenvalue $\lambda_i$ for $i\ne 1$, then
\begin{align*}
E[w\cdot P^{-1} e_{\sigma_v}]&\approx w\cdot P^{-1}\left(p+\lambda_1^{-m}\sum_{v':d(v,v')=m} (PQ)^m x_{v'}\right)\\
&=\lambda_1^{-m}\lambda_i^m\sum_{v':d(v,v')=m} w\cdot P^{-1}x_{v'}.
\end{align*}
In particular, this means that we only need an initial estimate for $w\cdot P^{-1}e_{\sigma_{v'}}$ for every vertex in the graph, rather than needing a full set of beliefs about the vertices' communities. Any random guesses we make will probably have correlation $\pm\Omega(1/\sqrt{n})$ with $w\cdot P^{-1}e_{\sigma_{v'}}$, so we can use them as a starting point.

Unfortunately, just like in the two-community symmetric case, the graph will run out of vertices before $m$ becomes large enough to amplify our beliefs enough. However, switching from a sum over all vertices $v'$ that are $m$ edges away from $v$ to a sum over all nonbacktracking walks of length $m$ ending in a vertex $v'$ fixes this problem the same way it does in the two-community symmetric case. Likewise, we can still compute this sum by randomly dividing $G$'s vertices between two sets $S_1$ and $S_2$ and then using the following algorithm.
\begin{algorithm}
\begin{enumerate}
\item For every $(v,v')\in E(G):$

\hspace{1 cm} If $v'\in S_2$, set $y_{v,v'}^{(1)}=1$

\hspace{1 cm} Otherwise, set $y_{v,v'}^{(1)}=-1$

\item For every $1<t\le m$ and $(v,v')\in E(G):$

\hspace{1 cm} Set $y_{v,v'}^{(t)}=\sum_{v'': (v'',v')\in E(G),v''\ne v} y_{v',v''}^{(t-1)}$

\item For every $(v,v')\in E(G)$ and $v\in G$

\hspace{1 cm} Set $y_v^{(m)}=\sum_{v':(v,v')\in E(G)} y_{v,v'}^{(m)}$
\end{enumerate}
\end{algorithm}

However, needing to compensate for the average value is a special case of a considerably more complicated phenomenon. The average value over all $(v,v')\in E(G)$ of $y_{v,v'}^{(1)}\cdot w_{\sigma_{v'}}$ will typically have a magnitude of $\Theta(1/\sqrt{n})$, and for general $t$ the average value over all $(v,v')\in E(G)$ of $y_{v,v'}^{(t)}\cdot w_{\sigma_{v'}}$ will typically have a magnitude of $\Theta(|\lambda_i^t|/\sqrt{n})$. Now, let $w'$ be an eigenvector of $PQ$ with an eigenvalue of $\lambda_{i'}$ which has greater magnitude than $\lambda_i$. Then the average value over all $(v,v')\in E(G)$ of $y_{v,v'}^{(t)}\cdot w'_{\sigma_{v'}}$ will typically have a magnitude of $\Theta(|\lambda_{i'}^t|/\sqrt{n})$. Since this grows faster than $y_{v,v'}^{(t)}\cdot w_{\sigma_{v'}}$ does, it will eventually become large enough to disrupt efforts to estimate $w_{\sigma_{v'}}$ using $y_{v,v'}^{(t)}$ the same way the average value of $y^{(t)}$ did in the $2$-community symmetric case. In fact, the issue with the average value is just the subcase of this when $i'=1$. So, in order to deal with this, we need to compensate for each eigenvalue, $\lambda_{i'}$, of $PQ$ with magnitude greater than $\lambda_i$ by choosing several indices $t_{0,i'},t_{1,i'},...t_{m',i'}$ and redefining $y_{v,v'}^{(t_{j,i'})}$ for each $j$ so that   
\[y_{v,v'}^{(t_{j,i'})}=-\lambda_{i'}y_{v,v'}^{(t_{j,i'}-1)}+\sum_{v'': (v'',v')\in E(G),v''\ne v} y_{v',v''}^{(t_{j,i'}-1)}\]
for every $(v,v')\in E(G)$. Assuming that this is done, $y_{v,v'}^{(1)}$ has a variance of approximately $1$, and then $y_{v,v'}^{(t)}$ has a variance of roughly $\lambda_1^t$. It becomes possible to determine which community $v$ is in with accuracy nontrivially greater than that obtained by guessing randomly based on $y_{v,v'}^{(t)}$ when the expected difference between its values for $v$ in different communities is within a constant factor of its standard deviation. In other words, $t$ needs to be large enough that $|\lambda_i^t|/\sqrt{n}$ is significant relative to $\sqrt{\lambda_1^t}$. If $\lambda_1\ge \lambda_i^2$ then this will never happen, so the algorithm requires that $$\lambda_i^2>\lambda_1.$$

The general acyclic belief propagation algorithm is almost the same as the $2$-community symmetric version. However, it takes a list of eigenvectors as input instead of just $\lambda_1$. Also, the step compensating for larger eigenvalues is changed from ``Set $y^{(m)}=Y M^{\lceil \frac{m-3r-1}{l}\rceil}e_m$ where $M$ is the matrix such that $M_{i,i}=1$ for all $i$, $M_{i,i+1}=-(1-\gamma)\lambda_1$ for all $i$ and all other entries of $M$ are $0$" to ``Set $y^{(m)}=Y \prod_{s'<s} M_{s'}^{\lceil \frac{m-r-(2r+1)s'}{l}\rceil}e_m$ where $M_{s'}$ is the matrix such that $(M_{s'})_{i,i}=1$ for all $i$, $(M_{s'})_{i,i+1}=-(1-\gamma)\lambda_{s'}$ for all $i$ and all other entries of $M_{s'}$ are $0$" Its effectiveness is described by the following theorem.

\begin{theorem}
Let $p\in (0,1)^k$ with $\sum p=1$, $Q$ be a symmetric matrix with nonnegative entries, $P$ be the diagonal matrix such that $P_{i,i}=p_i$, and $\lambda_1,...,\lambda_{h}$ be the eigenvalues of $PQ$ in order of nonincreasing magnitude. If $\lambda_2^2>\lambda_1$ then there exist constants $\epsilon, r,c$ and $m=\Theta(\log(n))$ such that when the acyclic belief propagation algorithm is run on these parameters and a random $G\in SBM(n,p,Q/n)$, with probability $1-o(1)$ there exist $\sigma$ and $\sigma'$ such that the difference between the fraction of vertices from community $\sigma$ that are in $S_1$ and the fraction of vertices from community $\sigma'$ that are in $S_1$ is at least $\epsilon$. The algorithm can be run in $O(n\log n)$ time.
\end{theorem}

\begin{remark}
Let $s=3$ if $|\lambda_2|=|\lambda_3|$ and $s=2$ otherwise. This theorem could alternately have stated that there exists $\epsilon$ such that for any two communities $\sigma$ and $\sigma'$ such that there exists an eigenvector $w$ of $PQ$ with eigenvalue $\lambda_s$ such that $w_\sigma\ne w_{\sigma'}$, the expected difference between the fraction of vertices from community $\sigma$ that are in $S_1$ and the fraction of vertices from community $\sigma'$ that are in $S_1$ is at least $\epsilon$.
\end{remark}

\subsection{Alternatives}
There are also a couple of other variants of these ideas that may be useful for community detection. For instance, if we pick $r\approx \ln(\ln (n))$ and then define $\Sigma$ to be the $n\times n$ symmetric matrix such that $\Sigma_{v,v'}$ is the number of nonbacktracking walks of length $r$ between $v$ and $v'$, we suspect that $\Sigma$'s eigenvector of second largest magnitude will have entries that are correlated with the corresponding vertices' communities. Like in the standard case, we expect that we could get an approximation of this eigenvector by taking a random vector $w$ and then computing $(\Sigma-\lambda'_1I)^{m'}\Sigma^{m-m'}w$ for suitable $m$ and $m'$ where $\lambda'_1$ is an estimate of $\Sigma$'s largest eigenvalue. 

We can compute $\Sigma$ as follows. First, let $\Sigma^{(t)}$ be the $n\times n$ matrix such that  $\Sigma^{(t)}_{v,v'}$ is the number of nonbacktracking walks of length $t$ between $v$ and $v'$. Then $\Sigma^{(0)}=I$, $\Sigma^{(1)}$ is the graph's adjacency matrix, and $\Sigma^{2}_{v,v'}$ is equal to the number of shared neighbors $v$ and $v'$ have for all $v$ and $v'$. For every $t>2$, we have that $\Sigma^{(t)}=\Sigma^{(1)}\cdot \Sigma^{(t-1)}-D\cdot \Sigma^{(t-2)}$, where $D$ is the diagonal matrix such that $D_{v,v}$ is one less than the degree of $v$ for all $v$. This can be used to efficiently compute $\Sigma=\Sigma^{(r)}$.

Also, instead of prohibiting repeating a vetex within $r$ steps, we could address the issue of tangles by dividing $G$'s edges between sets $E_0,...,E_{m'}$ for suitable $m'$ such that most of the edges are assigned to $E_0$ and the rest are assigned to one of the others at random. Then we count nonbacktracking walks with the restriction that edge $r$ of the walk must be from $E_1$, edge $2r$ must be from $E_2$ and so on, while all other edges must be from $E_0$ for suitable $r$. The periodic prohibitions on using edges from $E_0$ would force the walk to leave any tangle it had been in, while the fact that most of the edges are chosen from $E_0$ prevents the restriction from reducing the number of walks too severely.

\section{Crossing the KS threshold: proof technique}\label{proof_tech2}
Recall that the algorithm samples a typical clustering uniformly at random in the typical set 
\begin{align*}
T_\delta(G)=&\{   x \in \mathrm{Bal}(n,k,\delta):\\
&  \sum_{i=1}^k |\{  G_{u,v} : (u,v) \in {[n] \choose 2} \text{ s.t. } x_u=i, x_v=i \}|  \geq \frac{an}{2k} (1-\delta),\\
& \sum_{i,j \in [k], i<j} |\{  G_{u,v} : (u,v) \in {[n] \choose 2} \text{ s.t. } x_u=i, x_v=j \}|  \leq \frac{bn(k-1)}{2k} (1+\delta) \},
\end{align*}
where the previous two inequalities apply to the case $a>b$, and are flipped if $a<b$.
A first question is to estimate the likelihood that a bad clustering, i.e., one that has an overlap that is close to $1/k$, belongs to the typical set. This means the probability that a clustering which splits each of the true clusters into $k$ groups belonging to each community still manages to keep the right proportions of edges inside and across the clusters. This is unlikely to take place, but we care about the exponent of this rare event probability. 

\begin{figure}[h]
\centering
  \includegraphics[width=.9\linewidth]{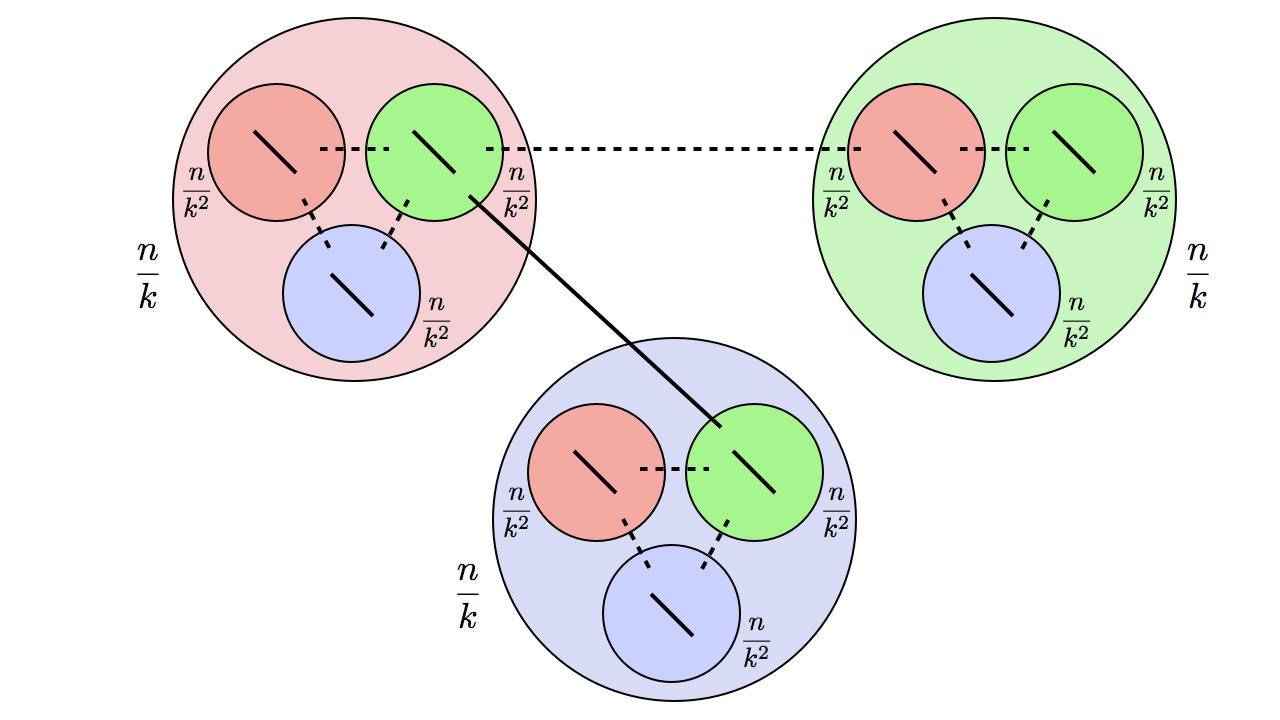}
  \caption{A bad clustering roughly splits each community equally among the $k$ communities. Each pair of nodes connects with probability $a/n$ among vertices of same communities (i.e., same color groups, plain line connections), and $b/n$ across communities (i.e., different color groups, dashed line connections). Only some connections are displayed in the Figure to ease the visualization.}
  \label{bad_cluster}
\end{figure}

As illustrated in Figure \ref{bad_cluster}, the number of edges that are contained in the clusters of a bad clustering is roughly distributed as the sum of two Binomial random variables, 
\begin{align}
E_\mathrm{in} \stackrel{\cdot}{\sim} \bin \left(\frac{n^2}{2k^2},\frac{a}{n}\right) + \bin \left(\frac{(k-1)n^2}{2k^2},\frac{b}{n}\right),
\end{align}
where we use $\stackrel{\cdot}{\sim}$ to emphasize that this is an approximation since the bad clustering may not be perfectly balanced. Note that the expectation of the above distribution is $\frac{n}{2k} \frac{a+(k-1)b}{k}$. In contrast, the true clustering would have a distribution given by $\bin(\frac{n^2}{2k},\frac{a}{n})$, which would give an expectation of $\frac{an}{2k}$.
In turn, the number of edges that are crossing the clusters of a bad clustering is roughly distributed as
\begin{align}
E_\mathrm{out} \stackrel{\cdot}{\sim}  \bin \left(\frac{n^2(k-1)}{2k^2} ,\frac{a}{n}\right) + \bin \left(\frac{n^2(k-1)^2}{2k^2} ,\frac{b}{n}\right),
\end{align}
which has an expectation of $\frac{n(k-1)}{2k} \frac{a+(k-1)b}{k}$.
In contrast, the true clustering would have the above replaced by $\bin(\frac{n^2(k-1)}{2k},\frac{b}{n})$, and an expectation of $\frac{bn(k-1)}{2k}$. 

Thus, we need to estimate the rare event that a Binomial sum deviates from its expectation. While there is a large list of bounds on Binomial tail events, the number of trials here is quadratic in $n$ and the success bias decays linearly in $n$, which requires particular care to ensure tight bounds. We derive these by hand in Lemma \ref{bad_bound}, which gives for a bad clustering $x$, 
\begin{align}
&\pp\{ x \text{ is typical}\} \approx \exp \left(- \frac{n}{k} A \right) 
\end{align}
where  
\begin{align}
&A =  \frac{a+b(k-1)}{2} \ln \frac{k}{(a+(k-1)b)} +  \frac{a}{2} \ln a + \frac{b(k-1)}{2} \ln b .
\end{align}
One can then use a union bound, since there are at most $k^n$ bad clusterings, to obtain a first regime where no bad clustering is typical with high probability. This already allows us to cross the KS threshold in some regime of the parameters when $k\geq 5$. However, this does not interpolate the correct behavior of the information-theoretic bound in the extreme regime of $b=0$, nor crosses at $k=4$. In fact, for $b=0$, the union bound requires $a>2k$ to imply no bad typical clustering with high probability, whereas as soon as $a>k$, an algorithm that simply separates the giants in $\sbm(n,k,a,0)$ and assigns communities uniformly at random for the other vertices solves detection. Thus when $a \in (k,2k]$, the union bound is loose. To remediate to this, we next take into account the topology of the SBM graph. 

Since the algorithm samples a typical clustering, we only need the number of bad and typical clusterings to be small compared to the total number of typical clusterings with high probability. Thus, we seek to better estimate the total number of typical clusterings. The first topological property of the SBM graph that we exploit is the large fraction of nodes that are in tree-like components outside of the giant. 
Conditioned on being on a tree, the SBM labels are distributed as in a broadcasting problem on a (Galton-Watson) tree. Specifically, for a uniformly drawn root node $X$, each edge in the tree acts as a symmetric channel, producing the output 
\begin{align}
Y=X + Z \mod k,
\end{align}  
where 
\begin{align}
Z \sim \nu:=\left(\frac{a}{a+(k-1)b}, \frac{b}{a+(k-1)b}, \dots, \frac{b}{a+(k-1)b} \right), 
\end{align}
and this propagates down the tree. Thus, labelling the nodes in the trees according to the above distribution and freezing the giant to the correct labels leads to a typical clustering with high probability. 

\begin{figure}[h]
\centering
  \includegraphics[width=1\linewidth]{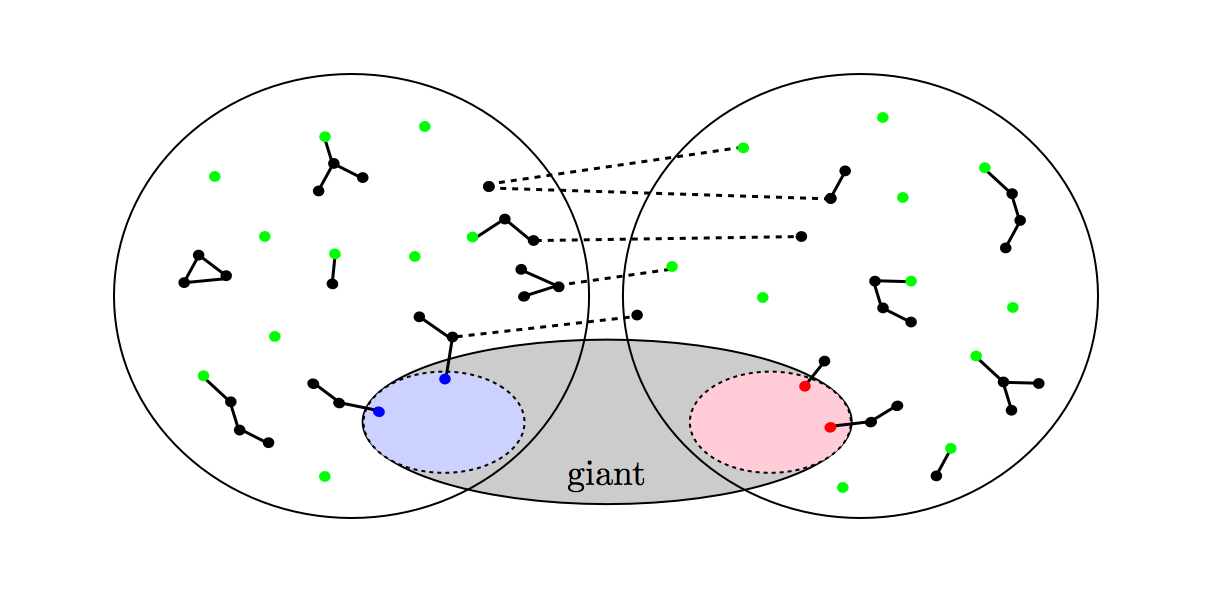}
  \caption{Illustration of the topology of $\sbm(n,k,a,b)$ for $k=2$. A giant component covering the two communities takes place when $d=\frac{a+(k-1)b}{k}>1$; a linear fraction of vertices belong to isolated trees (including isolated vertices), and a linear fraction of vertices in the giant are on planted trees. The following is used to estimate the size of the typical set in Section \ref{size}. For isolated trees, sample a bit uniformly at random for a vertex (green vertices) and propagate the bit according to the symmetric channel with flip probability $b/(a+(k-1)b)$ (plain edges do not flip whereas dashed edges flip). For planted trees, do the same but freeze the root bit to its true value.}
  \label{topo}
\end{figure}

We hence need to count the number of nodes $T$ and edges $M$ that belong to such trees in the SBM graph. This is done in a series of lemmas in Section \ref{size}, and requires combinatorial estimates similar to those carried for the Erd\H{o}s-R\'enyi case \cite{ER-seminal2}. The main part is to show that the fraction of such nodes and edges concentrates, i.e.,  
\begin{align}
&T/n  \in  \left[\frac{\tau}{d}\left(1-\frac{\tau}{2}\right) - \e,\frac{\tau}{d}\left(1-\frac{\tau}{2}\right)+\e \right] \\
&M/n  \in \left[\frac{\tau^2}{2d} - \e, \frac{\tau^2}{2d} + \e \right] ,
\end{align}
with high probability for any $\e>0$, where $\tau$ is the unique solution in $(0,1)$ of 
\begin{align}
\tau e^{-\tau}=de^{-d} \label{tau_def}
\end{align}
or equivalently $\tau = \sum_{j=1}^{+ \infty} \frac{j^{j-1}}{j!} (de^{-d})^j$.

Using next entropic bounds for these random variables, i.e., the fact that there are approximately $2^{N H(\nu)}$ typical sequences for the product distribution $\mu^N$, we can turn previous estimates into a bound on the typical set size (see Theorem \ref{size_bound}). The resulting bound gives that a sampled clustering is good with high probability when $a>k$ for $b=0$, i.e., it achieves the right bound in that extreme. However, this is unlikely to behave correctly for small values of $b$. To further improve on this, we next take into account the vertices in the giant that belong to trees, which we call planted tress, and follow the same program as above except that the root node (in the giant) is now frozen to the correct label rather than being uniformly drawn. This addition gives a bound conjectured to capture the correct behavior for $b$ small. We finish by tightening our estimates on the typical set's size by taking into account vertices that are not saturated, i.e., whose neighbors do not cover all communities and who can thus be swapped without affecting typicality. 

The final bound in Theorem \ref{main2} crosses the KS threshold at $k=4$, interpolates the optimal threshold at $a=0$, and is conjectured to be tight  in the scaling of $b$ for small $b$ and fixed $k$, as well as in the scaling of large $k$ for fixed $a,b$.

\section{Open problems}\label{open}
{\bf Impossibility statements.} Further conjectures were made in \cite{decelle} concerning impossibility statements: 
\begin{conjecture}
Under the same model as in Conjecture 1, 
irrespective of the value of $k$, if $\snr<1$, it is impossible to detect communities in polynomial time.
\end{conjecture}
This is already proved for $k=2$ in \cite{Mossel_SBM1}, using a reduction to the reconstruction problem on tree \cite{MosselRec}. The same program as in \cite{Mossel_SBM1} is likely to extend to $k=3$, at least for large degrees, as in this case it is shown in \cite{sly-potts} that it is impossible to detect below the KS threshold for the reconstruction problem on tree. For $k \ge 4$, the problem is much harder.
While \cite{decelle} provides evidences towards this conjecture, proving formally such a conjecture would require computational lower bounds that seem currently out of reach. Similar statements apply to related models such as planted clique or planted community  \cite{yash_clique,am_1comm,hajek_1comm} problems. 
To back up the current ``evidences'' on the impossibility of detecting efficiently below the KS threshold, one may rely on further reductions such as done in \cite{rigollet,chen-xu}, or subclasses of algorithms such as done in \cite{vempala_LB} with statistical query algorithms. Finding the exact expression of the IT threshold at all SNRs seem to be also an interesting problem to pursue. \\

\noindent
{\bf Accuracy.} As discussed in Section \ref{results}, since ABP solves detection as soon as $\snr >1$, it can be enhanced to achieve better accuracy using full BP. This is relatively straightforward for two communities. For more than two communities, this first requires converting the two sets that are correlated with their communities into a nontrivial assignment of a belief to each vertex. There are several ways one could do that, such as by making the following modifications to the algorithm.  
First, one would change the definition of $Y$ so that $Y$ is a $2|E(G)|\times m$ matrix such that for all $(v',v)\in E(G)$, $Y_{(v',v),t}=y^{(t)}_{(v',v)}$. 
Then, rather than dividing $G$'s vertices into those that have positive and negative values of $y'$, divide its directed edges into those that have positive and negative values of $y'$. 
Then use the frequencies of vertices with given numbers of edges with positive or negative values of $y'$, and edge densities between these vertices, to assign a probability distribution for the communities of the starting vertices of edges with given signs of $y'$. 
Finally, use these as the starting probabilities for a belief propagation algorithm of depth $\ln(n)/3\ln(\lambda_1)$.

For the general SBM, this requires some conditions and we make the following conjecture. 

\begin{conjecture}
Let $k \in \mZ_+$, $p\in (0,1)^k$, $Q$ be a $k\times k$ symmetric matrix with nonnegative entries and $G$ be drawn under $\sbm(n,p,Q/n)$. Furthermore, assume that for any two communities $i$ and $j$, there exists an eigenvector $w$ of $PQ$ with eigenvalue $\lambda_2$ such that $w_i\ne w_j$. Then there exist $c,r\in\mathbb{Z}^+$ and $m:\mathbb{Z}^+\rightarrow \mathbb{Z}^+$ such that with the above modifications, $\mathrm{ABP}(G,m(n),r,c,(\lambda_1,...,\lambda_h))$ classifies $G$'s vertices with an accuracy that is assymptotically at least as good as any other polynomial-time algorithm.
\end{conjecture}
If there exist communities $i$ and $j$ such that $w_i=w_j$ whenever $w$ is an eigenvector of $PQ$ with eigenvalue $\lambda_2$ then the original use of $ABP$ will not distinguish between these communities. We believe that one could classify vertices with optimal accuracy on such a graph by using the beliefs resulting from this algorithm as a starting point for another layer of $ABP$ and possibly going through several sucessive layers.\\

\noindent
{\bf Extensions.} Many variants of the SBM can be studied, such as the labelled block model \cite{airoldi,label_marc,jiaming}, the censored block model \cite{abbetoc,random,abbs,abbs-isit,rough,new-vu,florent_CBM}, the degree-corrected block model \cite{newman2}, overlapping block models \cite{fortunato} and more. While many of the fundamental challenges seem to be captured by the SBM already, these represent important extensions for applications.   Another important extension would be to tackle sublinear size communities, which is likely to raise new challenges (see planted clique for example).

\section{Proofs}
\subsection{Achieving the KS threshold}\label{proofs}
The proof works for the general model $\sbm(n,p,Q/n)$. Recall that $P=\diag(p)$, and $\lambda_1,...,\lambda_{h}$ are the distinct eigenvalues of $PQ$ in order of decreasing magnitude.

\begin{definition}
For any $r$ and $m$, an $r$-nonbacktracking walk of length $m$ is a sequence of vertices $v_0, v_1,..., v_m$ such that $v_i\ne v_j$ whenever $|i-j|\le r$ and $v_i$ is adjacent to $v_{i+1}$ for all $i$.
\end{definition}

\begin{definition}
Given $r,m>0$ and vertices $v$ and $v'$, let $W_{m[r]}(v,v')$ be the number of $r$-nonbacktracking walks of length $m$ from $v$ to $v'$.
\end{definition}

\begin{definition}
Given $r,m>0$, graph $G$, an assignment of a real number $x_v$ to every vertex $v$, a multiset of real numbers $S$, and a vertex $v$, let
\begin{align}W_{m/S[r]}(x,v)=\sum_{T\subseteq S} \prod_{y\in T} (-y)\sum_{v'\in G} x_{v'}W_{m-|T|[r]}(v',v). \label{def1} \end{align}
\end{definition}
In other words, $W_{m/\emptyset[r]}(x,v)$ is the sum over all $r$-nonbacktracking walks of length $m$ ending at $v$ of the values of $x$ at their initial vertices, and for $S\ne\emptyset$ and $y\in S$, we have that $W_{m/S[r]}(x,v)=W_{m/(S\backslash\{y\})[r]}(x,v)-y\cdot W_{m-1/(S\backslash\{y\})[r]}(x,v).$ Note that the ABP algorithm sets $Y_{v,t}$ equal to $W_{t/\emptyset[r]}(x,v)$ for all $0<t\le m$ and $v\in G$ in step $2b$. Then in step $2c$, it sets $y_v^{(m)}$ equal to $W_{m/S[r]}(x,v)$, where $S$ is the multiset containing $\lceil\frac{m-r-(2r+1)s'}{l}\rceil$ copies of $\lambda_{s'}$ for every $s'<s$.

Further intuition on \eqref{def1} will be provided with Definition \ref{def2}. The plan is to select the $x_v$ independently according to a Normal distribution, and then compute $W_{m/S[r]}(x,v)$ for appropriate $S$, $m$, and $r$. The assignment of $x$ will inevitably have slightly different average values in different communities, and under the right conditions, these differences will be amplified to the point of allowing differentiation between communities with asymptotically nonzero advantage.


\begin{proof}[Proof of Theorem \ref{main1}:]
When $ABP$ splits $\Gamma$ off of $G$, the remaining graph is still drawn from the SBM, albeit with connectivity $(1-\gamma)Q$. The formula for $\gamma$ ensures that if $\lambda_2^2>\lambda_1$ then $((1-\gamma)\lambda_2)^2>((1-\gamma)\lambda_1)$. Now, let $m''=\lfloor \sqrt{\log\log n}\rfloor+1$. For each $v\in G$ and $t>0$, let $N_t(v)$ be the set of all vertices $t$ edges away from $v$, $N_{\le t}(v)$ be the subgraph of $G$ induced by the vertices within $t$ edges of $v$, and $T_v=\{v': \exists v''\in N_{m''}(v) \mid (v',v'')\in \Gamma\}$. Unless the original graph has a cycle in $N_{\le m''}(v)\cup T_v$, we have that
\[y''_v=\sum_{v'\in T_v} y_{v'}^{(m)}\]
Now, let $w$ be an eigenvector of $PQ$ with eigenvalue $\lambda_i$ and magnitude $1$. Also, for every $v\in G$ and $t\ge 0$, let $W_t(v)=\sum_{v'\in N_t(v)} w_{\sigma_{v'}}/p_{\sigma_{v'}}$. Given a fixed value of $N_{\le t}(v)$, $v'\in N_t(v)$, a fixed value of $\sigma_{v'}$, and a community $j$, the expected number of vertices not in $N_{\le t}(v)$ that are in community $j$ and adjacent to $v'$ is $\left(1-\frac{|N_{\le t}(v)|}{n}\right)(1-\gamma)e_j\cdot PQ e_{\sigma_{v'}}$. So, we can show by induction on $t$ that for any $t\ge 0$, $n'>0$, and $z\in\mathbb{R}$, we have
\[E\left[W_{m''}(v) |W_{m''-t}(v)=z,|N_{m''-t}(v)|=n'\right]=(1-\gamma)^t\lambda_i^t z+O\left(\frac{((1-\gamma)\lambda_1)^{2t}(n')^2}{n}\right).\]
Similarly, for any $t\ge 0$, $0<n'\le \ln^{2m''-2t}(n)$, and $z\in\mathbb{R}$, we have
\[\Var\left[W_{m''}(v) |W_{m''-t}(v)=z,|N_{m''-t}(v)|=n'\right]=O\left(\sum_{t'=1}^t n'(1-\gamma)^{t+t'}\lambda_1^{t-t'} \lambda_i^{2t'}\right)\]

Now, given nonintersecting graphs $N'$ and $N''$ with $|N'|,|N''|=O(\ln^{2m''}(n))$, $v\in N'$, and $v'\in N''$, whether or not the subgraph of $G$ induced by $V(N')$ is $N'$ is independent of whether or not the subgraph of $G$ induced by $V(N'')$ is $N''$. Also, there are no edges between these two graphs with probability $1-O(|N'|\cdot |N''|/n)$ regardless of what form these graphs take. Furthermore, for any $v''\not\in N'\cup N''$, regardless of the value of $G\backslash\{v''\}$ there are no edges from $v''$ to $V(N')$ with probability $1-\lambda_1 |N'|/n+O(|N'|^2/n^2)$, there are no edges from $v''$ to $V(N'')$ with probability $1-\lambda_1 |N''|/n+O(|N''|^2/n^2)$, and there are no edges from $v''$ to either with probability $1-\lambda_1 (|N'|+|N''|)/n+O((|N'|+|N''|)^2/n^2)$. So, we have that 
\begin{align*}
&|P\left [N_{m''}(v)=N',N_{m''}(v')=N''\right]-P\left[N_{m''}(v)=N'\right]\cdot P\left[N_{m''}(v')=N''\right]|\\
&=O(|N'|\cdot |N''|/n)P\left [N_{m''}(v)=N',N_{m''}(v')=N''\right]
\end{align*}
Furthermore, the probability that $N_{m''}(v)$ and $N_{m''}(v')$ intersect is $O(\lambda_1^{2m''}/n)$. Also, $|W_{m''}(v)|\le n\cdot \max_i|w_i/p_i|$ for all $v$, and with probability $1-o(n^{-20})$, every vertex in $G$ has degree less than $\ln^2(n)$, so $W_{m''}(v)=O(\ln^{2m''}(n))$ for all $v$. So, for $v\ne v'$, the correlation between $W_{m''}(v)$ and $W_{m''}(v')$ is $o(1/\sqrt{n})$, and the correlation between  $W^2_{m''}(v)$ and $W^2_{m''}(v')$ is also $o(1/\sqrt{n})$. Therefore, with probability $1-o(1)$, the average value of $W_{m''}(v)$ over all $v$ in community $j$ is $(1-\gamma)^{m''}\lambda_i^{m''} (w_{j}/p_j+o(1))$ for all $j$. Also, there exists constant $\delta$ such that $\sum_{v\in G} (W_{m''}(v))^2\le n\delta\sum_{t=1}^{m''} (1-\gamma)^{m''+t}\lambda_1^{m''-t}\lambda_i^{2t}$ with probability $1-o(1)$.

On another note, for each $v$, we have that $y_v^{(m)}=W_{m/\{c_i\}}(x,v)$, so by lemma $\ref{varianceLemma}$, 
\[\Var[W_{m/\{c_i\}}(x,v)]=O\left(\prod_{j=1}^m((1-\gamma)\lambda_s-c_j)^2\right)\]
and the empirical variance of $\{y_v^{(m)}\}$ is $O\left(\prod_{j=1}^m((1-\gamma)\lambda_s-c_j)^2\right)$ with probability $1-o(1)$. Now, let $\overline{y}$ be the vector such that $\overline{y}_i=\sum_{v\in\Omega_i} y^{(m)}_v/n$ for all $i$. By Lemma \ref{communityAverage}, we know that with probability $1-o(1)$, the component of $\overline{y}$ on $PQ$'s eigenspace of eigenvalue $\lambda_s$ has magnitude $\Omega(\frac{1}{\sqrt n}\prod |(1-\gamma)\lambda_s-c_j|)$, while all of its components on $PQ$'s other eigenspaces have magnitudes of $O(\frac{1}{\log(n)\sqrt n}\prod |(1-\gamma)\lambda_s-c_j|)$.

Now, let $y'''_v=\sum_{v'\in T_v} y_v^{(m)}$ for all $v$, and recall that $y''_v=y'''_v$ unless the original graph has a cycle in $N_{\le m''}(v)\cup T_v$. Observe that
\begin{align*}
\sum_{v\in G} w_{\sigma_v}y'''_v/p_{\sigma_v}&=\sum_{v,v',v''\in G: v'\in N_{m''}(v), (v',v'')\in\Gamma} w_{\sigma_v} y_{v''}^{(m)}/p_{\sigma_v}\\
&= \sum_{(v',v'')\in\Gamma} W_{m''}(v')\cdot y_{v''}^{(m)}
\end{align*}
For each $v,v'\in G$ without an edge between them, 
\[P[(v,v')\in\Gamma|\sigma_v=j,\sigma_{v'}=j']=(1+O(1/n)) \gamma Q_{j,j'}/n\]

So, there exists $\delta_w>0$ such that for a fixed $G$, $\sigma$, and $x$, with probability $1-o(1)$ it will be the case that 
\[E\left[\sum_{v\in G} w_{\sigma_v}y'''_v/p_{\sigma_v}\right]=\gamma(1-\gamma)^{m''}\lambda_i^{m''+1} (w\cdot P^{-1}\overline{y}+o(1/\sqrt{n})) n\] 
and 
\begin{align*}
\Var\left[\sum_{v\in G} w_{\sigma_v}y'''_v/p_{\sigma_v}\right]&=(\delta_w/n) \sum_{v',v''} (W_{m''}(v')\cdot y_{v''}^{(m)})^2\\
&\le n \delta^2_w \sum_{t=1}^{m''} (1-\gamma)^{m''+t}\lambda_1^{m''-t}\lambda_i^{2t}\cdot\prod_{j=1}^m((1-\gamma)\lambda_s-c_j)^2\\
&=o\left(\left(\gamma(1-\gamma)^{m''}\lambda_s^{m''+1}\cdot \sqrt{n}\prod_{j=1}^m((1-\gamma)\lambda_s-c_j)^2\right)^2\right).
\end{align*}

This means that with probability $1-o(1)$, we have that 
\[\Var\left[\sum_{v\in G} w_{\sigma_v}y'''_v/p_{\sigma_v}\right]=o(\gamma^2(1-\gamma)^{2m''}\lambda_s^{2m''+2}||\overline{y}||_2^2 n^2).\]
 In particular, if we choose an orthonormal eigenbasis for $PQ$, $w_1,...,w_k$, then with probability $1-o(1)$ this will hold for all $w_i$ in the basis, which implies that with probability $1-o(1)$, we have that $\sum_{v\in\Omega_j} y'''_v/n=\gamma(1-\gamma)^{m''}((PQ)^{m''+1}\overline{y})_j+o(\gamma(1-\gamma)^{m''}\lambda_s^{m''+1} ||\overline{y}||_{2})$ for all $j\in [k]$. Also, for fixed $G$ and $x$ and with probability $1-o(1)$ we have

\begin{align*}
\sum_{v\in G} (y'''_v)^2&=\sum_{v,v_2,v_3,v_2',v_3'\in G: v_2,v_2'\in N_{m''}(v),(v_2,v_3)\in\Gamma,(v'_2,v'_3)\in\Gamma} y^{(m)}_{v_3}\cdot y^{(m)}_{v'_3}\\
&=\sum_{v,v_2,v_3,v_2',v_3'\in G: v_2,v_2'\in N_{m''}(v),(v_2,v_3)\ne(v'_2,v'_3)\in\Gamma} y^{(m)}_{v_3}\cdot y^{(m)}_{v'_3}\\
&+\sum_{v,v_2,v_3\in G: v_2\in N_{m''}(v),(v_2,v_3)\in\Gamma} (y^{(m)}_{v_3})^2\\
&=(1+o(1))\sum_{v\in G}\left( \sum_{v_2\in N_{m''}(v)} \gamma(Q\overline{y})_{\sigma_{v_2}}\right)^2\\
&+(1+o(1))\gamma(1-\gamma)^{m''}\lambda_1^{m''+1}\sum_{v\in G} (y_v^{(m)})^2\\
&\le (1+o(1))n\sum_{t=1}^{m''} \gamma^2 (1-\gamma)^{m''+t}\lambda_1^{m''-t}\lambda_s^{2t+2}||\overline{y}||_2^2/(\min p_i)+O(\gamma(1-\gamma)^{m''}\lambda_1^{m''+1}||\overline{y}||_2^2n)\\
&=(1+o(1)) n\cdot\frac{(1-\gamma)\lambda_s^2}{\min p_i[(1-\gamma)\lambda_s^2-\lambda_1]} \gamma^2 (1-\gamma)^{2m''}\lambda_s^{2m''+2}||\overline{y}||_2^2
\end{align*}
With probability $1-o(1)$, there are only $O(\ln(n))$ vertices $v\in G$ such that $y''_v\ne y'''_v$. With probability $1-o(1)$, no such vertex has more than $\ln^{2m''+2}(n)$ vertices within $m''+1$ edges of it and no such vertex has more than one cycle within $m''+1$ edges of it. Assuming this holds, these differences are small enough that the results above still hold if we substitute $y''_v$ for $y'''_v$. Among other things, this means that we can choose $\sigma_1$ and $\sigma_2$ such that the average value of $y''_v$ for $v\in\Omega_{\sigma_1}$ differs from the average value of $y''_v$ for $v\in\Omega_{\sigma_2}$ by at least $(1-o(1))\gamma(1-\gamma)^{m''}\lambda_s^{m''+1} ||\overline{y}||_2$.

Now, let $a$ be the difference between the average value of  $y''_v$ for $v\in\Omega_{\sigma_1}$ and the average value of  $y''_v$ for $v\in\Omega_{\sigma_2}$. Also, let $b$ be the average value of $(y''_v)^2$. There exists constant $\delta'$ such that $b\le \delta' a^2$ with probability $1-o(1)$. The average value of $1/2+y''_v/(C\sqrt{\sum (y''_{v'})^2/n})$ for $v\in\Omega_{\sigma_1}$ differs from its average value for $v\in\Omega_{\sigma_2}$ by $\frac{a}{C\sqrt{b}}$. For any $v$,
\[\min(|y''_v/(C\sqrt{\sum (y''_{v'})^2/n})|-1/2,0)\le 2(y''_v/(C\sqrt{b}))^2\]
So, the average value of $\min(|y''_v/(C\sqrt{\sum (y''_{v'})^2/n})|-1/2,0)$ for $v$ in any community is at most $2/(C^2\min p_i)$. So, the probability that a vertex from community $\sigma_1$ is assigned to group $2$ differs from the probability that a vertex from community $\sigma_2$ is assigned to group $2$ by at least $\frac{a}{C\sqrt{b}}-4/(C^2\min p_i)$. For a sufficiently large constant $C$, and small constant $\epsilon>0$ this will be greater than $\epsilon$ with probability $1-o(1)$, as desired.


ABP initializes in $O(n)$ time. Computing the $y^{(t)}$ takes $O(n\log n)$ expected time. The eigenvalue compensation step can be carried out in $O(n\log n)$ time if $\left(\prod_{s'<s} M_{s'}^{\lceil \frac{m-r-(2r+1)s'}{l}\rceil}\right)e_m$ is computed first and then $Y$ is multiplied by it, and computing $y''$ takes $o(n\log n)$ time. The assignment step also takes $O(n)$ time, so the entire algorithm can be run in $O(n\log n)$ time.
\end{proof}

\subsubsection{Preliminaries}
Our first step in proving the lemmas used above is to reexpress $W_{m/S[r]}(x,v)$ as a sum over $r$-nonbacktracking walks of an expression that we will eventually be able to bound. Then we will establish some properties of $r$-nonbacktracking walks that we will need.

\begin{definition}
For any $r\ge 1$ and series of vertices $v_0,v_1,...,v_m$, let $W_r((v_0,...,v_m))$ be $1$ if $v_0,...,v_m$ is an $r$-nonbacktracking walk, and $0$ otherwise.
\end{definition}

\begin{definition}\label{def2}
For any $r\ge 1$, series of vertices $v_0,...,v_m$, and series of real numbers $c_0,...,c_m$, let $W_{(c_0,...,c_m)[r]}((v_0,...,v_m))$ be the sum, over all subserieses $i_0,...,i_{m'}$ of $0,...,m$ of
\[\left(\prod_{i\not\in (i_0,...,i_{m'})} (-c_i/n)\right)\cdot W_r((v_{i_0},v_{i_1},...,v_{i_{m'}})).\]
\end{definition}
Note that for any $c_0,...,c_m$ consisting of the elements of $S$ and $0$'s with $c_0=c_m=0$,  $$W_{m/S[r]}(x,v)=\sum_{v_0,...,v_m\in G:v_m=v} [x_{v_0}\cdot W_{(c_0,...,c_m)[r]}((v_0,...,v_m))].$$

\begin{definition}
Given a series of vertices $v_0,...,v_m$, its walk graph is the graph consisting of every vertex in the series with an edge between each pair of vertices that are adjacent in the series. Note that the walk graph has a number of vertices equal to the number of distinct vertices in $v_0,...,v_m$, a number of edges equal to the cardinality of $\{\{v_i,v_{i+1}\}:0\le i<m\}$, and every vertex in it other than $v_0$ or $v_m$ must have degree at least $2$.
\end{definition}

\begin{lemma}
Consider the complete graph $K_n$. For any $m,b,w,e\ge 0$, there are at most $2n^{m-w+1}(m+1)^{4(w-e)}2^{2b(w-e)}$ walks of length $m$ in $K_n$ that have $m+1-w$ distinct vertices, $m-e$ distinct edges, no vertex repeated more than $b$ times, and no edge repeated twice in a row.
\end{lemma}

\begin{proof}
Choose such a walk, and let $H$ be its walk graph. Also, let the distinct vertices in the walk be $v_1,v_2,...,v_{m+1-w}$, in order of first appearance. There are at most $n^{m-w+1}$ possible choices of which vertices these are. For any $i$ such that $v_i$ is not adjacent to $v_{i+1}$ in the walk, the first edge in the walk after $v_i$ and the first edge leading to $v_{i+1}$ must both have not appeared previously in the walk. Otherwise, the edge between $v_i$ and $v_{i+1}$ must not have appeared previously. So, there are at most $m-e-(m-w)=w-e$ indices $i$ such that $v_{i+1}$ is not adjacent to $v_i$. There are at most $(m+1)^{w-e}$ choices of which indices have this property, and $m+1$ choices of which vertex the first edge leading to each such $v_{i+1}$ has on its other side. That accounts for $m-w$ of $H$'s edges, so there are at most $(m+1)^{2(w-e)}$ possibilities of what the others are. Thus, there are at most $n^{m-w+1}(m+1)^{4(w-e)}$ possible graphs $H$.

Now, let $v'_0,...,v'_m$ be the complete walk in question. Given the values of $v'_i$ and $v'_{i-1}$, the only possible values of $v'_{i+1}$ are the vertices that are adjacent to $v'_i$ other than $v'_{i-1}$. Furthermore, since no vertex appears in the walk more than $b$ times, that means that the number of possible walks for a given $H$ is at most
\begin{align*}
&2\prod_{v\in H:\mathrm{degree}(v)>1} (\mathrm{degree}(v)-1)^b\\
&\le 2\prod_{v\in H:\mathrm{degree}(v)>1} 2^{b(\mathrm{degree}(v)-2)}\\
&\le 2^{1+b\sum_{v\in H:\mathrm{degree}(v)>1} (\mathrm{degree}(v)-2)}\\
&\le 2^{1+2b(w-e)}.
\end{align*}
\end{proof}

\begin{definition}
Given a sequence of vertices $v_0,...,v_m$, a fresh segment is a consecutive subsequence $v_a,v_{a+1},...,v_b$ such that no interior vertex of the segment has degree greater than $2$ in the walk graph of $v_0,...,v_m$, no interior vertex of the segment is equal to $v_m$, and the walk $v_0,...,v_a$ does not traverse any of the segement's edges. Call the segment nonrepeated if none of its edges occur in the walk $v_b,...,v_m$ and repeated otherwise.
\end{definition}


\begin{lemma}[Walk decomposition lemma]
Let $v_0,...,v_m$ define a walk in $K_n$ that has $m+1-w$ distinct vertices and $m-e$ distinct edges. There exists a set of at most $3(w-e)+1$ fresh segments of the walk that are disjoint except at end vertices, and such that every edge in the walk is in one of these segments. Also, every end vertex of one of these segments is either $v_0$, $v_m$, or a vertex that is repeated in the walk.
\end{lemma}

\begin{proof}
Let $G'$ be the walk graph of $v_1,...,v_m$. Let $S$ consist of $v_1,v_m$, and all vertices in $G'$ of degree greater than $2$. The total degree of all vertices in $S$ is at most $6(w-e)+2$, so $G'$ consists of these vertices and at most $3(w-e)+1$ paths between them that do not contain any vertex from $S$. Once the walk goes onto one of these paths, the facts that it is nonbacktracking, $v_m$ is not in the path, and no vertex on the path has degree greater than $2$ forces it to continue to the end of the path. The walk must traverse each of these paths for the first time at some point. Therefore, the set of serieses of vertices corresponding to the first traversals of these paths is the desired set of fresh segments.
\end{proof}

\begin{definition}
Given a sequence of vertices $v_0,...,v_m$, let its standard decomposition be the collection of fresh segments generated by the above construction for $v_0,...,v_m$.
\end{definition}

\begin{definition}
Given a collection of fresh segments, let its measures be $d,d',t,t'$, where $d$ is the number of nonrepeated fresh segments in the collection, $t$ is the sum of their lengths, $d'$ is the number of repeated fresh segments in the collection, and $t'$ is the sum of their lengths. Also, let the measures of a series of vertices denote the measures of its standard decomposition.
\end{definition}

\subsubsection{The shard decomposition}

The next step in the proof is to establish an upper bound on $|E[W_{(c_0,...,c_m)[r]}((v_0,...,v_m))]|$ when $(c_0,...,c_m),(v_0,...,v_m)$ satisfies some special properties. Then we will show that $W_{(c_0,...,c_m)[r]}((v_0,...,v_m))$ can always be expressed as a linear combination of shards, expressions of the form $W_{(c'_0,...,c'_{m'})[r]}((v'_0,...,v'_{m'}))$ that have the aforementioned properties. After that, we will show that if $(c_0,...,c_m)$ is chosen correctly, the magnitudes of the expected values of many of the shards of $W_{(c_0,...,c_m)[r]}((v_0,...,v_m))$ must be fairly small. 

\begin{lemma}
Let $r\ge 1$ and $c_0,...,c_m$ be a series of real numbers with $c_0=c_m=0$. Also let $v_0,...,v_m$ be a series of vertices with standard decomposition $v_{a_1},...,v_{b_1}$, $v_{a_2},...,v_{b_2}$,...,$v_{a_d},...,v_{b_d}$,  ordered so that $b_i\le a_{i+1}$ for each $i$. Assume that for any $i,j$ such that $v_i$ occurs elsewhere in the series and $|i-j|\le r$, $c_j=0$. Then, 
\[|\E [W_{(c_0,...,c_m)[r]}((v_0,...,v_m))]|\le n^{-\sum_{i=1}^d b_i-a_i}(\min p_i)^{-d}\prod_{i=1}^d\sum_{j=1}^k\prod_{i'=a_i}^{b_i-1} |\lambda_j-c_{i'}|.\]
\end{lemma}

\begin{proof}
For any $i$ such that $c_i\ne 0$, $i$ is at least $r$ away from every element of the series that is repeated. So, deleting $v_i$ for all $i$ in any subset of $\{i: c_i\ne 0\}$ has no effect on whether or not $v_1,...,v_m$ is $r$-nonbacktracking. If $v_1,...v_m$ is not $r$-nonbacktracking, then the expected value of $W_{(c_0,...,c_m)[r]}((v_0,...,v_m))$ is $0$, and the conclusion holds. Now, consider the case where $v_1,...,v_m$ is $r$-nonbacktracking. Since any subsequence of $v_0,...,v_m$ resulting from deleting $v_i$ for which $c_i\ne 0$ is also $r$-nonbacktracking, the restriction that a vertex cannot repeat within $r$ steps is irrelevant and $W_{(c_0,...,c_m)[r]}((v_0,...,v_m))=W_{(c_0,...,c_m)[0]}((v_0,...,v_m))$. Also, $v_{a_i}$ and $v_{b_i}$ are all repeated vertices except possibly $v_0$ and $v_m$. So, $c_{a_i}= 0$ and $c_{b_i}=0$ for all $i$. That implies that
\begin{align*}
W_{(c_0,...,c_m)[0]}((v_0,...,v_m))&=W_{(c_0,...,c_{a_1})[0]}((v_0,...,v_{a_1}))\cdot W_{(c_{a_1},...,c_{b_1})[0]}((v_{a_1},...,v_{b_1}))\\
&\cdot  W_{(c_{b_1},...,c_{a_2})[0]}((v_{b_1},...,v_{a_2}))\cdot  W_{(c_{a_2},...,c_{b_2})[0]}((v_{a_2},...,v_{b_2}))\\
&\cdot \ldots \cdot  W_{(c_{b_d},...,c_{m})[0]}((v_{b_d},...,v_{m})).
\end{align*}
If $v_i$ is not part of one of the fresh segments, then $v_i$ is a repetition of a vertex that is in one of them, so $c_i=0$. Thus,
\[W_{(c_0,...,c_{a_1})[0]}((v_0,...,v_{a_1}))\cdot  W_{(c_{b_1},...,c_{a_2})[0]}((v_{b_1},...,v_{a_2}))\cdot...\cdot W_{(c_{b_d},...,c_{m})[0]}((v_{b_d},...,v_{m}))\] is either $0$ or $1$. If it is $0$ then there is some $v_i,v_{i+1}$ that are not part of any of the fresh segements and do not have an edge between them. By the previous lemma, there must exist $1\le i\le d$ and $a_i\le j\le b_i-1$ such that $\{v_i,v_{i+1}\}=\{v_j,v_{j+1}\}$, and since the vertices are repeated, $c_j=c_{j+1}=0$. So, the lack of an edge between them means that $W_{(c_{a_i},...,c_{b_i})[0]}((v_{a_i},...,v_{b_i}))=0$. Either way,
\[W_{(c_0,...,c_m)[0]}((v_0,...,v_m))=\prod_{i=1}^d W_{(c_{a_i},...,c_{b_i})[0]}((v_{a_i},...,v_{b_i})).\]
The fresh segments in a standard decomposition only intersect at their endpoints, so for a fixed assignment of communities to the endpoints,
\begin{align*}
|\E [W_{(c_0,...,c_m)[r]}((v_0,...,v_m))]|&=\prod_{i=1}^d |\E [W_{(c_{a_i},...,c_{b_i})[0]}((v_{a_i},...,v_{b_i}))]|\\
&=n^{-\sum b_i-a_i}\prod_{i=1}^d |e_{\sigma_{v_{a_i}}}P^{-1}\prod_{i'=a_i}^{b_i-1} (PQ-c_{i'}I) e_{\sigma_{v_{b_i}}}|\\
&\le n^{-\sum b_i-a_i}(\min p_i)^{-d}\prod_{i=1}^d\sum_{j=1}^k \prod_{i'=a_i}^{b_i-1} |\lambda_j-c_{i'}|.
\end{align*}
\end{proof}

Of course, for general $c_0,...,c_m$ and $v_0,...,v_m$ the condition that $c_i=0$ whenever there is a repeated vertex $v_j$ with $|i-j|\le r$ may not be satisfied. We can deal with that using the fact that for arbitrary $c_0,...,c_m$, $v_0,...,v_m$, $r$, and $i$,
\begin{align*}
&W_{(c_0,...,c_{i-1},c_i,c_{i+1},...,c_m)[r]}((v_0,...,v_{i-1},v_i,v_{i+1},...,v_m))\\
&=W_{(c_0,...,c_{i-1},0,c_{i+1},...,c_m)[r]}((v_0,...,v_{i-1},v_i,v_{i+1},...,v_m))\\
&-\frac{c_i}{n}W_{(c_0,...,c_{i-1},c_{i+1},...,c_m)[r]}((v_0,...,v_{i-1},v_{i+1},...,v_m)).
\end{align*}
So, for any $c_0,...,c_m$ with $c_0=c_m=0$, $v_0,...,v_m$, and $r$, we can express $W_{(c_0,...,c_m)[r]}((v_0,...,v_m))$ as a linear combination of expressions that the above lemma applies to by means of the following algorithm.

\begin{algorithm}
path-sum-conversion($W_{(c_0,...,c_m)[r]}((v_0,...,v_m)))$:
\begin{enumerate}
\item If there exist $i$ and $j$ such that $v_i=v_j$, $c_i\ne 0$, and $c_j= 0$, return
\begin{align*}
&\text{path-sum-conversion}(W_{(c_0,...,c_{i-1},0,c_{i+1},...,c_m)[r]}((v_0,...,v_{i-1},v_i,v_{i+1},...,v_m)))\\
&-\frac{c_i}{n}\text{path-sum-conversion}(W_{(c_0,...,c_{i-1},c_{i+1},...,c_m)[r]}((v_0,...,v_{i-1},v_{i+1},...,v_m)))
\end{align*}
\item Otherwise, if there exist $j\ne j'$ such that $v_j=v_{j'}$, and $i$ such that $0<|i-j|\le r$ and $c_i\ne0$, return
\begin{align*}
&\text{path-sum-conversion}(W_{(c_0,...,c_{i-1},0,c_{i+1},...,c_m)[r]}((v_0,...,v_{i-1},v_i,v_{i+1},...,v_m)))\\
&-\frac{c_i}{n}\text{path-sum-conversion}(W_{(c_0,...,c_{i-1},c_{i+1},...,c_m)[r]}((v_0,...,v_{i-1},v_{i+1},...,v_m)))
\end{align*}
\item Otherwise, if there exist $i\ne j$ such that $v_i=v_j$, $c_i\ne 0$, and $c_j\ne 0$, return
\begin{align*}
&\text{path-sum-conversion}(W_{(c_0,...,c_{i-1},0,c_{i+1},...,c_m)[r]}((v_0,...,v_{i-1},v_i,v_{i+1},...,v_m)))\\
&-\frac{c_i}{n}\text{path-sum-conversion}(W_{(c_0,...,c_{i-1},c_{i+1},...,c_m)[r]}((v_0,...,v_{i-1},v_{i+1},...,v_m)))
\end{align*}
\item Otherwise, if there exist $i\ne j$ such that $|i-j|\le r$, and $v_i=v_j$, return $0$.

\item Otherwise, return $W_{(c_0,...,c_m)[r]}((v_0,...,v_m)))$
\end{enumerate}
\end{algorithm}
If there is more than one $i$ satisfying the conditions of the case under consideration, the algorithm should choose one according to some rule, such as always using the smallest. Also, note that this algorithm exists as a proof technique allowing us to replace any expression of the form $W_{(c_0,...,c_m)[r]}((v_0,...,v_m)))$ with a sum of expressions of the form $W_{(c'_0,...,c'_{m'})[r]}((v'_0,...,v'_{m'})))$ that the previous lemma applies to. We would never actually run it.


\begin{lemma}
Path-sum-conversion($W_{(c_0,...,c_m)[r]}((v_0,...,v_m)))$ always terminates. Furthermore, if there are no $i,j$ such that $c_i\ne 0$, $c_j\ne 0$, and $|i-j|\le r$, then for any $W_{(c'_0,...,c'_{m'})[r]}((v'_0,...,v'_{m'}))$ that appears in the expression it outputs and any $v_i$ that was deleted in the process of going from $W_{(c_0,...,c_m)[r]}((v_0,...,v_m)))$ to $W_{(c'_0,...,c'_{m'})[r]}((v'_0,...,v'_{m'}))$, at least one of the following holds. There exists $i'$ such that $v'_{i'}=v_i$, or there exists $j$ such that $|i-j|\le r$, $v_j$ was not deleted in the process of going from $W_{(c_0,...,c_m)[r]}((v_0,...,v_m)))$ to $W_{(c'_0,...,c'_{m'})[r]}((v'_0,...,v'_{m'}))$, and $v_j$ appears more than once in the list $(v'_0,...,v'_{m'})$. Also, if $c_0=c_m=0$, then for any $W_{(c'_0,...,c'_{m'})[r]}((v'_0,...,v'_{m'}))$ that appears in the expression it outputs, $c'_0=c'_{m'}=0$.
\end{lemma}

\begin{proof}
The fact that the algorithm terminates follows immediately from the fact that when it recurses, each iteration of the algorithm has one fewer nonzero element of $(c_0,...,c_m)$ than its parent iteration, so it goes down at most $m$ layers. Now, assume that there are no $i,j$ such that $c_i\ne 0$, $c_j\ne 0$, and $|i-j|\le r$. The algorithm only ever deletes $c_i$ that have nonzero values, so this will continue to hold for any $(c'_0,...,c'_{m'})$ that $(c_0,...,c_m)$ is converted to. Likewise, if $c_0=c_m=0$, then the first and last elements of any $(c'_0,...,c'_{m'})$ that $(c_0,...,c_m)$ is converted to will always be $0$. Also, note that if $c_i=0$, then $v_i$ will never be deleted. 

Now, if the algorithm deletes $v_i$ as an instance of case $1$, then there exists $j$ such that $v_j=v_i$ and $c_j= 0$. $v_j$ will never be deleted, and will thus be in any expression resulting from this deletion. If the algorithm deletes $v_i$ as an instance of case $2$, then there exist $j\ne j'$ with $v_j=v_{j'}$ and $0<|i-j|\le r$. We know that $c_j=0$ because $c_i\ne 0$ and $i$ and $j$ are within $r$ of each other. $c_{j'}=0$ as well, because if it did not, $i=j',j=j$ would have satisfied the conditions for case $1$. So, neither $v_j$ nor $v_{j'}$ will ever be deleted, and thus $v_j$ will appear more than once in the paths for any expression resulting from this deletion. If the algorithm deletes $v_i$ as an instance of case $3$, then it must be the case that for every $i'$ such that $v_i=v_{i'}$, there do not exist $j\ne j'$ such that $0<|i'-j|\le r$ and $v_j=v_{j'}$. There is no way for any operation the algorithm performs to result in such $j$ and $j'$ coming into existence, so if at any future point there is only one element of $(v'_0,...,v'_{m'})$ that is equal to $v_i$, there is no case under which that element will be deleted. So, any $(v'_0,...,v'_{m'})$ resulting from this deletion will contain at least one element equal to $v_i$.
\end{proof}

\begin{definition}
$W_{(c'_0,...,c'_{m'})[r]}((v'_0,...,v'_{m'}))$ is a level $x$ shard of $W_{(c_0,...,c_{m})[r]}((v_0,...,v_{m}))$ if $W_{(c'_0,...,c'_{m'})[r]}((v'_0,...,v'_{m'}))$ is one of the expressions in the sum resulting from path-sum-conversion$(W_{(c_0,...,c_{m})[r]}((v_0,...,v_{m})))$ and exactly $x$ of the vertices that were deleted in the process of converting $W_{(c_0,...,c_{m})[r]}((v_0,...,v_{m}))$ to $W_{(c'_0,...,c'_{m'})[r]}((v'_0,...,v'_{m'}))$ are repetitions of vertices in $(v'_0,...,v'_{m'})$. For a given $W_{(c_0,...,c_{m})[r]}((v_0,...,v_{m}))$, $\text{Shar}(W_{(c_0,...,c_{m})[r]}((v_0,...,v_{m})))$ denotes the set of all of its shards, and  $\text{Shar}_x(W_{(c_0,...,c_{m})[r]}((v_0,...,v_{m})))$ denotes the set of its level $x$ shards.
\end{definition}

\begin{definition}
Let an eigenvalue approximation be an ordered tuple $\lambda'_1,...,\lambda'_{s-1}$ such that $|\lambda_i-\lambda'_i|\le |\lambda_i-\lambda'_{i'}|$ for all $i$ and $i'$. Call such a tuple $\Lambda$-bounded if $|\lambda'_i|\le \Lambda$ and $|\lambda_i|\le\Lambda$ for all $i$. Also, let the error of such an approximation be $\max_i |\lambda_i-\lambda'_i|$.
\end{definition}

\begin{definition}
Given integers $r,l$, a tuple $\lambda'_1,...,\lambda'_{s-1}$, and a sequence $c_0,...,c_m$, let the sequence be an $l[r]$-cycle of $(\lambda'_1,...,\lambda'_{s-1})$ if and only if $(2r+1)(s-1)\le l$, $c_i=\lambda'_j$ for all $i$ and $j$ such that $i<m-r$ and the remainder of $i$ when divided by $l$ is $(2r+1)j$, and $c_i=0$ whenever there is no $j$ such that the above holds. Note that the first and last elements of an $l[r]$ cycle are always $0$, and that if $(2r+1)(s-1)\le l$ there exists a unique $l[r]$-cycle of length $m$ for every $m$.
\end{definition}

\begin{lemma}
Let $\lambda'_1,...,\lambda'_{s-1}$ be a $\Lambda$-bounded eigenvalue approximation and $c_0,...,c_m$ be an $l[r]$-cycle of this approximation for some $l,r>0$. Then, let $v_0,...,v_m$ be a series of vertices, and let $W_{(c'_0,...,c'_{m'})[r]}((v'_0,...,v'_{m'}))$ be a level $x$ shard of $W_{(c_0,...,c_m)[r]}((v_0,...,v_m)))$. Finally, let $d,d',t,t'$ be the measures of $(v'_0,...,v'_{m'})$ and $t''=\lceil \frac{t-d(l+2r+1)}{l}\rceil$. 
Then
\begin{align*}
|&\E (W_{(c'_0,...,c'_{m'})[r]}((v'_0,...,v'_{m'}))|\\
&\le n^{-t-t'}k^{d+d'}(\min p_i)^{-d-d'}\lambda_1^{t'+\max(t-(2x+d)(s-1)-l\cdot\max(t''-2x,0),0)}\\
&\cdot (\lambda_1+\Lambda)^{t-\max(t-(2x+d)(s-1)-l\cdot\max(t''-2x,0),0)-l\cdot\max(t''-2x,0)}\\
& \cdot
\left(\max_{1\le j\le k} |\lambda_j|^{l-s+1}\prod_{i=1}^{s-1} |\lambda_j-\lambda'_i|\right)^{\max(t''-2x,0)}.
\end{align*}
\end{lemma}

\begin{proof}
 Let $v'_{a_1},...,v'_{b_1}$, $v'_{a_2},...,v'_{b_2}$,...,$v'_{a_d},...,v'_{b_d}$ be the nonrepeated fresh segments in the standard decomposition of $(v'_0,...,v'_{m'})$, and $v'_{a'_1},...,v'_{b'_1}$, $v'_{a'_2},...,v'_{b'_2}$,...,$v'_{a'_{d'}},...,v'_{b'_{d'}}$ be the repeated fresh segments in its standard decomposition. For arbitrary $1\le i\le d'$ and $a'_i\le i'\le b'_i$, we have that $v'_{i'}$ either is a repeated vertex or is within one step of one. So, $c'_{i'}=0$. We know from a previous lemma that 
\begin{align*}
&|\E [W_{(c'_0,...,c'_{m'})[r]}((v'_0,...,v'_{m'}))]|\\
&\le n^{-t-t'}(\min p_i)^{-d-d'}\left(\prod_{i=1}^d\sum_{j=1}^k\prod_{i'=a_i}^{b_i-1} |\lambda_j-c'_{i'}|\right)\left(\prod_{i=1}^{d'}\sum_{j=1}^k\prod_{i'=a'_i}^{b'_i-1} |\lambda_j-c'_{i'}|\right)\\
&\le  n^{-t-t'}k^d(\min p_i)^{-d-d'}\left(\prod_{i=1}^d\max_{1\le j\le k}\prod_{i'=a_i}^{b_i-1} |\lambda_j-c'_{i'}|\right)\left(\prod_{i=1}^{d'}\sum_{j=1}^k\prod_{i'=a'_i}^{b'_i-1} |\lambda_j|\right)\\
&\le  n^{-t-t'}k^{d+d'}(\min p_i)^{-d-d'}\left(\prod_{i=1}^d \max_{1\le j\le k}\prod_{i'=a_i}^{b_i-1} |\lambda_j-c'_{i'}|\right)\left(\prod_{i=1}^{d'}\prod_{i'=a'_i}^{b'_i-1} |\lambda_1|\right)\\
&\le  n^{-t-t'}k^{d+d'}(\min p_i)^{-d-d'}\lambda_1^{t'}\left(\prod_{i=1}^d \max_{1\le j\le k}\prod_{i'=a_i}^{b_i-1} |\lambda_j-c'_{i'}|\right).
\end{align*}

Now, for each $0\le i\le m'$, define $f(i)$ such that $v_{f(i)}$ corresponds to $v'_i$. Also, for each $0\le i\le m$, define $g(i)$ so that $v'_{g(i)}$ corresponds to $v_i$, or $v_{i+1}$ if $v_i$ has been deleted. Also, for each $1\le i\le d$, let $y_i$ be the largest integer such that $g(f(a_i+r+1)+l\cdot y_i)\le b_i-r$, or $0$ if $a_i+2r\ge b_i$. Note that for any $i'\ge i$, $i'-i\ge g(i')-g(i)$. As a result, 
\begin{align*}
\sum_{i=1}^d y_i&=\frac{1}{l}\sum_{i=1}^d f(a_i+r+1)+l\cdot y_i-f(a_i+r+1)\\
&\ge \frac{1}{l}\sum_{i=1}^d g(f(a_i+r+1)+l\cdot y_i)-g(f(a_i+r+1))\\
&\ge\frac{1}{l}\sum_{i=1}^d b_i-r-l-(a_i+r+1)=\frac{t-d(2r+l+1)}{l},
\end{align*}
because $|g(x)-g(x')|\le |x-x'|$ for all $x$ and $x'$ and $g(f(a_i+r+1)+l\cdot(y_i+1))>b_i-r$. So, $\sum_{i=1}^d y_i\ge t''$.

For every $i$, either $a_i=0$ or $v'_{a_i}$ is repeated somewhere else in $(v'_0,...,v'_{m'})$. Likewise, either $b_i=m'$ or $v'_{b_i}$ is repeated somewhere else in $(v'_0,...,v'_m)$. The first $r+1$ and last $r+1$ of the $c_i$ are $0$, and whenever there exists $j$ such that $|i-j|\le r$ and $v'_j$ is a repeated vertex $c'_i$ is $0$, so $c'_j=0$ whenever $a_i\le j\le a_i+r$ or $b_i-r\le j\le b_i$ for some $i$. Also, for any $1\le i\le d$ and $1\le j\le y_i$, it is the case that $g(f(a_i+r+1)+l\cdot j)-g(f(a_i+r+1)+l\cdot (j-1))=l$ unless one of the vertices $v_{f(a_i+r+1)+l\cdot( j-1)},...,v_{f(a_i+r+1)+l\cdot j-1}$ was deleted and $\{c'_{g(f(a_i+r+1)+l\cdot (j-1))},...,c'_{g(f(a_i+r+1)+l\cdot j-1)}\}$ consists of one copy of $\lambda'_{h''}$ for each $h''$ and $l-(s-1)$ copies of $0$ unless one of the vertices $v_{f(a_i+r+1)+l\cdot( j-1)},...,v_{f(a_i+r+1)+l\cdot j-1}$ was deleted or one of $c_{f(a_i+r+1)+l\cdot( j-1)},...,c_{f(a_i+r+1)+l\cdot j-1}$ was set to $0$ as a result of the corresponding vertex being within distance $r$ of a vertex that was repeated somewhere else in the walk. The entirety of the block is at least $r$ away from any vertex that is repeated in $(v'_0,...v'_{m'})$, so the second case is only possible if the other copy of that vertex was subsequently deleted. So, at most $x$ blocks fall under each of these two cases, which means that there are at least $\max(t''-2x,0)$ uncorrupted blocks.  Each corrupted block has at most one $c'_i$ with a value of $\lambda'_{h''}$ for each $h''$ and there is at most one $c'_i$ with a value of $\lambda_{h''}$ between $c'_{g(f(a_i+r+1)+l\cdot y_i)}$ and $c'_{b_i}$ for each $i$ and $h''$. So, 
\begin{align*}
&\prod_{i=1}^d \max_{1\le j\le k}\prod_{i'=a_i}^{b_i-1} |\lambda_j-c'_{i'}|\\
&\le \left(\prod_{1\le i \le d:b_i-a_i\le 2r} \max_{1\le j\le k}\prod_{i'=a_i}^{b_i-1} |\lambda_j-c'_{i'}|\right)\\
&\cdot \prod_{1\le i \le d: \atop{b_i-a_i> 2r}}\max_{1\le j\le k} \prod_{i'=a_i}^{a_i+r}|\lambda_j-c'_{i'}|\prod_{i'=1}^{y_i}\prod_{i''=g(f(a_i+r+1)+l\cdot (i'-1))}^{g(f(a_i+r+1)+l\cdot i')-1}|\lambda_j-c'_{i'}|\prod_{i'=g(f(a_i+r+1)+l\cdot y_i)}^{b_i-1}|\lambda_j-c'_{i'}| \\
&\le \left(\max_{1\le j\le k} |\lambda_j|^{l-s+1}\prod_{i=1}^{s-1} |\lambda_j-\lambda'_i|\right)^{\max(t''-2x,0)}\lambda_1^{\max(t-(2x+d)(s-1)-l\cdot\max(t''-2x,0),0)}\\
&\cdot (\lambda_1+\Lambda)^{t-\max(t-(2x+d)(s-1)-l\cdot\max(t''-2x,0),0)-l\cdot\max(t''-2x,0)}.
\end{align*}
\end{proof}

\begin{lemma}
Let $c_0,...,c_m$ be a sequence of real numbers with at most $y$ nonzero elements, and $v'_0,...v'_{m'}$ be vertices. For any integer $x$, there are at most $3^y(m+1)^xn^{m-m'-x}$ pairs of a sequence of vertices $v_1,...,v_m$ and a sequence of real numbers $c'_0,...,c'_{m'}$ such that $W_{(c'_0,...,c'_{m'})[r]}((v'_0,...,v'_{m'}))$ is a level $x$ shard of $W_{(c_0,...,c_{m})[r]}((v_0,...,v_{m}))$.
\end{lemma}

\begin{proof}
Path-sum-conversion only deletes vertices that have a nonzero coresponding element of $c_0,...,c_m$, so there are at most $3^y$ possibilities for where the vertices that were deleted originally were, and which of them were copies of vertices in $v'_0,...,v'_{m'}$. There are at most $(m+1)^x$ possibilities for the identities of the $x$ deleted vertices that are repetitions of vertices in $(v'_0,...,v'_{m'})$ and $n^{m-m'-x}$ possiblilities for the remaining $m-m'-x$ deleted vertices. Furthermore, once the deleted vertices and their original locations are specified, the only sequence of vertices that could give rise to $v'_0,...,v'_{m'}$ after that sequence of deletions is the sequence formed by splicing the deleted vertices into $v'_0,...,v'_{m'}$ at the specified locations. Also, whenever path-sum-conversion recurses, the vertex that the current iteration was considering is deleted in one child iteration, and it is not deleted in any of the expressions resulting from the other branch. So, only one of the expressions in the sum output by $path-sum-conversion(W_{(c_0,...,c_{m})[r]}((v_0,...,v_{m})))$ results from the specified series of deletions, which means that the list of deleted vertices and their original locations implies a unique value of $(c'_0,...,c'_{m'})$.
\end{proof}

\begin{lemma}
Let $\lambda'_1,...,\lambda'_{s-1}$ be a $\Lambda$-bounded eigenvalue approximation and $c_0,...,c_m$ be an $l[r]$-cycle of this approximation for some $l$ and $r>0$. Then, let $v_0,...,v_m$ be a series of vertices, and let $W_{(c'_0,...,c'_{m'})[r]}((v'_0,...,v'_{m'}))$ be a level $x$ shard of $W_{(c_0,...,c_m)[r]}((v_0,...,v_m)))$ such that $v'_0,...,v'_{m'}$ is nonbacktracking. The absolute value of the coefficient of  $W_{(c'_0,...,c'_{m'})[r]}((v'_0,...,v'_{m'}))$ in the sum resulting from path-sum-conversion($W_{(c_0,...,c_{m})[r]}((v_0,...,v_{m}))$) is at most $(\Lambda/n)^{m-m'}$.
\end{lemma}

\begin{proof}
First, note that since $r\ge 1$, for any $i<i'$ such that $c_i\ne 0$ and $c_{i'}\ne 0$, it must be the case that $i'-i\ge 3$. Now, let $I$ and $I'$ be two distinct subsets of $\{i: c_i\ne 0\}$ with $m-m'$ elements, and choose the smallest $i\in (I\backslash I')\cup (I'\backslash I)$. Assume without loss of generality that $i\in I$. Note that $m\ge i+3$ because $I'/I\ne \emptyset$, and the smallest $i'>i$ such that $c_{i'}\ne 0$ is at least $i+3$. Now, let $t=|I\cap \{0,...,i-1\}|=|I'\cap \{0,...,i-1\}|$. Deleting all vertices from $v_0,...,v_m$ that have indices in $I$ leaves $v_{i+2}$ in the $(i+1-t)$th position, while deleting the vertices that have indices in $I'$ leaves $v_{i+1}$ in that position. So, if $v_{i+1}\ne v_{i+2}$ the resulting sequences are different and if $v_{i+1}=v_{i+2}$ the resulting sequences backtrack. Thus, there is only one subseries of vertices that can be deleted from $v_0,...,v_m$ to yield $v'_0,...,v'_{m'}$. That means that there is only one route through which path-sum-conversion($W_{(c_0,...,c_{m})[r]}((v_0,...,v_{m}))$) arrives at $W_{(c'_0,...,c'_{m'})[r]}((v'_0,...,v'_{m'}))$. Every time the algorithm deletes a vertex, it multiplies the expression's coefficient by $-c_i/n$, where $i$ is that vertex's original index in $v_0,...,v_m$. So, the conclusion holds because $|-c_i/n|\le \Lambda/n$ for all $i$.
\end{proof}

\subsubsection{Bounding the variance of $W_{m/S}$}
The previous section provides the key techniques we will need to prove bounds. Now we need to apply them. We need to prove that expression involving $W_{m/S}$ are within a certain range with high probability rather than merely computing their expected values, which requires us to bound their variances. That, in turn, will require us to bound the expected values of expressions of the form $W_{(c_0,...,c_m)[r]}((v_0,...,v_m))\cdot W_{(c_0,...,c_m)[r]}((v''_0,...,v''_m))$. In order to convert that to something more familiar, let $u_1,...,u_r$ be some new vertices that are adjacent to every vertex in $r$, and then note that for any $c_0,...,c_{m_1},c''_0,...,c''_{m_2}\in \mathbb{R}$ and $v_0,...,v_{m_1},v''_0,...,v''_{m_2}\in G$, 
\begin{align*}
&W_{(c_0,...,c_{m_1})[r]}((v_0,...,v_{m_1}))\cdot W_{(c''_0,...,c''_m)[r]}((v''_0,...,v''_{m_2}))\\
&=W_{(c_0,...,c_{m_1},0,...,0,c''_{m_2},c''_{m_2-1},...,c''_0)[r]}((v_0,...,v_{m_1},u_1,u_2,...,u_r,v''_{m_2},v''_{m_2-1},...,v''_0))
\end{align*}
because all of the vertices are connected to the $u_i$ and they create enough distance between the $v_i$ and the $v''_i$ that they will never be within $r$ of each other. Next we will establish a series of bounds on the expected values of these expressions, starting with the following.

\begin{lemma}\label{lem10}
Let $\lambda'_1,...,\lambda'_{s-1}$ be a $\Lambda$-bounded eigenvalue approximation with error at most 
$2\Lambda(\min_i |\lambda_s-\lambda_i|/4\Lambda)^{s-1}|\lambda_s/\Lambda|^{l-s+1}$ and $c_0,...,c_{m_1}$ and $c''_0,...,c''_{m_2}$ be $l[r]$-cycles of this approximation for some $l$ and $r$. Then, let $(c'''_0,...,c'''_{m_1+m_2+r+1})=(c_0,...,c_{m_1},0,...,0,c''_{m_2},c''_{m_2-1},...,c''_0)$, where there are $r$ $0$s in the middle. Next, let $v_0,...,v_{m_1},v''_0,...,v''_{m_2}\in G$ and $W_{(c'_0,...,c'_{m'}[r])}((v'_0,...,v'_{m'}))$ be a level $x$ shard of $W_{(c'''_0,...,c'''_{m_1+m_2+r+1})[r]}((v_0,...,v_m,u_1,u_2,...,u_r,v''_m,v''_{m-1},...,v''_0))$. Define $w$ and $e$ such that there are $m'+1-w$ distinct vertices in $v'_0,...,v'_{m'}$ and $m'-e$ distinct unordered pairs $\{v'_i,v'_{i+1}\}$, and let $m''=m'-r-1$. Then $m_1+m_2-m''\le x+6(w-e)+2+e/r$. Also, if $(\Lambda/|\lambda_s|)^{l-s+1}\ge 2^{s-1}$ and $\lambda_s^2>\Lambda$, then
\begin{align*}
|&\E (W_{(c'_0,...,c'_{m'})[r]}((v'_0,...,v'_{m'}))|\\
&\le n^{e-m''}(2^{s-1}k/\min p_i)^{3w-3e+2}|\lambda_s|^{m''}\left(2\Lambda/|\lambda_s|\right)^{(s-1)m''/l}\\
&\cdot (\Lambda/\lambda_s^2)^e \left((\Lambda/|\lambda_s|)^{l-s+1}2^{1-s}\right)^{\frac{(3w-3e+2)(l+2r+1)}{l}}(\Lambda/|\lambda_s|)^{2(l-s+1)x}\\
\end{align*}
and also
\begin{align*}
|&\E (W_{(c'_0,...,c'_{m'})[r]}((v'_0,...,v'_{m'}))|\\
&\le n^{e-m''}(k/(\min p_i))^{3(w-e)+2}2^{(2x+3(w-e)+2)(s-1)}\left(\max_{s\le j\le h} |\lambda_j|^{l-s+1}\prod_{i=1}^{s-1} |\lambda_j-\lambda'_i|\right)^{m''/l}\\
&\cdot \left(\max_{s\le j\le h} (\lambda_j^2/\Lambda)^{l-s+1}\prod_{i=1}^{s-1} (\lambda_j-\lambda'_i)^2/\Lambda\right)^{-e/l}\\
&\cdot \left(\max_{s\le j\le h} (|\lambda_j|/\Lambda)^{l-s+1}\prod_{i=1}^{s-1} |\lambda_j-\lambda'_i|/\Lambda\right)^{-\frac{(3w-3e+2)(l+2r+1)}{l}-2x}.
\end{align*}
\end{lemma}

\begin{proof}
First, consider the standard decomposition of $(v'_0,...,v'_{m'})$. The only vertices that could have been deleted are vertices within $r$ of the edge of a fresh segement with nonzero corresponding values of $c'''_i$, vertices that are not in a fresh segment with nonzero corresponding values of $c'''_i$, and the $x$ vertices that were deleted as copies of vertices in $v'_0,...,v'_{m'}$. There are at most $2e$ indices $i$ such that $(v'_i,v'_{i+1})$ is not in a fresh segment and there are no $i$ and $i'$ such that $|i-i'|\le 2r$, $c'''_i\ne 0$, and $c'''_{i'}\ne 0$. So, there are at most $6(w-e)+2+e/r$ vertices with nonzero $c'''_i$ that are not in a fresh segment and farther than $r$ from its edge. 

Now, remove $u_1,...,u_r$ from whichever fresh segment they are in, splitting it into two fresh segments on either side of $u_1,...,u_r$ if necessary. Let $d,d',t,t'$ be the measures of the resulting set of fresh segments, and $t''=\lceil \frac{t-d(l+2r+1)}{l}\rceil$. By the walk decomposition lemma, $d+d'\le 3(w-e)+2$. Also, note that every edge in one of the repeated fresh segments is repeated later in the walk. So, $t+2t'\le m''$. Every edge except those that involve the $u_i$ appears exactly once in a fresh segment, so $t+t'=m''-e$ and $t\ge m''-2e$. On another note, for every $j<s$, we have that
\begin{align*} & |\lambda_j|^{l-s+1}\prod_{i=1}^{s-1} |\lambda_j-\lambda'_i|\le  \Lambda^{l-s+1}|\lambda_j-\lambda'_j|(2\Lambda)^{s-2} \\
&\le |\lambda_s|^{l-s+1}(\min_i |\lambda_s-\lambda_i|/2)^{s-1}\le |\lambda_s|^{l-s+1}\prod_{i=1}^{s-1}|\lambda_s-\lambda'_i|.\end{align*}
Also, for every $j\ge s$, we have that 
\[|\lambda_j|^{l-s+1}\prod_{i=1}^{s-1} |\lambda_j-\lambda'_i|\le |\lambda_s|^{l-s+1}(2\Lambda)^{s-1}.\]
That means that if $(\Lambda/|\lambda_s|)^{l-s+1}\ge 2^{s-1}$ and $\lambda_s^2> \Lambda$ then

\begin{align*}
|&\E (W_{(c'_0,...,c'_{m'})[r]}((v'_0,...,v'_{m'}))|\\
&\le n^{-t-t'}k^{d+d'}(\min p_i)^{-d-d'}\lambda_1^{t'+\max(t-(2x+d)(s-1)-l\cdot\max(t''-2x,0),0)}\\
&\cdot (\lambda_1+\Lambda)^{t-\max(t-(2x+d)(s-1)-l\cdot\max(t''-2x,0),0)-l\cdot\max(t''-2x,0)}\\
&\cdot
\left(\max_{1\le j\le k} |\lambda_j|^{l-s+1}\prod_{i=1}^{s-1} |\lambda_j-\lambda'_i|\right)^{\max(t''-2x,0)}\\
&\le n^{e-m''}k^{d+d'}(\min p_i)^{-d-d'}(2\Lambda)^{(2x+d)(s-1)}\\
&\cdot\Lambda^{m''-e-(2x+d)(s-1)-l\cdot\max(t''-2x,0)}\left(\max_{s\le j\le h} |\lambda_j|^{l-s+1}\prod_{i=1}^{s-1} |\lambda_j-\lambda'_i|\right)^{\max(t''-2x,0)}\\
&\le n^{e-m''}(k/(\min p_i))^{3(w-e)+2}2^{(2x+3(w-e)+2)(s-1)}\Lambda^{m''-e}\\
&\cdot \left(\max_{s\le j\le h} (|\lambda_j|/\Lambda)^{l-s+1}\prod_{i=1}^{s-1} |\lambda_j-\lambda'_i|/\Lambda\right)^{\frac{m''-2e-(3w-3e+2)(l+2r+1)}{l}-2x}\\
&= n^{e-m''}(k/(\min p_i))^{3(w-e)+2}2^{(2x+3(w-e)+2)(s-1)}\left(\max_{s\le j\le h} |\lambda_j|^{l-s+1}\prod_{i=1}^{s-1} |\lambda_j-\lambda'_i|\right)^{m''/l}\\
&\cdot \left(\max_{s\le j\le h} (\lambda_j^2/\Lambda)^{l-s+1}\prod_{i=1}^{s-1} (\lambda_j-\lambda'_i)^2/\Lambda\right)^{-e/l}\\
&\cdot \left(\max_{s\le j\le h} (|\lambda_j|/\Lambda)^{l-s+1}\prod_{i=1}^{s-1} |\lambda_j-\lambda'_i|/\Lambda\right)^{-\frac{(3w-3e+2)(l+2r+1)}{l}-2x}.
\end{align*}
Alternately,

\begin{align*}
|&\E (W_{(c'_0,...,c'_{m'})[r]}((v'_0,...,v'_{m'}))|\\
&\le n^{-t-t'}k^{d+d'}(\min p_i)^{-d-d'}\lambda_1^{t'+\max(t-(2x+d)(s-1)-l\cdot\max(t''-2x,0),0)}\\
&\cdot (\lambda_1+\Lambda)^{t-\max(t-(2x+d)(s-1)-l\cdot\max(t''-2x,0),0)-l\cdot\max(t''-2x,0)}\\
&\cdot
\left(\max_{1\le j\le k} |\lambda_j|^{l-s+1}\prod_{i=1}^{s-1} |\lambda_j-\lambda'_i|\right)^{\max(t''-2x,0)}\\
&\le n^{e-m''}k^{d+d'}(\min p_i)^{-d-d'}2^{(2x+d)(s-1)}\Lambda^{m''-e}\left((|\lambda_s|/\Lambda)^{l-s+1}2^{s-1}\right)^{t''-2x}\\
&\le n^{e-m''}k^{d+d'}(\min p_i)^{-d-d'}2^{d(s-1)}\Lambda^{m''-e}\left((|\lambda_s|/\Lambda)^{l-s+1}2^{s-1}\right)^{\frac{m''-2e-d(l+2r+1)}{l}}(\Lambda/|\lambda_s|)^{2(l-s+1)x}\\
&\le n^{e-m''}k^{d+d'}(\min p_i)^{-d-d'}2^{d(s-1)}\\
&\cdot\left(|\lambda_s|^{l-s+1}(2\Lambda)^{s-1}\right)^{m''/l}\left((\Lambda/\lambda_s^2)^{l-s+1}(1/4\Lambda)^{s-1}\right)^{e/l}\\
&\cdot
\left((|\lambda_s|/\Lambda)^{l-s+1}2^{s-1}\right)^{-\frac{d(l+2r+1)}{l}}(\Lambda/|\lambda_s|)^{2(l-s+1)x}\\
&\le n^{e-m''}k^{d+d'}(\min p_i)^{-d-d'}2^{d(s-1)}|\lambda_s|^{m''}\left((2\Lambda/|\lambda_s|)^{s-1}\right)^{m''/l}\\
&\cdot(\Lambda/\lambda_s^2)^e\left((|\lambda_s|/\Lambda)^{l-s+1}2^{s-1}\right)^{-\frac{d(l+2r+1)}{l}}(\Lambda/|\lambda_s|)^{2(l-s+1)x}\\
&\le n^{e-m''}(2^{s-1}k/\min p_i)^{3w-3e+2}|\lambda_s|^{m''}\left(2\Lambda/|\lambda_s|\right)^{(s-1)m''/l}\\
&\cdot(\Lambda/\lambda_s^2)^e\left((\Lambda/|\lambda_s|)^{l-s+1}2^{1-s}\right)^{\frac{(3w-3e+2)(l+2r+1)}{l}}(\Lambda/|\lambda_s|)^{2(l-s+1)x}.
\end{align*}
\end{proof}

\begin{definition}
$W_{(c'_0,...,c'_{m'})[r]}((v'_0,...,v'_{m'}))$ is a degree $z$ shard for $(\sigma_1,\sigma_2,\sigma_3,\sigma_4)$, $m_1$, $m_2$, $r$, $(c_0,...,c_{m_1})$, and $(c''_0,...,c''_{m_2})$ if there exists $x\in\mathbb{Z}$ and $v_0,...,v_{m_1},v''_0,...,v''_{m_2}$ such that $W_{(c'_0,...,c'_{m'})[r]}((v'_0,...,v'_{m'}))$ is a level $x$ shard of \newline $W_{(c_0,...,c_{m_1},0,...,0,c''_{m_2},c''_{m_2-1},...,c''_0)[r]}((v_0,...,v_{m_1},u_1,...,u_r,v''_{m_2},v''_{m_2-1},...,v''_0))$, the number of distinct sets $\{v'_i,v'_{i+1}\}$ minus the number of distinct vertices in $v'_0,...,v'_{m'}$ is $z-x-1$, and $v_0$, $v_{m_1}$, $v''_0$, and $v''_{m_2}$ are in communities $\sigma_1$, $\sigma_2$, $\sigma_3$, and $\sigma_4$ respectively.
\end{definition}

\begin{lemma}[Degree bound lemma] 
For any functions $m_1$ and $m_2$ of $n$ such that $m_1,m_2=O(\log n)$, there exists $r_0$ such that for any $r\ge r_0$, $z\ge 0$, and $\Lambda,s,l$ such that $(\Lambda/|\lambda_s|)^{l-s+1}\ge 2^{s-1}$, $\lambda_s^2> \Lambda>\lambda_1$, and $l\ge (2r+1)(s-1)$ the following holds. Let $\lambda'_1,...,\lambda'_{s-1}$ be a $\Lambda$-bounded eigenvalue approximation and $c_0,...,c_{m_1}$ and $c''_0,...,c''_{m_2}$ be $l[r]$-cycles of this approximation with error less than $2\Lambda(\min_i |\lambda_s-\lambda_i|/4\Lambda)^{s-1}|\lambda_s/\Lambda|^{l-s+1}$. For any $\sigma_1, \sigma_2, \sigma_3, \sigma_4$, the expected value of the sum of the absolute values of all shards of degree at least $z$ for $(\sigma_1, \sigma_2, \sigma_3, \sigma_4)$, $m_1$, $m_2$, $r$, $(c_0,...,c_{m_1})$, and $(c''_0,...,c''_{m_2})$ times their coefficients is $O(n^{(15-5z)/6}\prod|\lambda_s-c_i|\cdot\prod|\lambda_s-c''_i|)$.
\end{lemma}

\begin{proof}
Let $m=m_1+m_2$. For any given $x$, $w$, $m'$, and $e$ with $x+w-e>z$ and $x\le m_1+m_2+r+1-m'\le x+6(w-e)+2+e/r$, the expected sum of the values of all level $x$ shards for $(\sigma_{v_0}, \sigma_{v_{m_1}}, \sigma_{v''_0}, \sigma_{v''_{m_2}})$, $m_1$, $m_2$, $r$, $(c_0,...,c_{m_1})$, and $(c''_0,...,c''_{m_2})$ with $m'+1$ vertices, $m'+1-w$ distinct vertices, and $m'-e$ distinct adjacent pairs of vertices has an absolute value of at most

\begin{align*}
& \sum_{(v'_0,...,v'_{m'}): |\{v'\in (v'_0,...,v'_{m'})\}|=m'+1-e,|\{\{v'_i,v'_{i+1}\}:0\le i<m'\}|=m'-e,(c'_0,...,c'_{m'})\in\mathbb{R}^{m'+1}} (\Lambda/n)^{m+r+1-m'}\\
&\cdot |\E [W_{(c'_0,...,c'_{m'})}(v'_0,...,v'_{m'})]|\\
&\cdot |\{(v_1,...,v_{m_1-1},v''_1,...,v''_{m_2-1}): W_{(c'_0,...,c'_{m'})}(v'_0,...,v'_{m'})\\
&\in Shar_x(W_{(c_0,...,c_{m_1},0,...,0,c''_{m_2},...,c''_0)[r]}((v_0,...,v_{m_1},u_1,...,u_r,v''_{m_2},...,v''_0))) \}|\\
&\le [2n^{m'-w+1}(m+r+2)^{4(w-e)}2^{2\lceil\frac{m'+1}{r+1}\rceil(w-e)}]\cdot(\Lambda/n)^{m+r+1-m'}\\
&\cdot [3^{m/r}(m+2)^xn^{m+1-m'-x}]\cdot [n^{e-m'+r+1}(2^{s-1}k/\min p_i)^{3w-3e+2}|\lambda_s|^{m'-r-1}\\
&\cdot \left(2\Lambda/|\lambda_s|\right)^{(s-1)\cdot(m'-r-1)/l}\cdot(\Lambda/\lambda_s^2)^e\left((\Lambda/|\lambda_s|)^{l-s+1}2^{1-s}\right)^{\frac{(3w-3e+2)(l+2r+1)}{l}}(\Lambda/|\lambda_s|)^{2(l-s+1)x}]\\
&\le  n^{e-w+13/6-x}2^{2\lceil\frac{m'+1}{r+1}\rceil(w-e)}(m+r+2)^{4(w-e)+x}\Lambda^{m+r+1-m'}(\Lambda/\lambda_s^2)^e\\
&2(2^{s-1}k/\min p_i)^{3w-3e+2}|\lambda_s|^{m'-r-1}\left(2\Lambda/|\lambda_s|\right)^{(s-1)\cdot(m'-r-1)/l} \\
&\cdot \left((\Lambda/|\lambda_s|)^{l-s+1}2^{1-s}\right)^{\frac{(3w-3e+2)(l+2r+1)}{l}}(\Lambda/|\lambda_s|)^{2(l-s+1)x}\\
&=  |\lambda_s|^{m}\left(2\Lambda/|\lambda_s|\right)^{(s-1)m/l}\cdot n^{e-w+13/6-x}2^{2\lceil\frac{m'+1}{r+1}\rceil(w-e)}\\
&\cdot 2(m+r+2)^{4(w-e)+x}(\Lambda/\lambda_s^2)^e(2^{s-1}k/\min p_i)^{3w-3e+2}(\Lambda/|\lambda_s|(2\Lambda/|\lambda_s|)^{-(s-1)/l})^{m+r+1-m'}\\
&\cdot \left((\Lambda/|\lambda_s|)^{l-s+1}2^{1-s}\right)^{\frac{(3w-3e+2)(l+2r+1)}{l}}(\Lambda/|\lambda_s|)^{2(l-s+1)x}\\
\end{align*}
and the last term above is upper-bounded by

\begin{align*}
& |\lambda_s|^{m}\left(2\Lambda/|\lambda_s|\right)^{m(s-1)/l} n^{e-w+13/6-x}2^{2\lceil\frac{m+r+2}{r+1}\rceil(w-e)}\\
&\cdot 2(m+r+2)^{4(w-e)+x}(\Lambda/\lambda_s^2)^e(2^{s-1}k/\min p_i)^{3w-3e+2}(\Lambda/ |\lambda_s|(2\Lambda/|\lambda_s|)^{-(s-1)/l})^{ x+6(w-e)+2+e/r}\\
&\cdot \left((\Lambda/|\lambda_s|)^{l-s+1}2^{1-s}\right)^{\frac{(3w-3e+2)(l+2r+1)}{l}}(\Lambda/|\lambda_s|)^{2(l-s+1)x}\\
&\le 2  |\lambda_s|^{m}\left(2\Lambda/|\lambda_s|\right)^{m(s-1)/l}(2^{s-1}k/\min p_i)^2(\Lambda/|\lambda_s|(2\Lambda/|\lambda_s|)^{-(s-1)/l})^2\\
&\left((\Lambda/|\lambda_s|)^{l-s+1}2^{1-s}\right)^{\frac{2(l+2r+1)}{l}}n^{13/6}\cdot \left((m+r+2)(\Lambda /|\lambda_s|(2\Lambda/|\lambda_s|)^{-(s-1)/l})(\Lambda/|\lambda_s|)^{2(l-s+1)}/n\right)^x\\
&\cdot\left(2^{2\lceil\frac{m+r+2}{r+1}\rceil}(m+r+2)^4 (\frac{2^{s-1}k}{\min p_i})^3(\frac{\Lambda}{|\lambda_s|}(2\Lambda/|\lambda_s|)^{-(s-1)/l})^6\left((\Lambda/|\lambda_s|)^{l-s+1}2^{1-s}\right)^{\frac{3(l+2r+1)}{l}}/n\right)^{w-e}\\
&\cdot \left(\frac{\Lambda^r}{\lambda^{2r}_s}\cdot\frac{\Lambda}{|\lambda_s|}(2\Lambda/|\lambda_s|)^{-(s-1)/l}\right)^{e/r}\\
&\le |\lambda_s|^m n^{(14-5z)/6}
\end{align*}
provided that $r> 12m/\log_2(n)$ and $n$ is sufficiently large. There are at most $2(s-1)+(s-1)m/l$ indices $i$ such that $c_i\ne 0$ or $c''_i\ne 0$, so 
$|\lambda_s|^{-m}\prod|\lambda_s-c_i|\prod|\lambda_s-c''_i|\ge (\min_{s'<s} |\lambda_s-\lambda_{s'}|/2)^{(s-1)m/l}$. So, as long as $r_0$ is large enough that $\frac{m}{2r+1}\ln(\min_{s'<s} |\lambda_s-\lambda_{s'}|/2)<\ln(n)/3$, then $|\lambda_s|^m n^{(14-5z)/6}\le n^{(15-5z)/6}\prod|\lambda_s-c_i|\cdot \prod|\lambda_s-c''_i|$.
\end{proof}

\begin{lemma}\label{communityAverage}
Assume that $m$, $\Lambda$, $\lambda'_1,...,\lambda'_{s-1}$, $l$, $r$, and $(c_0,...,c_m)$ satisfy the conditions of the degree bound lemma. Also asume that $m=\Omega(\log n)$ and either $s=h'$ or $|\lambda_s|>|\lambda_{s+1}|$. Now, let $w$ be an eigenvector of $PQ$ with an eigenvalue of $\lambda_j$. If $j\ne s$ then with probability $1-o(1)$ the average value over all $v$ of $w_{\sigma_v} W_{m/\{c_i\}}(x,v)/p_{\sigma_v}$ is $O(\frac{1}{\ln(n)\sqrt{n}}\prod |\lambda_s-c_i|)$ and if $j=s$ then the average value over all $v$ of $w_{\sigma_v} W_{m/\{c_i\}}(x,v)/p_{\sigma_v}$ is $\Omega(\frac{1}{\sqrt{n}}\prod |\lambda_j-c_i|)$ with probability $1-o(1)$.

\end{lemma}

\begin{proof}
Let $m'''=\lfloor \sqrt{\ln(n)}\rfloor$, unless $c_{\lfloor \sqrt{\ln(n)}\rfloor}\ne 0$, in which case let $m'''=\lfloor \sqrt{\ln(n)}\rfloor+1$. The first step in proving the desired bounds is to justify approximating $\sum_{v\in\Omega_j} W_{m/\{c_i\}}(x,v)$ with $\sum_{v_0,...,v_{m'''}\in G} x_{v_0}\cdot W_{(c_0,...,c_{m'''})[r]}((v_0,...,v_{m'''}) e_{\sigma_{v_{m'''}}}\cdot \prod_{i=m'''+1}^m (PQ-c_i) e_j$. In order to do that, we observe that for any $j$,
\begin{align*}
&E\left[\left(\sum_{v\in\Omega_j} W_{m/\{c_i\}}(x,v)-\sum_{v_0,...,v_{m'''}\in G} x_{v_0}\cdot W_{(c_0,...,c_{m'''})[r]}((v_0,...,v_{m'''}) e_{\sigma_{v_{m'''}}}\cdot \prod_{i=m'''+1}^m (PQ-c_i) e_j\right)^2\right]\\
&=E[(\sum_{v_0,...,v_m\in G:v_m\in \Omega_j} x_{v_0}W_{(c_0,...,c_{m})[r]}(v_0,...,v_{m})\\
&-x_{v_0}\cdot W_{(c_0,...,c_{m'''})[r]}((v_0,...,v_{m'''}) e_{\sigma_{v_{m'''}}}\cdot \prod_{i=m'''+1}^m \frac{1}{n}(PQ-c_i) e_j)^2]\\
&=E[(\sum_{v_0,...,v_m,v'_0,...,v'_{m}\in G: v_m,v'_m\in\Omega_j} x_{v_0}x_{v'_0}W_{(c_0,...,c_{m},0,...,0,c_m,...,c_0)[r]}(v_0,...,v_{m},u_1,...,u_r,v'_m,...,v'_0)\\
&-x_{v_0}x_{v'_0}W_{(c_0,...,c_{m},0,...,0,c_{m'''},...,c_0)[r]}(v_0,...,v_{m},u_1,...,u_r,v'_{m'''},...,v'_0) e_{\sigma_{v'_{m'''}}}\cdot \prod_{i=m'''+1}^m \frac{1}{n}(PQ-c_i) e_j\\
&-x_{v_0}x_{v'_0}W_{(c_0,...,c_{m'''},0,...,0,c_m,...,c_0)[r]}(v_0,...,v_{m'''},u_1,...,u_r,v'_m,...,v'_0) e_{\sigma_{v_{m'''}}}\cdot \prod_{i=m'''+1}^m \frac{1}{n}(PQ-c_i) e_j\\
&+x_{v_0}x_{v'_0}W_{(c_0,...,c_{m'''},0,...,0,c_{m'''},...,c_0)[r]}(v_0,...,v_{m'''},u_1,...,u_r,v'_{m'''},...,v'_0)\\
&\indent\indent \left(e_{\sigma_{v_{m'''}}}\cdot \prod_{i=m'''+1}^m \frac{1}{n}(PQ-c_i) e_j\right) \left(e_{\sigma_{v'_{m'''}}}\cdot \prod_{i=m'''+1}^m \frac{1}{n}(PQ-c_i) e_j\right))]\\
\end{align*}

Note that $E[v_0v'_0]$ is $1$ if $v_0=v'_0$ and $0$ otherwise. So, the only $(v_0,...,v_m),(v'_0,...,v'_m)$ that make a nonzero contribution to the expected value of the sum above are those for which $v_0=v'_0$. Also, given any $(v_0,...,v_m),(v'_0,...,v'_m)$ such that $v_{m'''+1},...,v_m$ are distinct, and $\{v_{m'''+1},...,v_m\}\cap(\{v_0,...,v_{m'''}\}\cup\{v'_0,...,v'_m\})=\emptyset$, then for any assignment of communities to $v_{m'''},v_m$, and $v'_m$, and any possible value of the subgraph of $G$ induced by $\{v_0,...,v_{m'''}\}\cup\{v'_0,...,v'_m\}$, the expected value of $W_{(c_0,...,c_m)[r]}(v_0,...,v_m)$ given these values is $W_{(c_0,...,c_{m'''})[r]}(v_0,...,v_{m'''}) e_{\sigma_{m'''}}\cdot \prod_{i=m'''+1}^m (PQ-c_i)e_j$. That means that the expected contribution to the sum above of the terms corresponding to such $(v_0,...,v_m),(v'_0,...,v'_m)$ is $0$. By the same logic, all $(v_0,...,v_m),(v'_0,...,v'_m)$ such that $v'_{m'''+1},...,v'_m$ are distinct and $\{v'_{m'''+1},...,v'_m\}\cap(\{v'_0,...,v'_{m'''}\}\cup\{v_0,...,v_m\})=\emptyset$ have an expected contribution of $0$ to the sum above.

Now, let $V$ be the set of all pairs of tuples $((v_0,...,v_m), (v'_0,...,v'_m))$ such that $v_0=v'_0$, either $|\{v_{m'''+1},...,v_m\}|< m-m'''$ or $\{v_{m'''+1},...,v_m\}\cap(\{v_0,...,v_{m'''}\}\cup\{v'_0,...,v'_m\})\ne \emptyset$, and either  $|\{v'_{m'''+1},...,v'_m\}|< m-m'''$ or $\{v'_{m'''+1},...,v'_m\}\cap(\{v'_0,...,v'_{m'''}\}\cup\{v_0,...,v_m\})\ne \emptyset$. Note that for every $v_0,...,v_{m'''}$ and $v'_0,...,v'_m$, there are at most $(2m+2)^2 n^{m-m'''-1}$ possible choices of $v_{m'''+1},...,v_m$ such that $((v_0,...,v_m),(v'_0,...,v'_m))\in V$. So,
\begin{align*}
&E[\sum_{((v_0,...,v_m),(v'_0,...,v'_m))\in V:v_m,v'_m\in\Omega_j} |x_{v_0}x_{v'_0}W_{(c_0,...,c_{m},0,...,0,c_{m'''},...,c_0)[r]}(v_0,...,v_{m},u_1,...,u_r,v'_{m'''},...,v'_0)\\
&\indent e_{\sigma_{v'_{m'''}}}\cdot \prod_{i=m'''+1}^m \frac{1}{n}(PQ-c_i) e_j|]\\
&\le [(2m+2)^2 n^{m-m'''-1}]\cdot  \left[n^{m'''-m}/(\min_i p_i)\max_{1\le i'\le h} \prod_{i=m'''+1}^m |\lambda_{i'}-c_i|\right]\\
&\indent \cdot E\left[\sum_{v_0,...,v_{m},v'_0,...,v'_{m'''}\in G:v_m,v'_m\in\Omega_j} |W_{(c_0,...,c_{m},0,...,0,c_{m'''},...,c_0)[r]}(v_0,...,v_{m},u_1,...,u_r,v'_{m'''},...,v'_0)\right]\\
&= [(2m+2)^2 n^{m-m'''-1}]\cdot  \left[n^{m'''-m}/(\min_i p_i)\max_{1\le i'\le h} \prod_{i=m'''+1}^m |\lambda_{i'}-c_i|\right] \\
&\quad \cdot O\left(n^{10/6}\prod_{i=1}^{m'''}|\lambda_s-c_i|\prod_{i=1}^m |\lambda_s-c_i|\right)\\
&=O\left( n^{5/6}\prod_{i=1}^{m} (\lambda_s-c_i)^2\right)
\end{align*}

where the second to last equality follows from the degree bound lemma and the fact that $v_0=v'_0$ implies that every shard of a $W$ in this expression has degree at least $1$. The last equality follows from the facts that $\prod_{i={m'''}}^m |\lambda_{i'}-c_i|=O(\prod_{i={m'''}}^m |\lambda_s-c_i|)$ for any $i'$ and $m=O(\ln(n))$. The same logic applies to the analagous expressions for $W_{(c_0,...,c_{m'''},0,...,0,c_{m},...,c_0)[r]}(v_0,...,v_{m'''},u_1,...,u_r,v'_{m},...,v'_0)$ and \newline $W_{(c_0,...,c_{m'''},0,...,0,c_{m'''},...,c_0)[r]}(v_0,...,v_{m'''},u_1,...,u_r,v'_{m'''},...,v'_0)$. This implies that
\begin{align*}
&E\left[\left(\sum_{v\in\Omega_j} W_{m/\{c_i\}}(x,v)-\sum_{v_0,...,v_{m'''}\in G} x_{v_0}\cdot W_{(c_0,...,c_{m'''})[r]}((v_0,...,v_{m'''}) e_{\sigma_{v_{m'''}}}\cdot \prod_{i=m'''+1}^m (PQ-c_i) e_j\right)^2\right]\\
&=E\left[\sum_{((v_0,...,v_m),(v'_0,...,v'_{m}))\in V: v_m,v'_m\in\Omega_j} W_{(c_0,...,c_{m},0,...,0,c_m,...,c_0)[r]}(v_0,...,v_{m},u_1,...,u_r,v'_m,...,v'_0)\right]\\
&\indent\indent\indent+O\left(n^{5/6}\prod_{i=1}^{m} (\lambda_s-c_i)^2\right)\\
\end{align*}

If the $W_{(c_0,...,c_{m},0,...,0,c_m,...,c_0)[r]}(v_0,...,v_{m},u_1,...,u_r,v'_m,...,v'_0)$ are broken down into weighted sums of shards, then every resulting shard has degree at least $1$ for the same reason as in the previous cases, and all resulting shards of degree at least $2$ have a combined contribution to the expected value of at most $n^{5/6}\prod_{i=1}^m (\lambda_s-c_i)^2$ by the degree bound lemma. Now, let $W_{(c'_0,...,c'_{m'})[r]}(v''_0,...,v''_{m'})$ be a degree $1$ shard of $ W_{(c_0,...,c_{m},0,...,0,c_m,...,c_0)[r]}(v_0,...,v_{m},u_1,...,u_r,v'_m,...,v'_0)$ for some $((v_0,...,v_m),(v'_0,...,v'_{m}))\in V$. Also, let $H$ be the walk graph of $v''_0,...,v''_{m'}$ with $u_1,...,u_r$ removed. We know that $v''_0=v_0=v'_0=v''_{m'}$, so $H$ is connected. That means that the only way that $W_{(c'_0,...,c'_{m'})[r]}(v''_0,...,v''_{m'})$ can be a degree $1$ shard is if $H$ is a tree and no vertex that was deleted in the process of converting $ W_{(c_0,...,c_{m},0,...,0,c_m,...,c_0)[r]}(v_0,...,v_{m},u_1,...,u_r,v'_m,...,v'_0)$ to $W_{(c'_0,...,c'_{m'})[r]}(v''_0,...,v''_{m'})$ is repeated in $v''_0,...,v''_{m'}$. So, there must exist some $t_0\ge 0$ such that $v''_i=v''_{m'-i}$ for all $i\le t_0$, $v''_{t_0+1},...,v''_{m'-t_0-1}$ are distinct, and $\{v''_0,...,v''_{t_0}\}\cap \{v''_{t_0+1},...,v''_{m'-t_0-1}\}=\emptyset$. Furthermore, since $((v_0,...,v_m),(v'_0,...,v'_m))\in V$, there exist $i$ and $i'$ such that $i> m''$ and $v_i=v'_{i'}$ or $v_i=v_{i'}$ with $i\ne i'$. Either neither of the repeated elements were deleted in the process of converting  $ W_{(c_0,...,c_{m},0,...,0,c_m,...,c_0)[r]}(v_0,...,v_{m},u_1,...,u_r,v'_m,...,v'_0)$ to $W_{(c'_0,...,c'_{m'})[r]}(v''_0,...,v''_{m'})$ or they both were, in which case $v_i$ must have been within $r$ of a repeated vertex that was not deleted during the conversion. The only vertices that could have been deleted during the conversion are those with nonzero corresponding $c_i$, so either way, we have that $t_0\ge \frac{2r}{2r+1}m''-r$. Also, by lemma \ref{lem10}, we have that

\begin{align*}
&|E[W_{(c'_0,...,c'_{m'})[r]}(v''_0,...,v''_{m'})]|\\
&\le n^{t_0-m'+r+1}(k/(\min p_i))^{5}2^{5(s-1)}\left(\max_{s\le j\le h} |\lambda_j|^{l-s+1}\prod_{i=1}^{s-1} |\lambda_j-\lambda'_i|\right)^{(m'-r-1)/l}\\
&\cdot \left(\max_{s\le j\le h} (\lambda_j^2/\Lambda)^{l-s+1}\prod_{i=1}^{s-1} (\lambda_j-\lambda'_i)^2/\Lambda\right)^{-t_0/l}\\
&\cdot \left(\max_{s\le j\le h} (|\lambda_j|/\Lambda)^{l-s+1}\prod_{i=1}^{s-1} |\lambda_j-\lambda'_i|/\Lambda\right)^{-\frac{5(l+2r+1)}{l}}.
\end{align*}

For fixed values of $t_0$ and $m'$, there are at most $2m\cdot n^{(m'-t_0-r)}$ possible values of $(v''_0,...,v''_{m'}).$ Furthermore, for each possible value of $(v''_0,...,v''_{m'})$, there are at most $2^{t_0/r+1}n^{2m+r+1-m'}$ possible values of $v_0,...,v_m$, $v'_0,...,v'_m$, and $c'_0,...,c'_{m'}$ such that $W_{(c'_0,...,c'_{m'})[r]}(v''_0,...,v''_{m'})$ is a degree $1$ shard of $ W_{(c_0,...,c_{m},0,...,0,c_m,...,c_0)[r]}(v_0,...,v_{m},u_1,...,u_r,v'_m,...,v'_0)$, and this shard has a coefficient with an absolute value of at most $(\Lambda/n)^{2m+r+1-m'}$. So, the combined contribution to the expected value from all degree $1$ shards with a given value of $t_0$ and $m'$ is at most
\begin{align*}
&2^{t_0/r+2}m\Lambda^{2m+r+1-m'}n (2^{s-1}k/(\min p_i))^{5}\left(\max_{s\le j\le h} |\lambda_j|^{l-s+1}\prod_{i=1}^{s-1} |\lambda_j-\lambda'_i|\right)^{(m'-r-1)/l}\\
&\cdot \left(\max_{s\le j\le h} (\lambda_j^2/\Lambda)^{l-s+1}\prod_{i=1}^{s-1} (\lambda_j-\lambda'_i)^2/\Lambda\right)^{-t_0/l}\\
&\cdot \left(\max_{s\le j\le h} (|\lambda_j|/\Lambda)^{l-s+1}\prod_{i=1}^{s-1} |\lambda_j-\lambda'_i|/\Lambda\right)^{-\frac{5(l+2r+1)}{l}}\\
&\le 2^{2m''/(2r+1)+1}m\Lambda^{2m''/(2r+1)}n (2^{s-1}k/(\min p_i))^{5}\left(\max_{s\le j\le h} |\lambda_j|^{l-s+1}\prod_{i=1}^{s-1} |\lambda_j-\lambda'_i|\right)^{(2m-2m''/(2r+1))/l}\\
&\cdot \left(\max_{s\le j\le h} (\lambda_j^2/\Lambda)^{l-s+1}\prod_{i=1}^{s-1} (\lambda_j-\lambda'_i)^2/\Lambda\right)^{-(\frac{2r}{2r+1}m''-r)/l}\\
&\cdot \left(\max_{s\le j\le h} (|\lambda_j|/\Lambda)^{l-s+1}\prod_{i=1}^{s-1} |\lambda_j-\lambda'_i|/\Lambda\right)^{-\frac{5(l+2r+1)}{l}}\\
&=O\left(n\ln^{-5}(n)\prod_{i=0}^m (\lambda_s-c_i)^2\right)
\end{align*}

There are only $O(\ln^2(n))$ possible values of $t_0$ and $m'$, so the combined contribution of all degree $1$ shards is $O\left(n\ln^{-3}(n)\prod_{i=0}^m (\lambda_s-c_i)^2\right)$. So, with probability $1-o(1)$,  the average value over all $v\in\Omega_\sigma$ of $W_{m/\{c_i\}}(x,v)$ is within $o(\frac{\sqrt{n}}{\log(n)}\prod |\lambda_s-c_i|)$ of 
\[\frac{1}{p_\sigma} e_\sigma\cdot\prod_{i=m''+1}^m (PQ-c_i) \sum_{(v_0,...,v_{m'''})} x_{v_0}\cdot W_{(c_0,...,c_{m'''})[r]}((v_0,...,v_{m'''})) e_{\sigma_{v_{m''}}}\]

If $w$ is an eigenvector of $PQ$ with eigenvalue $\lambda_j$, then this means that the average value of $w_{\sigma_v} W_{m/\{c_i\}}(x,v)/p_{\sigma_v}$ is within $o(\frac{1}{\log(n)}\prod |\lambda_s-c_i|)$ of  
\[\frac{1}{n}\prod_{i=m''+1}^m (\lambda_j-c_i) \sum_{(v_0,...,v_{m'''})} x_{v_0}\cdot W_{(c_0,...,c_{m'''})[r]}((v_0,...,v_{m'''})) w_{\sigma_{v_{m''}}}/p_{\sigma_{v_{m''}}}\]
For fixed $G$ and $\sigma$ but not fixed $x$, the probability distribution of this expression is a normal distribution with mean $0$ and variance
\[\frac{1}{n^2}\prod_{i=m''+1}^m (\lambda_j-c_i)^2\sum_{v_0} \left(\sum_{(v_1,...,v_{m'''})} W_{(c_0,...,c_{m'''})[r]}((v_0,...,v_{m'''})) w_{\sigma_{v_{m''}}}/p_{\sigma_{v_{m''}}}\right)^2\]

With probability $1-o(1)$, there is no $v_0\in G$ that has more than one cycle within $m'''$ edges of it, there are $O(\lambda_1^{2m'''})$ vertices that have a cycle within $m'''$ edges of them, and no vertex in the graph has degree greater than $\ln^2(n)$. This implies that 
\[\sum_{v_0} \left(\sum_{(v_1,...,v_{m'''})} W_{(c_0,...,c_{m'''})[r]}((v_0,...,v_{m'''})) w_{\sigma_{v_{m''}}}/p_{\sigma_{v_{m''}}}\right)^2=O(n\ln^{2m'''}(n))\]
If $j\ne s$, then there exists $\epsilon>0$ such that $ \prod_{i=m''+1}^m (\lambda_j-c_i)^2=O(n^{-\epsilon}\prod_{i=m''+1}^m (\lambda_s-c_i)^2),$ so the variance is $o(\frac{1}{n\log^2(n)}\prod(\lambda_s-c_i)^2)$. Thus, the specified average is $O(\frac{1}{\log(n)\sqrt{n}}\prod |\lambda_s-c_i|)$ with probability $1-o(1)$, as desired.

Now, consider the case where $j=s$. If there is no cycle in the portion of the graph within $m''$ edges of $v_0$ then every nonbacktracking walk of length $m''$ or less starting at $v_0$ is a path. So, conditioned on the abscence of such cycles near $v_0$ and a fixed value of $\sigma_{v_0}$, we have that
\[ \left(\sum_{(v_1,...,v_{m'''})} W_{(c_0,...,c_{m'''})[r]}((v_0,...,v_{m'''})) w_{\sigma_{v_{m''}}}/p_{\sigma_{v_{m''}}}\right)^2=(1+o(1))\prod_{i=1}^{m'''-1} (\lambda_s-c_i) w_{\sigma_{v_0}}/p_{\sigma_{v_0}}\]
This means that the expected value of the square of this sum is $\Omega(\prod_{i=1}^{m'''-1} (\lambda_s-c_i)^2)$.
Furthermore, for any $v$ and $v'$ there is no vertex within $m''$ edges of both $v$ and $v'$ with probability $1-O(\lambda_1^{2m'''}/n)$. So, the joint probability distribution of the subgraphs of $G$ within $m''$ edges of $v$ and $v'$ differs from the product of the individual distributions by $O(\lambda_1^{2m''}/n)$. So, with probability $1-o(1)$ we have that 
\[\sum_{v_0} \left(\sum_{(v_1,...,v_{m'''})} W_{(c_0,...,c_{m'''})[r]}((v_0,...,v_{m'''})) w_{\sigma_{v_{m''}}}/p_{\sigma_{v_{m''}}}\right)^2=\Omega\left(n\prod_{i=1}^{m'''-1} (\lambda_s-c_i)^2\right).\]
Therefore, with probability $1-o(1)$, the average value over all $v$ of $w_{\sigma_v} W_{m/\{c_i\}}(x,v)/p_{\sigma_v}$ is $\frac{1}{\sqrt{n}}\Omega(\prod |\lambda_s-c_i|)$.
\end{proof}

\begin{lemma}
There exists a constant $r_0$ and $m_0=\Theta(\log n)$ such that if $m$, $\Lambda$, $\lambda'_1,...,\lambda'_{s-1}$, $l$, $r$, and $(c_0,...,c_m)$ satisfy the conditions of the degree bound lemma, $r>r_0$ and $m>m_0$, then for any communities $\sigma_1,\sigma_2$,
\begin{align*}
&|\E \left[\sum_{v_0,...,v_m,v''_0,...,v''_m\in G: \atop{v_m=v''_m,v_0=v''_0,\sigma_{v_m}=\sigma_1,\sigma_{v_0}=\sigma_2}} W_{(c_0,...,c_m)[r]}((v_0,...,v_m))\cdot W_{(c_0,...,c_m)[r]}((v''_0,...,v''_m))\right]|\\
&=O\left(\left(\sum_{i=s}^h \prod_{j=0}^m |\lambda_i-c_j|\right)^2\right)
\end{align*}
\end{lemma}

\begin{proof}
Let $W_{(c'_0,...,c'_{m'})[r]}((v'_0,...,v'_{m'}))$ be a level $x$ shard of $$W_{(c_0,...,c_m,0,...,0,c_m,...,c_0)[r]}((v_0,...,v_m,u_1,...,u_r,v''_m,...,v''_0)),$$ and $G'$ be $(v'_0,...,v'_{m'} )$'s walk graph. $c_0=c''_0=0$, so $v'_0=v'_{m'}$ is the only vertex in $G'$ that can have degree $1$. Since $c_m=c''_m=0$, $v_m\in G$ and it is adjacent to $u_1$, $u_r$, and $V_{m-1}$, meaning that $v_m$ has degree greater than $2$ and this shard has nonzero degree. 

By the degree bound lemma, the contribution to this expected value of all shards of degree greater than $2$ is $O(\prod (\lambda_s-c_i)^2)$. Now, assume that this shard has degree less than $3$, and let $G''$ be $G'$ with $u_1,...,u_r$ and all of their edges removed. $G''$ is still connected, and the only vertices that might have degree $1$ in it are $v_0$ and $v_m$. First, consider the case in which $G''$ is a tree. It must be a path, so $v'_0,...,v'_{m'}$ simply consists of a path, followed by a cycle, followed by the same path in reverse. There are at most $m/r$ nonzero elements of $(c_0,...,c_m,0.,...,0,c_m,...,c_0)$, so if $I$ is the set of indices of vertices in $(v_0,...,v_m,u_1,...,u_r,v''_m,...,v''_0)$ that were deleted in the process of converting it to this shard, there are at most $2^{m/r}$ possible values of $I$. Also, $m'=2m+r+1-|I|$ and for a given value of $I$, there are at most $n^{(2m+2-|I|)/2}$ possible values of $(v'_0,...,v'_{m'})$, each of which could have come from at most $n^{|I|}$ possible values of $(v_0,...,v_m)$ and $(v'_0,...,v'_m)$. The coefficient of the shard is at most $(\Lambda/n)^{|I|}$, and $c'_i=0$ for all $i$. That means that $W_{(c'_0,...,c'_{m'})[r]}((v'_0,...,v'_{m'}))$ is one if each of the $m-|I|/2$ edges in $G''$ are also in $G$ and $0$ otherwise. So, $|\E [W_{(c'_0,...,c'_{m'})[r]}((v'_0,...,v'_{m'}))]|=O((\lambda_1/n)^{m-|I|/2})$. That means that the total contribution to the expected value of all such shards is $O(2^{m/r}\Lambda^{m/r}\lambda_1^{m-m/2r}n)$. So, as long as
\[m>\ln(n)/\left(\ln(\lambda^2_s/\lambda_1)-\ln(4\Lambda^2/\lambda_1)/2r+\frac{2(s-1)}{l}\ln\left(\min_{i<s} |\lambda_i-\lambda_s|/2|\lambda_s|\right)\right)\]
 the contribution to the expected value of all such shards is $O\left(\left(\sum_{i=s}^h \prod_{j=0}^m |\lambda_i-c_j|\right)^2\right)$, as desired.

That leaves the case where $G''$ is not a tree. The fact that the shard has degree less than $3$ implies that $G'$ has at most $1$ more edges than vertices, so $G''$ has at least as many vertices as edges. So, if it is not a tree, it must have an equal number of edges and vertices, which means that it contains exactly one cycle. The only vertices in $G''$ that might have degree $1$ are $v_0$ and $v_m$, so $G''$ consists of a cycle and two paths (possibly of length $0$) that connect $v_0$ and $v_m$ to the cycle. Now, let $e$ be the length of the path from $v_0$ to the cycle, $e'$ be the length of the path from $v_m$ to the cycle, and $|I|$ be the set of all indices of vertices in $(v_0,...,v_m,u_1,...,u_r,v''_m,...,v''_0)$ that were deleted in the process of converting it to $(v'_0,...,v'_{m'})$. $m'=2m+r+1-|I|$, and the shard will have a coefficient of magnitude at most $(\Lambda/n)^{|I|}$. There are a few subcases to consider.
\begin{enumerate}
\item No edge in the cycle in $G''$ shows up more than once in the walk defined by $(v'_0,...,v'_{m'})$:

 In this case, there must be $m'-2e-2e'-r-1=2m-2e-2e'-|I|$ edges in the cycle. For given values of $e$ and $e'$, there are at most $(e+e')/r+2$ indices of vertices that could have been deleted in the transition from $W_{(c_0,...,c_m,0,...,0,c_m,...,c_0)[r]}((v_0,...,v_m,u_1,...,u_r,v''_m,...,v''_0))$ to $W_{(c'_0,...,c'_{m'})[r]}((v'_0,...,v'_{m'}))$. For fixed values of $e$, $e'$, and $I$, there are at most $n^{2m-e-e'-|I|}$ possible values of $(v'_0,...,v'_{m'})$, each of which could be derived from at most $n^{|I|}$ possible values of $(v_0,...,v_m)$ and $(v''_0,...,v''_m)$. Also, the absolute value of the shard's expected value is at most
\begin{align*}
&n^{e+e'-2m+|I|}(k/(\min p_i))^{8}2^{8(s-1)}\left(\max_{s\le j\le h} |\lambda_j|^{l-s+1}\prod_{i=1}^{s-1} |\lambda_j-\lambda'_i|\right)^{(2m-|I|)/l}\\
&\cdot \left(\max_{s\le j\le h} (\lambda_j^2/\Lambda)^{l-s+1}\prod_{i=1}^{s-1} (\lambda_j-\lambda'_i)^2/\Lambda\right)^{-(e+e')/l}\\
&\cdot \left(\max_{s\le j\le h} (|\lambda_j|/\Lambda)^{l-s+1}\prod_{i=1}^{s-1} |\lambda_j-\lambda'_i|/\Lambda\right)^{-\frac{8(l+2r+1)}{l}}
\end{align*}
which means that all such shards for given values of $e$ and $e'$ have a combined expected contribution of absolute value at most
\begin{align*}
&\Lambda^{(e+e'+2r)/r}2^{(e+e'+2r)/r}(k/(\min p_i))^{8}2^{8(s-1)}\left(\max_{s\le j\le h} |\lambda_j|^{l-s+1}\prod_{i=1}^{s-1} |\lambda_j-\lambda'_i|\right)^{(2m-(e+e')/r-2)/l}\\
&\cdot \left(\max_{s\le j\le h} (\lambda_j^2/\Lambda)^{l-s+1}\prod_{i=1}^{s-1} (\lambda_j-\lambda'_i)^2/\Lambda\right)^{-(e+e')/l}\\
&\cdot \left(\max_{s\le j\le h} (|\lambda_j|/\Lambda)^{l-s+1}\prod_{i=1}^{s-1} |\lambda_j-\lambda'_i|/\Lambda\right)^{-\frac{8(l+2r+1)}{l}}
\end{align*}
and thus all such shards for all $e$ and $e'$ have a combined expected contribution of absolute value at most
\begin{align*}
&4\Lambda^2(k/(\min p_i))^{8}2^{8(s-1)}\left(\max_{s\le j\le h} |\lambda_j|^{l-s+1}\prod_{i=1}^{s-1} |\lambda_j-\lambda'_i|\right)^{(2m-2)/l}\\
&\cdot \left(\max_{s\le j\le h} (|\lambda_j|/\Lambda)^{l-s+1}\prod_{i=1}^{s-1} |\lambda_j-\lambda'_i|/\Lambda\right)^{-\frac{8(l+2r+1)}{l}}\\
&\cdot \left[1-(4\Lambda)^{1/2r} \left(\max_{s\le j\le h} (\lambda_j^2/\Lambda)^{l-s+1}\prod_{i=1}^{s-1} (\lambda_j-\lambda'_i)^2/\Lambda\right)^{-(2r+1)/2rl}\right]^{-2}
\end{align*}
which is within a constant factor of 
\[\left(\max_{s\le j\le h} |\lambda_j|^{l-s+1}\prod_{i=1}^{s-1} |\lambda_j-\lambda'_i|\right)^{2m/l}=O\left(\left(\sum_{i=s}^h \prod_{j=0}^m (\lambda_i-c_j)\right)^2\right).\]

\item Some but not all of the edges in the cycle in $G''$ show up more than once in the walk defined by $(v'_0,...,v'_{m'})$:

The only way for this to happen is if the edges on one side of the cycle are repeated and the edges on the other side are not. The only way for this to happen is if $(v'_0,...,v'_{m'})$ simply goes across the repeated branch of the cycle when it goes from $v_0$ to $v_m$ and then goes around the cycle one and a fraction times on its way back or vice-versa. So, each edge in the repeated branch shows up exactly $3$ times in the walk defined by $(v'_0,...,v'_{m'})$. Let $e''$ be the number of edges on the repeated side, and $e'''$ be the number of edges on the non-repeated side. $e'''= 2m-2e-2e'-3e''-|I|\le 2(m-e-e'-e'')-|I|$, and there are at most $(2e+2e'+3e''+2r)/2r$ indices of vertices that could be in $I$. For fixed values of $e$, $e'$, $e''$, and $I$, there are at most $n^{2m-e-e'-2e''-|I|}$ possible labelings of the vertices in $G''$, each of which corresponds to at most $2$ possible values of $(v'_0,...,v'_{m'})$. Each of those could be derived from at most $n^{|I|}$ possible values of $(v_0,...,v_m)$ and $(v''_0,...,v''_m)$. Also, the absolute value of the shard's expected value is at most
\begin{align*}
&n^{e+e'+2e''-2m+|I|}(k/(\min p_i))^{8}2^{8(s-1)}\left(\max_{s\le j\le h} |\lambda_j|^{l-s+1}\prod_{i=1}^{s-1} |\lambda_j-\lambda'_i|\right)^{(2m-|I|)/l}\\
&\cdot \left(\max_{s\le j\le h} (\lambda_j^2/\Lambda)^{l-s+1}\prod_{i=1}^{s-1} (\lambda_j-\lambda'_i)^2/\Lambda\right)^{-(e+e'+2e'')/l}\\ & \cdot \left(\max_{s\le j\le h} (|\lambda_j|/\Lambda)^{l-s+1}\prod_{i=1}^{s-1} |\lambda_j-\lambda'_i|/\Lambda\right)^{-\frac{8(l+2r+1)}{l}}.
\end{align*}
So, by the same logic as in the previous case, all such shards have a combined expected contribution of absolute value at most
\begin{align*}
&4\Lambda(k/(\min p_i))^{8}2^{8(s-1)}\left(\max_{s\le j\le h} |\lambda_j|^{l-s+1}\prod_{i=1}^{s-1} |\lambda_j-\lambda'_i|\right)^{(2m-1)/l}\\
&\cdot \left(\max_{s\le j\le h} (|\lambda_j|/\Lambda)^{l-s+1}\prod_{i=1}^{s-1} |\lambda_j-\lambda'_i|/\Lambda\right)^{-\frac{8(l+2r+1)}{l}}\\
&\cdot \left[1-(4\Lambda)^{1/2r} \left(\max_{s\le j\le h} (\lambda_j^2/\Lambda)^{l-s+1}\prod_{i=1}^{s-1} (\lambda_j-\lambda'_i)^2/\Lambda\right)^{-(2r+1)/2rl}\right]^{-3}
\end{align*}
which is within a constant factor of 
\[\left(\max_{s\le j\le h} |\lambda_j|^{l-s+1}\prod_{i=1}^{s-1} |\lambda_j-\lambda'_i|\right)^{2m/l}=O\left(\left(\sum_{i=s}^h \prod_{j=0}^m (\lambda_i-c_j)\right)^2\right).\]

\item Every edge in the cycle in $G''$ appears more than once in the walk defined by  $(v'_0,...,v'_{m'})$:

Let $e''$ and $e'''$ be the numbers of edges on the two branches of the cycle, and $f$ be the number of times the walk defined by $(v'_0,...,v'_{m'})$ makes around the cycle on its way from $v_0$ to $v_m$, and $f'$ be the number of times it goes around the cycle on its way back. $G''$ has at most $(2m-|I|)/2$ edges, so $e'+e''+e'''+e''''\le m-|I|/2$. For fixed values of $e,e',e'',e''',f,f'$, and $I$, there are at most $n^{e+e'+e''+e'''}$ options for $G''$. Each of these options corresponds to at most $4$ possible values of $(v'_0,...,v'_{m'})$, and each of those corresponds to at most $n^{|I|}$ possible values of $(v_0,...,v_m)$ and $(v''_0,...,v''_m)$. Also, $W_{(c'_0,...,c'_{m'})[r]}((v'_0,...,v'_{m'}))$ is one if every edge in $G''$ is also in $G$ and $0$ otherwise. That means that 
\[|\E [W_{(c'_0,...,c'_{m'})[r]}((v'_0,...,v'_{m'}))]|=O((\lambda_1/n)^{e+e'+e''+e'''})\] There are at most $m/r$ nonzero elements of $(c_0,...,c_m,0...0,c_m,...,c_0)$, and thus at most $2^{m/r}$ possible values of $I$. So, all shards for given values of $e,e',e'',e''',f,f'$ and $I$ make a combined contribution with an absolute value of
\[O(\Lambda^{|I|}\lambda_1^{m-|I|/2})\subseteq O(\Lambda^{m/r}\lambda_1^{m-m/2r}).\]
Since $0\le e,e',e'',e''',f,f'\le m$, the combined expected contribution of all these shards is 
\[O(m^62^{m/r}\Lambda^{m/r}\lambda_1^{m-m/2r})\subseteq o\left(\left(\sum_{i=s}^h \prod_{j=0}^m (\lambda_i-c_j)\right)^2\right).\]
\end{enumerate}
\end{proof}

\begin{lemma}\label{varianceLemma}
There exists a constant $r_0$ and $m_0=\Theta(\log n)$ such that if $m$, $\Lambda$, $\lambda'_1,...,\lambda'_{s-1}$, $l$, $r$, and $(c_0,...,c_m)$ satisfy the conditions of the degree bound lemma, $r>r_0$ and $m>m_0$, then for all $v$, we have that
\[\Var[W_{m/\{c_i\}}(x,v)]=O\left(\prod (\lambda_s-c_i)^2\right)\]
\end{lemma}

\begin{proof}We have 
\[W^2_{m/\{c_i\}}(x,v)=\sum_{v_0,...,v_m,v''_0,...,v''_m\in G: \atop{ v_m=v''_m=v}} x_{v_0}x_{v''_0}W_{(c_0,...,c_m)}((v_0,...,v_m))\cdot W_{(c_0,...,c_m)}((v''_0,...,v''_m)).\] 

If $v_0\ne v''_0$ then $E[x_{v_0}\cdot x_{v''_0}]=0$, while if $v_0=v''_0$ then $E[x_{v_0}\cdot x_{v''_0}]=1$. So,
\begin{align*}
E[W^2_{m/\{c_i\}}(x,v)]&=\sum_{v_0,...,v_m,v''_0,...,v''_m\in G: \atop{ v_0=v''_0, v_m=v''_m=v}} E[W_{(c_0,...,c_m)}((v_0,...,v_m))\cdot W_{(c_0,...,c_m)}((v''_0,...,v''_m))]\\
&=O\left(\prod (\lambda_s-c_i)^2\right)
\end{align*}
where the last equality follows by the previous lemma.

\end{proof}

\subsection{Crossing the KS threshold}
For $x \in [k]^n$ and $\e>0$, define the set of bad clusterings with respect to $x$ as
\begin{align}
B_\e(x)&=\{y \in [n]^k : \frac{1}{n} d_*(x,y) > 1-\frac{1}{k} - \e \},
\end{align}
where $d_*(x,y) = \min_{\pi \in S_k} d_H(x,\pi(y))$, $d_H$ is the Hamming distance and $\pi(y)$ denotes the application of $\pi$ to each component of $y$.  

Recall that 
\begin{align}
\mathrm{Bal}(n,k,\e)&=\{x \in [k]^n:  \forall i \in [k], |\{ u \in [n]: x_u = i\}|/n \in [ 1/k -\e, 1/k + \e]  \},
\end{align}
and 
\begin{align*}
T_\delta(G)=&\{   x \in \mathrm{Bal}(n,k,\delta):\\
&  \sum_{i=1}^k |\{  G_{u,v} : (u,v) \in {[n] \choose 2} \text{ s.t. } x_u=i, x_v=i \}|  \geq \frac{an}{2k} (1-\delta),\\
& \sum_{i,j \in [k], i<j} |\{  G_{u,v} : (u,v) \in {[n] \choose 2} \text{ s.t. } x_u=i, x_v=j \}|  \leq \frac{bn(k-1)}{2k} (1+\delta) \}.
\end{align*}

\subsubsection{Atypicality of a bad clustering}

\begin{lemma}\label{bad_bound}
Let $\e>0$, $x,y \in \mathrm{Bal}(n,k,\delta)$ such that $y \in B_\epsilon(x)$, and $(\sigma,G) \sim \mathrm{SBM}(n,k,a/n,b/n)$. Then,
\begin{align*}
&\pp\{ y \in T_\delta(G) | \sigma=x \}  \leq \exp \left(- \frac{n}{k} A(\e, \delta) \right) 
\end{align*}
where $A(\e, \delta)$ is continuous at $(\e,\delta)=(0,0)$ and 
\begin{align*}
&A(0, 0) =  \frac{a+b(k-1)}{2} \ln \frac{k}{(a+(k-1)b)} +  \frac{a}{2} \ln a + \frac{b(k-1)}{2} \ln b .
\end{align*}
\end{lemma}
Recall that $d:=\frac{a+(k-1)b}{k}$.
\begin{corollary}\label{cross}
Detection is solvable in $\mathrm{SBM}(n,k,a/n,b/n)$ if 
\begin{align*}
\frac{1}{2 \ln k} \left( \frac{a \ln a + (k-1) b \ln b}{k}  - d \ln d \right) >1.
\end{align*}
\end{corollary}
Note that for $a=0$, the above bound for $b$ is $\Theta(\ln k/k)$ times the bound given by the KS threshold. Further, this improves on the KS threshold for $k\ge 5$. However, for $b=0$, the above bound gives $a>2k$, which is worse than the KS threshold $a>k$. For the same reasons, it is loose for $k=2$ and $a=0$. In the next section, we improve the bound to capture the right behaviour at the extremal regimes.  

\begin{proof}[Proof of the Corollary \ref{cross}]
Let $(\sigma,G) \sim \sbm(n,k,a/n,b/n)$ where $k, a,b$ satisfy \begin{align}
\frac{1}{2 \ln k} \left( \frac{a \ln a + (k-1) b \ln b}{k}  - d \ln d \right) >1. \label{hypo}
\end{align}
Let $\hat{\sigma}_\delta(G)$ be uniformly drawn in $T_\delta(G)$. We have 

\begin{align}
\pp\{ \hat{\sigma}_\delta(G) \in B_\epsilon(\sigma) \} &= \pp\{ \hat{\sigma}_\delta(G) \in B_\epsilon(\sigma), \sigma\in  \mathrm{Bal}(n,k,\delta) \}+o(1)\label{typ}\\
&=\E \frac{|T_\delta(G) \cap B_\epsilon(\sigma)|\1(\sigma \in \mathrm{Bal}(n,k,\delta) )}{|T_\delta(G)|} +o(1)\\
&\le \E  |T_\delta(G) \cap B_\epsilon(\sigma)| \1(\sigma \in \mathrm{Bal}(n,k,\delta) ) + o(1)
\end{align}

where we use the fact that $|T_\delta(G)| \geq 1$ with high probability, since $\sigma \in T_\delta(G)$ with high probability, and $\sigma \in \mathrm{Bal}(n,k,\delta)$ with high probability. 
Moreover, 
\begin{align}
\pp\{ \hat{\sigma}_\delta(G) \in B_\epsilon(\sigma) \} 
&\leq \E  \sum_{y \in [k]^n } \1(y \in T_\delta(G), y \in B_\epsilon(\sigma), \sigma \in \mathrm{Bal}(n,k,\delta))   + o(1)  \\ 
&\leq   \sum_{y \in [k]^n } \sum_{x \in \mathrm{Bal}(n,k,\delta)} \pp\{y \in T_\delta(G), y \in B_\epsilon(\sigma)| \sigma=x \} \pp\{\sigma=x \}   + o(1) \\
&=    \sum_{y \in [k]^n } \sum_{x \in \mathrm{Bal}(n,k,\delta): y \in B_\epsilon(x)} \pp\{y \in T_\delta(G)| \sigma=x \} \pp\{\sigma=x \}   + o(1).
\end{align}
Letting $\tau$ be the bound obtained from Lemma \ref{bad_bound} on $\pp\{y \in T_\delta(G)| \sigma=x \}$, we have
\begin{align}
\pp\{ \hat{\sigma}_\delta(G) \in B_\epsilon(\sigma) \} \label{drive}
&\le    k^n \tau   + o(1).
\end{align}
We can now take $\delta>0$ such that for a strictly positive $\e$, \eqref{hypo} implies that \eqref{drive} vanishes. 
\end{proof}

\begin{proof}[Proof of Lemma \ref{bad_bound}.] 
Assume that $a\ge b$ (the other case is treated similarly). 
Let $(\sigma,G) \sim \sbm(n,k,a,b)$ and $x,y \in \mathrm{Bal}(n,k,\delta)$ such that $y \in B_\e(x)$. 
Denote by $T^{(in)}_\delta(G)$ and $T^{(out)}_\delta(G)$ the sets of clusterings having typical fraction of edges inside and across clusters respectively. We have 
\begin{align}
\pp\{ y \in T_\delta(G) | \sigma=x \} = \pp\{ y  \in T^{(in)}_\delta(G) | \sigma=x \} \pp\{ y  \in T^{(out)}_\delta(G) |  \sigma=x \}, \label{merge}
\end{align}
with 
\begin{align*}
&\pp\{ y  \in T^{(in)}_\delta(G) | \sigma=x \} \\
&= \pp \left\{  \sum_{i=1}^k |\{  G_{u,v} : (u,v) \in {[n] \choose 2} \text{ s.t. } y_u=i, y_v=i  | \}|  \geq \frac{an}{2k} (1-\delta)|  \sigma=x \right\} \\
&\le \pp \left\{ \mathrm{Bin}\left(\frac{n(n-1)}{2k^2}(1+\xi),\frac{a}{n}\right) + \mathrm{Bin}\left(\frac{(k-1)n(n-1)}{2k^2}(1+\xi),\frac{b}{n}\right) \geq \frac{an}{2k} (1-\delta) \right\}
\end{align*}
\begin{align*}
&\pp\{ y  \in T^{(out)}_\delta(G) |  \sigma=x \}\\ 
 &= \pp \left\{  \sum_{i,j \in [k], i<j} |\{  G_{u,v} : (u,v) \in {[n] \choose 2} \text{ s.t. } y_u=i, y_v=j \}|  \leq \frac{bn(k-1)}{2k} (1+\delta) |  \sigma=x \right\} \\
&\le \pp \left\{ \mathrm{Bin}\left(\frac{n^2(k-1)}{2k^2}(1-\xi),\frac{a}{n}\right) + \mathrm{Bin}\left(\frac{n^2(k-1)^2}{2k^2} (1-\xi),\frac{b}{n}\right) \leq  \frac{bn(k-1)}{2k} (1+\delta) \right\},
\end{align*}
where $\mathrm{Bin}(N,p)$ denotes a Binomial random variable with $N$ trials of bias $p$, all Binomial random variables are independent, and $\xi \le 2k^2 (\e + \delta)$ (the latter gives a loose bound on the offset of the trial counts in the Binomials due to the fact that $x, y$ are not exactly but approximately balanced). Since the dependence in $\e,\delta$ is continuous, we next proceed without the $\xi$ terms to prove the lemma. We have 
\begin{align*}
&\pp \left\{ \mathrm{Bin}\left(\frac{n(n-1)}{2k^2},\frac{a}{n}\right) + \mathrm{Bin}\left(\frac{(k-1)n(n-1)}{2k^2},\frac{b}{n}\right) \geq \frac{an}{2k}  \right\} \\
&\leq O(n^2) \max_{ c n \in [0,\frac{n^2}{2k^2}]} \pp \left\{ \mathrm{Bin}\left(\frac{n^2}{2k^2},\frac{a}{n}\right) = cn \right\} \pp\left\{\mathrm{Bin}\left(\frac{(k-1)n^2}{2k^2},\frac{b}{n}\right) = \frac{an}{2k} - cn  \right\}
\end{align*}
and from \eqref{binap},
\begin{align*}
& \pp \left\{ \mathrm{Bin}\left(\frac{n^2}{2k^2},\frac{a}{n}\right) = cn \right\} =\exp\left( - \frac{n}{k} \left(\frac{a}{2k} + kc \ln \frac{2k^2c}{ea} \right)  +o(n)\right),\\
& \pp\left\{\mathrm{Bin}\left(\frac{(k-1)n^2}{2k^2},\frac{b}{n}\right) = \frac{an}{2k} - cn  \right\}\\
&=\exp\left( - \frac{n}{k} \left( \frac{(k-1)b}{2k} + (a/2-ck) \ln \frac{2k(a/2-ck)}{eb(k-1)} \right) +o(n)\right)
\end{align*}
and 
\begin{align*}
& \pp \left\{ \mathrm{Bin}\left(\frac{n^2}{2k^2},\frac{a}{n}\right) = cn \right\} \pp\left\{\mathrm{Bin}\left(\frac{(k-1)n^2}{2k^2},\frac{b}{n}\right) = \frac{an}{2k} - cn  \right\} \\&=\exp\left( - \frac{n}{k} \left(\frac{a+(k-1)b}{2k}+ \frac a2 \ln \frac{2k(a/2-ck)}{eb(k-1)}  + kc \ln \frac{bc(k-1)k}{a(a/2-ck)} \right)+o(n)  \right)\\
&=\exp\left( - \frac{n}{k} \left(\frac{a+(k-1)b}{2k}+ \frac a2 \ln \frac{2k}{eb(k-1)} +  \left(\frac a2-ck \right)\ln \left(\frac a2-ck\right) + kc \ln \frac{bc(k-1)k}{a} \right)+o(n)  \right).
\end{align*}
Note that the function 
\begin{align*}
c \to (a/2-ck)\ln (a/2-ck) + kc \ln \frac{bc(k-1)k}{a}
\end{align*}
has a single minimum at $\frac{a^2}{2k(a+(k-1)b)}$. Plugging in this value for $c$ gives 
\begin{align}
& \pp \left\{ \mathrm{Bin}\left(\frac{n^2}{2k^2},\frac{a}{n}\right) = cn \right\} \pp\left\{\mathrm{Bin}\left(\frac{(k-1)n^2}{2k^2},\frac{b}{n}\right) = \frac{an}{2k} - cn  \right\} \\
&\leq  
\exp\left( - \frac{n}{k} \left(  \frac{a+(k-1)b}{2k} + \frac{a}{2} \ln \frac{ak}{e(a+(k-1)b)}\right) \right). \label{part1}
\end{align}

In addition, 
\begin{align*}
&\pp \left\{ \mathrm{Bin}\left(\frac{n^2(k-1)}{2k^2} ,\frac{a}{n}\right) + \mathrm{Bin}\left(\frac{n^2(k-1)^2}{2k^2},\frac{b}{n}\right) \leq  \frac{bn(k-1)}{2k} \right\},\\
&\leq O(n^2) \\
& \cdot \max_{ c n \in [0,\frac{n^2 (k-1)}{2k^2}]} \pp \left\{ \mathrm{Bin}\left(\frac{n^2(k-1)}{2k^2},\frac{a}{n}\right) = cn \right\} \pp\left\{\mathrm{Bin}\left(\frac{n^2(k-1)^2}{2k^2},\frac{b}{n}\right) = \frac{bn(k-1)}{2k} - cn  \right\}
\end{align*}
and from \eqref{binap},
\begin{align*}
& \pp \left\{ \mathrm{Bin}\left(\frac{n^2(k-1)}{2k^2},\frac{a}{n}\right) = cn \right\} =\exp\left( - \frac{n}{k} \left(\frac{a(k-1)}{2k} + kc \ln \frac{2k^2c}{ea(k-1)} \right)+o(n)  \right),\\
& \pp\left\{\mathrm{Bin}\left(\frac{n^2(k-1)^2}{2k^2},\frac{b}{n}\right) = \frac{bn(k-1)}{2k} - cn  \right\}\\
&=\exp\left( - \frac{n}{k} \left( \frac{(k-1)^2b}{2k} + (b(k-1)/2-ck) \ln \frac{(b(k-1)-2ck)k}{eb(k-1)^2} \right)+o(n) \right)
\end{align*}
and 
\begin{align*}
& \pp \left\{ \mathrm{Bin}\left(\frac{n^2(k-1)}{2k^2},\frac{a}{n}\right) = cn \right\} \pp\left\{\mathrm{Bin}\left(\frac{n^2(k-1)^2}{2k^2},\frac{b}{n}\right) = \frac{bn(k-1)}{2k} - cn  \right\}\\
&=\exp ( - \frac{n}{k} [ (\frac{a(k-1)}{2k} + kc \ln \frac{2k^2c}{ea(k-1)}  +  \frac{(k-1)^2b}{2k}\\
&\qquad  + (b(k-1)/2-ck) \ln \frac{(b(k-1)-2ck)k}{eb(k-1)^2} ) ] +o(n))\\
&=\exp ( - \frac{n}{k} [ (\frac{(k-1)(a+(k-1)b)}{2k} + \frac{b(k-1)}{2}\ln \frac{(b(k-1)-2ck)k}{eb(k-1)^2} \\
& \qquad + kc \ln \frac{2cb(k-1)k}{a(b(k-1)-2ck)}  ) ] +o(n))\\
&=\exp ( - \frac{n}{k} [ (\frac{(k-1)(a+(k-1)b)}{2k}+ \frac{b(k-1)}{2}\ln \frac{2k}{eb(k-1)^2}  \\
&\quad + (b(k-1)/2 - ck)  \ln (b(k-1)/2-ck) + kc \ln \frac{cb(k-1)k}{a}  ) ]+o(n) ).
\end{align*}
Note that the function 
\begin{align*}
c \to (b(k-1)/2 - ck)  \ln (b(k-1)/2-ck) + kc \ln \frac{cb(k-1)k}{a} 
\end{align*}
has a single minimum at $\frac{(k-1) ab}{2k(a+(k-1)b)}$, thus 
\begin{align}
&\pp \left\{ \mathrm{Bin}\left(\frac{n^2}{2k^2},\frac{a}{n}\right) = cn \right\} \pp\left\{\mathrm{Bin}\left(\frac{(k-1)n^2}{2k^2},\frac{b}{n}\right) = \frac{an}{2k} - cn  \right\} \\
&\leq 
\exp\left( - \frac{n}{k} \left( \frac{(k-1)(a+(k-1)b)}{2k} + \frac{b(k-1)}{2} \ln \frac{bk}{e(a+(k-1)b)} \right)+o(n) \right).  \label{part2}
\end{align}
Combining \eqref{part1} and \eqref{part2}, 
the result follows from algebraic manipulations. 
\end{proof}

\begin{lemma}[Binomial estimates]\label{tech}
Let $N$ be a positive integer and $t,s >0$. Then, 
\begin{align}
\pp\{ \mathrm{Bin}(N^2,s/N) = \lfloor t N \rfloor \} &= \exp\left( - \left(s  -t + t \ln \frac{t}{s} \right) N  + O(\ln N) \right), \label{binap} \\
\pp\{ \mathrm{Bin}(N^2,s/N) \leq  t N  \} &= \exp\left( - \left(s  -t + t \ln \frac{t}{s} \right) N  + O(\ln N) \right), \quad \text{if } s>t.
\end{align}
\end{lemma}
\begin{proof}
Stirling's approximation gives
\begin{align}
N! = e^{N \ln N -N +O(\ln N)} 
\end{align}
thus
\begin{align}
{N^2 \choose \lceil t N \rceil } = \exp\left( t N \ln N + tN  \ln \frac{e}{t}  + O(\ln N) \right)
\end{align}
and
\begin{align}
&\pp\{ \mathrm{Bin}(N^2,s/N) =  \lceil t N \rceil \} \\
&= {N^2 \choose \lceil t N \rceil } \left( \frac{s}{N} \right)^{ \lceil t N \rceil } \left(1- \frac{s}{N} \right)^{N^2- \lceil t N \rceil } \\
&= \exp\left( t N \ln N + tN  \ln \frac{e}{t}  + O(\ln N) \right) \exp \left( -t N \ln N + t N \ln s - sN  + O(\ln N) \right)\\
&= \exp\left(  \left(t  \ln \frac{e s}{t} - s \right) N  + O(\ln N) \right) ,
\end{align}
and the second statement follows from the fact that the Binomial distribution is unimodal. 
\end{proof}

\subsubsection{Size of the typical set}\label{size}
The goal of this section is to lower bound the size of the typical set $T_\delta(G)$.
Recall that the expected node degree in $\sbm(n,k,a,b)$ is
\begin{align*}
d:=\frac{a+(k-1)b}{k}.
\end{align*}

\begin{theorem}\label{size_bound}
Let $T_\delta(G)$ denote the typical set of clusterings for $G$ drawn under $\sbm(n,k,a,b)$. Then, for any $\e>0$, 
\begin{align*}
\pp\{ |T_\delta(G)| < \max(k^{(\psi - \e) n},2^{(e^{-a/k}(1-(1-e^{-b/k})^{k-1})-\epsilon)n})\}=o(1),
\end{align*}
where 
\begin{align*}
\psi &:=\frac{\tau}{d}\left(1-\frac{\tau}{2}\right) \\
&+ \frac{1}{\ln(k)} \left( \frac{a}{a+(k-1)b} \ln\left(\frac{a+(k-1)b}{a}\right) +\frac{ (k-1)b}{a+(k-1)b} \ln\left(\frac{a+(k-1)b}{b}\right) \right) \\
& \phantom{- \frac{1}{\ln(k)}} \cdot \left( \frac{\tau^2}{2d} + (d -\tau)  e^{-(d-\tau)}\right),
\end{align*}
and $\tau=\tau_d$ is the unique solution in $(0,1)$ of 
\begin{align}
\tau e^{-\tau}=de^{-d} \label{tau_def}
\end{align}
or equivalently $\tau = \sum_{j=1}^{+ \infty} \frac{j^{j-1}}{j!} (de^{-d})^j$.
\end{theorem}

As a first step to proving this theorem, we prove one half of the inequality.
\begin{lemma}\label{Tbound2}
Let $T_\delta(G)$ denote the typical set of clusterings for $G$ drawn under $\sbm(n,k,a,b)$. Then, for any $\e>0$, 
\[\pp\{ |T_\delta(G)| < 2^{(e^{-a/k}(1-(1-e^{-b/k})^{k-1})-\epsilon)n}\}=o(1).\]
\end{lemma}

\begin{proof}
Construct a clustering of $G$, $\sigma'$, as follows. Start with $\sigma'=\sigma$. Then, go through the vertices of $G$ in a random order, and for each $v$ with no neighbor $v'$ such that $\sigma'_v=\sigma'_{v'}$, change $\sigma'_{v}$ to a random element  of $[k]\backslash \{\sigma'_{v'}: (v,v')\in E(G)\}$.

Initially, $\sigma'=\sigma$, so the probability distributions of $(G,\sigma)$ and $(G,\sigma')$ are identical. For a given ordering of the vertices, if these probability distributions are identical immediately before $\sigma'_{v}$ is changed, then for given values of $G$ and $\sigma'\backslash \sigma'_v$, it is equally likely that $\sigma'_v$ had each value that it may be assigned in this step. So, the probability distribution is unchanged. Therefore, the probability distribution of $(G,\sigma')$ is identical to the probability distribution of $(G,\sigma)$ at any point in this algorithm. In particular, $\sigma'\in T_\delta(G)$ with probability $1-o(1)$.

Now, for each $v\in G$, let $z_v$ be $1$ if there is more than one option for $\sigma'_v$ when it is time for the algorithm to assign it and $0$ otherwise. A vertex has no neighbors in its own community with probability $e^{-a/k}+o(1)$, and no neighbors in any given other community with probability $e^{-b/k}+o(1)$. So, a vertex has no neighbors in its own community and at least one other community with probability $e^{-a/k}(1-(1-e^{-b/k})^{k-1})+o(1)$. Thus, for each $v$, $$\E(z_v)=e^{-a/k}(1-(1-e^{-b/k})^{k-1})+o(1).$$ Also, given two different vertices, they both have no neighbors in their own community and at least one other with probability $e^{-2a/k}(1-(1-e^{-b/k})^{k-1})^2+o(1)$. Given any $v$ and $v'$, assume without loss of generality that $v$ is before $v'$ in the random ordering. With probability $1-o(1)$, $v'$ will not be in the last $\sqrt{n}$ vertices, and at the time $v'$ comes up in the ordering, $e^{-a/k}(1-(1-e^{-b/k})^{k-1})+o(1)$ of the remaining vertices will have no neighbors claimed to be in their communities and at least one other. So, symmetry between the vertices implies that the correlation between $z_v$ and $z_{v'}$ is $o(1)$. That means that with probability $1-o(1)$, we have 
\[\sum_{v\in G} z_v=e^{-a/k}(1-(1-e^{-b/k})^{k-1})n+o(n)\]

For fixed values of $(G,\sigma)$ and any given $\sigma'_0$, $z_0$, we have that $\pp\{ \sigma'=\sigma'_0,\sum_{v\in G} z_v\ge z_0\} \le 2^{-z_0}$ because every time the algorithm has more than one option for $\sigma'_v$, there is at most a $1/2$ chance that it sets $\sigma'_v$ to $(\sigma'_0)_v$, and if $\sum z_v\ge z_0$ then there are at least $z_0$ times when it had such a choice. So, for a fixed graph $G$,
\begin{align}
& \pp \left\{\sigma'\in T_\delta(G), \sum_{v\in G} z_v>e^{-a/k}(1-(1-e^{-b/k})^{k-1})n-\epsilon n/2\right\} \\
& \le |T_\delta(G)|\cdot 2^{e^{-a/k}(1-(1-e^{-b/k})^{k-1})n-\epsilon n/2}.
\end{align}
We already know that with probability $1-o(1)$ the probability in question is $1-o(1)$, so the conclusion holds.
\end{proof}

To prove the rest of this theorem, we need some topological properties of the SBM graph, analog to the Erd\H{o}s-R\'enyi case \cite{ER-seminal2}.

\begin{lemma}\label{topo1}
Let $T_j$ denote the number of isolated $j$-trees (i.e., trees on $j$ vertices) in $\sbm(n,k,a,b)$ and 
$M_j$ the number of edges contained in those trees. 
Then, for any $\e>0$,  
\begin{align*}
&\pp\{ T_j/n  \notin [\tau_j/d - \e, \tau_j/d + \e] \}=o(1),\\
&\pp\{ M_j/n  \notin [(j-1)\tau_j/d - \e, (j-1)\tau_j/d + \e] \}=o(1),
\end{align*}
where 
\begin{align*}
\tau_j=\frac{j^{j-2}(de^{-d})^j }{ j!}.
\end{align*}
\end{lemma}

\begin{lemma}\label{topo2}
Let $T$ denote the number of isolated trees in $\sbm(n,k,a,b)$ and 
$M$ the number of edges contained in those trees. 
Then, for any $\e>0$,  
\begin{align*}
&\pp\left\{ T/n  \notin \left[\frac{\tau}{d}\left(1-\frac{\tau}{2}\right) - \e, \frac{\tau}{d}\left(1-\frac{\tau}{2}\right) + \e \right] \right\}=o(1),\\
&\pp\left\{ M/n  \notin \left[\frac{\tau^2}{2d} - \e, \frac{\tau^2}{2d} + \e \right] \right\}=o(1),
\end{align*}
where $\tau$ is defined in \eqref{tau_def}.
\end{lemma}

\begin{lemma}\label{giant}\cite{bollo_inhomo}
Let $Q$ denote the number of nodes that are in the giant component in $\sbm(n,k,a,b)$, i.e., 
nodes that are in a linear size component of $\sbm(n,k,a,b)$.
Then, for any $\e>0$,  
\begin{align*}
&\pp\{ Q/n  \notin [ \beta - \e, \beta + \e] \}=o(1),
\end{align*}
where $\beta = 1-\tau/d$ and $\tau$ is defined in \eqref{tau_def}. 
\end{lemma}

\begin{lemma}\label{topo3}
Let $F$ denote the number of edges that are in planted trees of the giant in $\sbm(n,k,a,b)$, i.e., 
trees formed by nodes that have a single edge connecting them to a linear size component in $\sbm(n,k,a,b)$.
Then, for any $\e>0$,  
\begin{align*}
&\pp\{ F/n  \notin [ \phi - \e, \phi + \e] \}=o(1),\\
&\text{where }\phi:= (d - \tau)  e^{-(d - \tau)},
\end{align*}
and $\tau$ is defined in \eqref{tau_def}.
\end{lemma}

\begin{proof}[Proof of Theorem \ref{size_bound}]
Let $G \sim \sbm(n,k,a,b)$, and let $T$ be the number of isolated trees in $G$, $M$ the number of edges in those trees, and $F$ the number of edges in the planted trees of the largest connected component of $G$ (i.e., the giant). 
We have from Lemma \ref{topo1}, Lemma \ref{topo2} and Lemma \ref{topo3} that, for $\e>0$, with high probability on $G$,
\begin{align*}
T &\in \left[\frac{\tau}{d}\left(1-\frac{\tau}{2}\right) - \e, \frac{\tau}{d}\left(1-\frac{\tau}{2}\right) + \e \right], \\
R &\in \left[\frac{\tau^2}{2d} + (d -\tau)  e^{-(d-\tau)} - \e, \frac{\tau^2}{2d} + (d -\tau)  e^{-(d-\tau)} + \e \right]. 
\end{align*}
Assume that $T$ and $R$ take typical values as above. We now build a typical vertex labelling on these trees:
\begin{itemize}
\item Pick an arbitrary node in each isolated tree, denote these by $\{v_1,\dots,v_T\}$, and denote the set of edges contained in these trees by $\{E_1,\dots, E_M\}$; 
\item Pick the root node for each planted tree that is in the giant, denote these by $\{w_1,\dots,w_k\}$, where $k$ is the number of planted trees (this number is not relevant in the following computations) and denote by $\{E_{M+1},\dots, E_{M+F}\}$ the number of edges contained in those planted trees;
\item Assign the labels $U_1^T:=(U_{v_1},\dots, U_{v_T})$ uniformly at random in $[k]$, and set $\hat{X}_{w_1}=\sigma_{w_1},\dots, \hat{X}_{w_k}=\sigma_{w_k}$, i.e., assign the latter labels to their true community assignments. Then broadcast each of these labels in their corresponding trees by forwarding the labels on each edge with an independent $k$-ary symmetric channel of flip probability $\frac{b}{a+(k-1)b}$. This means that the variables $Z_1,\dots, Z_{M+F}$ are drawn i.i.d.\ from the distribution 
\begin{align}
\nu:=\left(\frac{a}{a+(k-1)b}, \frac{b}{a+(k-1)b}, \dots, \frac{b}{a+(k-1)b} \right), 
\end{align}
on $\F_k:=\{0,1,\dots,k-1\}$, and that for each edge $e$ in the trees, the input bit is forwarded by adding to it the $Z_e$ variable modulo $k$;
\item Assign any other label (that is not contained in the trees) to their true community assignments. Define $R:=M+F$, $Z_1^R:=(Z_1,\dots,Z_R)$, and denote by $\hat{X}(U_1^T,Z_1^R)$ the previously defined assignment. 
\end{itemize}
Note that the above gives the induced label-distribution on trees in $\sbm(n,k,a,b)$. Thus, with high probability on $G$, as $T$ and $M$ grow with $n$, this assignment is typical with high probability on $U_1^T,Z_1^{R}$, i.e., 
\begin{align}
\pp_{U_1^T,Z_1^{R}}\{ \hat{X}(U_1^T,Z_1^R) \in T_\delta(G)\} = 1-o(1). \label{b1}
\end{align}  

Denote by $\mathrm{Emp}(Z_1^R)$ the empirical distribution on $\F_k$ of the $R$-vector $Z_1^R$, and by $\mathcal{B}_\e(\nu)$ the $l_1$-ball around $\mu$ of radius $\e$.  Denote by $\eta$ the uniform distribution on $[k]$. Then by Sanov's Theorem, for any $\e>0$, 
\begin{align}
\pp_{Z_1^R}\{ \mathrm{Emp}(U_1^T) \in \mathcal{B}_\e(\eta), \mathrm{Emp}(Z_1^R) \in \mathcal{B}_\e(\nu) \} \to 1, \label{lln}
\end{align}  
as $T,R$ diverge with $n$. 
Define now the set of realizations of $Z_1^R$ that have a typical likelihood by  
\begin{align*}
A_{\e}(\nu):=\{z_1^R \in \F_k^R :  k^{ -R ( H(\nu) +\e)} \leq \pp\{Z_1^R=z_1^r\} \leq k^{ -R ( H(\nu) -\e)} \}
\end{align*}
where $H$ is the entropy with the logarithm in base $k$.
For convenience of notation, define also $A_{\e}(\eta)=\mathcal{B}_\e(\eta)$. Again, for all $\e>0$
\begin{align*}
\pp\{U_1^T \in A_{\e}(\eta), Z_1^R \in A_{\e}(\nu) \} \to 1.
\end{align*}
Therefore, 
\begin{align}
&\pp_{U_1^T,Z_1^{R}}\{ \hat{X}(U_1^T,Z_1^R) \in T_\delta(G)\} \\
&=\pp_{U_1^T,Z_1^{R}}\{ \hat{X}(U_1^T,Z_1^R) \in T_\delta(G) , U_1^T \in  A_{\e}(\eta), Z_1^R \in  A_{\e}(\nu)  \}  + o(1)\\
&= \sum_{u_1^T \in A_{\e}(\eta), z_1^{R} \in   A_{\e}(\nu) } \1( \hat{X}(u_1^T,z_1^R)  \in T_\delta(G)) \pp\{ U_1^T= u_1^T\} \pp\{ Z_1^R= z_1^R\} + o(1) \\
&\leq \sum_{u_1^T \in A_{\e}(\eta), z_1^{R} \in A_{\e}(\nu) } \1( \hat{X}(u_1^T,z_1^R)  \in T_\delta(G)) k^{-T}   k^{ -R ( H(\nu) -\e)} +o(1).\label{b2}
\end{align}  
We have 
\begin{align*}
\sum_{u_1^T \in A_{\e}(\eta), z_1^{R} \in A_{\e}(\nu) } \1( \hat{X}(u_1^T,z_1^R)  \in T_\delta(G)) \leq  |T_\delta(G)|, 
\end{align*}  
since the left hand side counts a subset of the typical clusterings.
We thus obtain from \eqref{b1} and \eqref{b2} that with high probability on $G$, 
\begin{align*}
|T_\delta(G)| \geq (1-o(1)) k^{T +R(H(\nu) -\e)}.
\end{align*}
Thus, with high probability on $G$,  
\begin{align*}
|T_\delta(G)| \geq (1-o(1)) k^{n (\psi -\e)  +o(n)},
\end{align*}
where
\begin{align*}
\psi=\frac{\tau}{d}\left(1-\frac{\tau}{2}\right) + H(\nu) \left( \frac{\tau^2}{2d} + (d -\tau)  e^{-(d-\tau)}\right),
\end{align*}
which proves the half of the claim not covered by lemma \ref{Tbound2}. 
\end{proof}

\begin{proof}[Proof of Lemma \ref{topo1}]
Let $T_j$ denote the number of isolated $j$-trees (i.e., trees on $j$ vertices) in $\sbm(n,k,a,b)$, i.e., 
\begin{align*}
T_j = \sum_{S \in \mathcal{T}_j(n)} \1_{\mathrm{iso}}(S)
\end{align*}
where $\mathcal{T}_j(n)$ denotes the set of all $j$-tress on the vertex set $[n]$, and where 
$\1_{\mathrm{iso}}(S)$ is equal to 1 if $S$ is an isolated tree, and 0 otherwise. 
Since the number of different trees on $j$ vertices is given by $j^{j-2}$ \cite{cayley}, and since there is an edge on a designated pair of vertices with probability $d/n$, 
we have 
\begin{align*}
\E [T_j] &= {n \choose j} j^{j-2} \left(\frac{d}{n} \right)^{j-1} \left(1-\frac{d}{n} \right)^{j(n-j)}  \\
&  =  \frac{n}{d}\left( d e^{-d} \right)^j  \frac{ j^{j-2}}{j!} + O(1) 
\end{align*}
and
\begin{align*}
\Var T_j &= \frac{1}{2} \frac{n}{d} \left( d e^{-d} \right)^{2j} \left(\frac{  j^{j-2}}{j!} \right)^2 + O(1) = O(n).
\end{align*}
Thus, by Chebyshev's inequality, for any $\e>0$,  
\begin{align*}
&\pp\{ T_j/n  \notin [\tau_j/d- \e, \tau_j/d + \e] \}=O(1/n),
\end{align*}
where 
\begin{align*}
\tau_j=\frac{j^{j-2}(de^{-d})^j }{j!}.
\end{align*}
and since each tree contains $j-1$ edges, 
\begin{align*}
&\pp\{ M_j/n  \notin [(j-1)\tau_j/d - \e, (j-1)\tau_j/d + \e] \}=O(1/n).
\end{align*}
\end{proof}

\begin{proof}[Proof of Lemma \ref{topo2}]
Let $T$ denote the number of isolated tress in $\sbm(n,k,a,b)$, i.e., 
\begin{align*}
T=\sum_{j=1}^n T_j,
\end{align*}
where $T_j$ is the number of isolated $j$-trees. Hence,
\begin{align*}
\E[ T]= \sum_{j=1}^n \E [T_j] = \frac{n}{d}  \sum_{j=1}^n \frac{j^{j-2}(de^{-d})^j }{ j!} + O(1).
\end{align*}
Since 
\begin{align*}
\sum_{j=n}^\infty \frac{j^{j-2}(de^{-d})^j }{ j!} = O(n^{-3/2}),
\end{align*}
we have 
\begin{align*}
\lim_{n \to \infty} \frac{1}{n} \E[ T] = \frac{1}{d}  \sum_{j=1}^\infty \frac{j^{j-2}(de^{-d})^j }{ j!}.
\end{align*}
Finally, since $\Var (T) = O(n)$, the result for $T$ follows from Chebyshev's inequality and the fact that (see \cite{renyi_tree})
\begin{align}
\sum_{j=1}^\infty \frac{j^{j-2}(de^{-d})^j }{ j!} = \tau - \tau^2/2 \label{idr}
\end{align}
when
\begin{align*}
\tau = \sum_{j=1}^\infty \frac{j^{j-1}(de^{-d})^j }{ j!}.  
\end{align*}
For $M$, the result follows from similar arguments and the fact that 
\begin{align*}
\E [M]= \sum_{j=1}^n (j-1) \E [T_j] &= \frac{n}{d}  \sum_{j=1}^n (j-1) \frac{j^{j-2}(de^{-d})^j }{ j!} + O(1) = \frac{n}{d} \frac{\tau^2}{2} + O(1),
\end{align*}
which uses again \eqref{idr}.
\end{proof}

\begin{proof}[Proof of Lemma \ref{topo3}] 
From Lemma \ref{giant}, the giant component has with high probability a relative size in $[\beta -\e, \beta + \e]$.
The probability that a node is connected to a giant component by a single edge is thus given by 
\begin{align}
\beta n (d/n)(1-d/n)^{\beta n} +o(1)=\beta d e^{- \beta d} +o(1), 
\end{align}
and the expected number of such nodes is 
\begin{align}
n \beta d e^{- \beta d}  +o(n).
\end{align}
Now, let $v$ and $v'$ be random vertices. Conditioning on $v$ being connected to the giant component by a single edge, the giant component still has relative size in $[\beta -\e, \beta + \e]$ for any $\e$ with probability $1-o(1)$, so the probability that $v'$ is connected to the giant component by a single edge is also $\beta d e^{- \beta d} +o(1)$. That means that the expected value of the square of the number of nodes that are connected to the giant component by a single edge is
\[(n \beta d e^{- \beta d})^2  +o(n^2)\]
hence the variance in the number of such nodes is $o(n^2)$ and the lemma follows. 
\end{proof}

\subsubsection{Sampling probability estimates}
\begin{proof}[Proof of Theorem \ref{main2}]
Let $t=\max( k^{(\psi -\e) n},2^{(e^{-a/k}(1-(1-e^{-b/k})^{k-1})-\epsilon)n})$. 
We have
\begin{align}
& \pp\{ \hat{\sigma}(G) \in B_\e(\sigma) \} \\
&= \E \frac{|T_\delta(G) \cap B_\e(\sigma)|}{|T_\delta(G)|}\\
&\leq \E \frac{|T_\delta(G) \cap B_\e(\sigma)|}{|T_\delta(G)|} \1(|T_\delta(G)| \geq t)\1( \sigma \in \mathrm{Bal}(n,k,\delta)) \\
&+ \E \1 (|T_\delta(G)| < t)+\E\1( \sigma \not\in \mathrm{Bal}(n,k,\delta)) \\
&\leq (1/t) \E |T_\delta(G) \cap B_\e(\sigma)| \1(|T_\delta(G)| \geq t)\1( \sigma \in \mathrm{Bal}(n,k,\delta)) \\
&+ \pp\{ |T_\delta(G)| < t \}+\pp( \sigma \not\in \mathrm{Bal}(n,k,\delta))   \\
&\leq (1/t)\cdot \E |T_\delta(G) \cap B_\e(\sigma)| \1( \sigma \in \mathrm{Bal}(n,k,\delta)) +  o(1)\\
&\leq  (1/t)k^n e^{-A(\e,\delta) n/k} +  o(1).  \label{expo} 
\end{align}
Using the bound $t\ge 2^{(e^{-a/k}(1-(1-e^{-b/k})^{k-1})-\epsilon)n}$ and the bound on $A(\e,\delta)$ from Lemma \ref{bad_bound}, the exponent in \eqref{expo} can be made to vanish if 
\begin{align}
&\frac{1}{k} \left(  \frac{a+b(k-1)}{2} \ln \frac{k}{(a+(k-1)b)} +  \frac{a}{2} \ln a + \frac{b(k-1)}{2} \ln b \right)+e^{-a/k}(1-(1-e^{-b/k})^{k-1})\ln(2) \notag \\
&\qquad >\ln(k)  \label{ignore}.\end{align}

Alternately, plugging in the values of $\psi$ from Theorem \ref{size_bound} and $A(\e,\delta)$ from Lemma \ref{bad_bound}, the exponent in \eqref{expo} is vanishing if 
\begin{align}
&\frac{1}{k \ln (k)} \left(  \frac{a+b(k-1)}{2} \ln \frac{k}{(a+(k-1)b)} +  \frac{a}{2} \ln a + \frac{b(k-1)}{2} \ln b \right)\\
&+ \frac{\tau}{d}\left(1-\frac{\tau}{2}\right) \\
&+ \frac{1}{\ln(k)} \left( \frac{a}{a+(k-1)b} \ln\left(\frac{a+(k-1)b}{a}\right) +\frac{ (k-1)b}{a+(k-1)b} \ln\left(\frac{a+(k-1)b}{b}\right) \right) \label{drop} \\
& \phantom{- \frac{1}{\ln(k)}} \cdot \left( \frac{\tau^2}{2d} + (d -\tau)  e^{-(d-\tau)}\right)\\
&>1 . \label{tot}
\end{align}
Since $\tau$ is the solution in $(0,1)$ of $\tau e^{-\tau}=de^{-d}$, we have 
\begin{align*}
  \frac{\tau^2}{2d} + (d -\tau)  e^{-(d-\tau)} = \tau \left(1-\frac{\tau}{2d}\right),
\end{align*}
and this simplifies to 
\begin{align}
&\frac{1}{2 \ln (k)} \left( - d \ln d +  \frac{a \ln a + b(k-1) \ln b}{k} \right)\\
&+ \frac{1}{2 \ln(k)} \left( d \ln d + d \ln k - \frac{a \ln a + b(k-1) \ln b}{k }  \right) \cdot  \frac{2\tau}{d} \left(1-\frac{\tau}{2d}\right) \\
&>1 - \frac{\tau}{d}\left(1-\frac{\tau}{2}\right). \label{tot2}
\end{align}
Algebraic manipulations lead to the bound in the theorem. 
\end{proof}

\begin{proof}[Proof of Corollary \ref{corol_ext}]
Note that dropping the term in \eqref{drop} above and ignoring the contribution of \eqref{ignore} leads to the weaker bound 
\begin{align*}
\frac{1}{2 \ln k} \left( \frac{a \ln a + (k-1) b \ln b}{k}  - d \ln d \right) >1 - \frac{\tau}{d}\left(1-\frac{\tau}{2}\right),
\end{align*}
which implies Corollary \ref{corol_ext} part 2 since $\frac{\tau}{d}\left(1-\frac{\tau}{2}\right)$ tends to $1/2$ as $\tau$ tends to 1 when $b$ tends to 0 and $d$ tends to $a/k$. Part 1 follows from algebraic manipulations. 
\end{proof}


\subsection{Learning the model}
The proof is analog to the case $k=2$ from \cite{Mossel_SBM1}.
\begin{lemma}
Let $G$ be drawn from $SBM(n,p,W)$, $m>0$, and $v_0,v_1,...,v_m$ be vertices such that $v_i\ne v_j$ whenever $|i-j|<\max(m,2)$. For any fixed values of $\sigma_{v_0}$ and $\sigma_{v_m}$, the probability that there is an edge between $v_i$ and $v_{i+1}$ for all $1\le i<m$ is $e_{\sigma_{v_0}} \cdot P^{-1} (PW)^m e_{\sigma_{v_m}}$.
\end{lemma}

\begin{proof}
We proceed by induction on $m$. If $m=1$, then the probability that there is an edge between $v_0$ and $v_1$ is $W_{\sigma_{v_0},\sigma_{v_1}}=e_{\sigma_{v_0}}\cdot P^{-1} PWe_{\sigma_{v_m}}$, as desired. Now, assume that the lemma holds for $m=m_0$ and consider the case where $m=m_0+1$. The probability that there is an edge between $v_i$ and $v_{i+1}$ for all $1\le i<m$ is
\begin{align*}&\sum_i (e_{\sigma_{v_0}} \cdot P^{-1} (PW)^{m-1} e_i) p_i W_{i,\sigma_{v_m}}\\
&=\sum_i (e_{\sigma_{v_0}} \cdot P^{-1} (PW)^{m-1} e_i) (p_i e_i \cdot W e_{\sigma_{v_m}})=e_{\sigma_{v_0}} \cdot P^{-1} (PW)^m e_{\sigma_{v_m}}
\end{align*} 
as desired.
\end{proof}

\begin{corollary}
The expected number of cycles of length $m$ in $\sbm(n,p,Q/n)$ is asymptotic to $\frac{1}{2m}\sum_{i=1}^{k} \lambda_i^m$, where $\{\lambda_i\}$ are the eigenvalues of $PQ$, with multiplicity.
\end{corollary}

\begin{proof}
There are $n(n-1)...(n-m+1)=\Theta(n^m)$ possible serieses of $m$ distinct vertices in the graph, and each cycle has $2m$ possible choices of a starting vertex and a direction. The starting vertex is in community $i$ with probability $p_i$. So, the expected number of length $m$ cycles is asymptotic to
\begin{align*}
\frac{n^m}{2m}\sum_i p_i e_i \cdot P^{-1} (PQ/n)^m e_i&=\frac{1}{2m}\sum_i e_i \cdot (PQ)^m e_i\\
&=\frac{1}{2m}\mathrm{Tr}((PQ)^m)\\
&=\frac{1}{2m} \sum_{i=1}^k \lambda_i^m.
\end{align*} 
\end{proof}

\begin{proof}[Proof of Lemma \ref{learn}]
The variance in the number of cycles is asymptotic to the mean, so the expected difference between the actual number of cycles of a given size and the expected number is proportional to the square root of the expected number. In the symmetric SBM where two vertices in the same community are connected with probability $a/n$ and two vertices in different communities are connected with probability $b/n$, this means that with high probability, the number of cycles of length $m$ is
\[ \frac{1}{2m}\left(\frac{a+(k-1)b}{k}\right)^m+\frac{k-1}{2m}\left(\frac{a-b}{k}\right)^m  + O\left(\left(\frac{a+(k-1)b}{k}\right)^{m/2}/\sqrt{m}\ln(n)\right).\]
Now, assume that $\left(\frac{a-b}{k}\right)^2>\frac{a+(k-1)b}{k}$. Clearly $\frac{a+(k-1)b}{k}$ can be computed up to an error of $O(1/\sqrt{n})$ by counting the edges in the graph, so given the number of cycles of length $m$ for all $m\le M=\omega(1)$, one can determine $(a-b)/k$ and $k$ with error asymptotic to $0$, which provides enough information to determine $a$ and $b$ with error asymptotic to $0$. 

In order to obtain an efficient estimator, one can count non-backtracking walks instead of cycles. 
For $m=o(\log_d^{1/4}(n))$, the neighborhood at depth $m$ of a vertex in $\sbm(n,k,a,b)$ contains at most one cycle with high probability. Thus the number of cycles of length $m$ in $\sbm(n,k,a,b)$ is with high probability equal to $\sum_{v \in [n]} C_v$ where $C_v$ is the number of non-backtracking walks of length $m$ starting and ending at $v$. This is efficiently computable as shown in \cite{Mossel_SBM1}.
\end{proof}


\bibliographystyle{amsalpha}
\bibliography{gen_sbm}

\end{document}